\documentclass[
a4paper,reqno]{amsart}
\usepackage[T1]{fontenc}
\usepackage[english]{babel}
\usepackage{amssymb,bbm,enumerate}
\usepackage{color}

\usepackage{xypic}

\usepackage{mathtools}
\usepackage{mathrsfs}
\usepackage{dsfont}

\makeatletter
\DeclareRobustCommand\widecheck[1]{{\mathpalette\@widecheck{#1}}}
\def\@widecheck#1#2{%
    \setbox\z@\hbox{\m@th$#1#2$}%
    \setbox\tw@\hbox{\m@th$#1%
       \widehat{%
          \vrule\@width\z@\@height\ht\z@
          \vrule\@height\z@\@width\wd\z@}$}%
    \dp\tw@-\ht\z@
    \@tempdima\ht\z@ \advance\@tempdima2\ht\tw@ \divide\@tempdima\thr@@
    \setbox\tw@\hbox{%
       \raise\@tempdima\hbox{\scalebox{1}[-1]{\lower\@tempdima\box
\tw@}}}%
    {\ooalign{\box\tw@ \cr \box\z@}}}
\makeatother

\newcommand{\mc}[1]{\mathcal{#1}}
\newcommand{\mf}[1]{\mathfrak{#1}}
\newcommand{\mb}[1]{\mathbb{#1}}
\newcommand{\id}{\mathbbm{1}}

\DeclareMathOperator{\Hom}{Hom}
\DeclareMathOperator{\End}{End}
\DeclareMathOperator{\ad}{ad}
\DeclareMathOperator{\Cend}{Cend}
\DeclareMathOperator{\RCend}{RCend}
\DeclareMathOperator{\Vect}{Vect}

\newcommand{\vac}{|0\rangle}
\def\ch{\mathrm{ch}}
\def\cl{\mathrm{cl}}
\DeclareMathOperator{\Aut}{Aut}

\makeatletter
\def\thmhead@plain#1#2#3{%
  \thmname{#1}\thmnumber{\@ifnotempty{#1}{ }\@upn{#2}}%
  \thmnote{ {\the\thm@notefont#3}}}
\let\thmhead\thmhead@plain
\def\swappedhead#1#2#3{%
  \thmnumber{#2}%
  \thmname{\@ifnotempty{#2}{~}#1}%
  \thmnote{ {\the\thm@notefont#3}}}
\let\swappedhead@plain=\swappedhead
\makeatother

\makeatletter
\def\th@definition{
  \thm@notefont{}%
  \normalfont
}
\makeatother

\theoremstyle{plain}
\newtheorem{theorem}{Theorem}[section]
\newtheorem{lemma}[theorem]{Lemma}
\newtheorem{proposition}[theorem]{Proposition}

\theoremstyle{definition}
\newtheorem{definition}[theorem]{Definition}
\newtheorem{example}[theorem]{Example}

\theoremstyle{remark}
\newtheorem{remark}[theorem]{Remark}

\numberwithin{equation}{section}

\definecolor{light}{gray}{.9}

\begin{document}

\title[Conformal operads and basic VA cohomology]{Conformal operads and the basic vertex algebra cohomology complex}

\author{Alberto De Sole}
\address{
Dipartimento di Matematica, Universit\`a di Roma La Sapienza, \& INFN,
P.le A. Moro 5, 00185 Roma, Italy}
\email{desole@mat.uniroma1.it}
\urladdr{www1.mat.uniroma1.it/\$$\sim$\$desole}

\author{Victor G. Kac}
\address{Department of Mathematics, MIT, Cambridge MA 02138}
\email{kac@math.mit.edu}
\urladdr{https://math.mit.edu/\$$\sim$\$kac/}

\author{Reimundo Heluani}
\address{IMPA, Rio De Janeiro, Brasil}
\email{heluani@potuz.net}
\urladdr{https://w3.impa.br/\$$\sim$\$heluani/}




\begin{abstract}
We develop the notion of a (pro-) conformal pseudo operad and apply it to the construction
of the basic cohomology complex of a vertex algebra.
The paper heavily uses the ideas and constructions of the work of Tamarkin \cite{Tam02}.
\end{abstract}


\maketitle

\tableofcontents

\section{Introduction}\label{sec:1}

The developments of conformal field theory and the theory of integrable systems gave rise to the study of novel algebraic structures, such as vertex algebras, Lie conformal algebras and Poisson vertex algebras,
see \cite{BPZ84,Bor86,Kac98,FBZ04,DSK06} and many other papers and books.

A unified point of view on cohomologies of these algebraic structures is provided by the theory of \emph{linear unital symmetric superoperads} \cite{Mar96, MSS02,LV12}. 
Such an operad ${\mathscr P}$ is a sequence of vector superspaces ${\mathscr P}(n),\,n\in\mb Z_{\geq0}$, with right linear actions of the symmetric groups $S_n$, $n\geq1$,
endowed with parity preserving linear maps \eqref{eq:operad1} called composition maps, satisfying the associativity axiom \eqref{eq:operad2} and the equivariance condition \eqref{eq:operad4}. The operad ${\mathscr P}$ is called \emph{unital} if there exists an element $1\in{\mathscr P}(1)$ satisfying \eqref{eq:operad3}.

In the paper we will use an equivalent definition for unital operad ${\mathscr P}$ in terms of bilinear products
\begin{equation}\label{eq:intro1}
\circ_i:\,{\mathscr P}(n)\times{\mathscr P}(m)\to{\mathscr P}(n+m-1)\,,\,\,i=1,\dots,n
\,,
\end{equation}
defined by \eqref{eq:operad8}.
These products satisfy the associativity axioms \eqref{eq:circ-assoc1} and \eqref{eq:circ-assoc2},
and the equivariance axiom \eqref{eq:circ-equiv}.
A sequence of vector superspaces ${\mathscr P}(n),\,n\in\mb Z_{\geq0}$,
endowed with a right action of $S_n$ and parity preserving bilinear products \eqref{eq:intro1} 
satisfying the associativity axioms \eqref{eq:circ-assoc1} and \eqref{eq:circ-assoc2},
and the equivariance axiom \eqref{eq:circ-equiv}, is called a \emph{pseudo operad} \cite{Mar96}.

The operadic point of view on cohomology was adopted in a series of papers \cite{BDSHK19,BDSHK20,BDSHKV21,BDSK20,BDSK21}.
Actually all operads considered in these papers are unital, hence can be equivalently defined in terms of pseudo operads, as in \cite[Sec.3.1]{BDSHK19}.
The main objects of study of the present paper are \emph{conformal} pseudo operads, or, more generally \emph{pro-conformal pseudo operads}, which are not necessarily unital.
The latter can be used in order to introduce the \emph{basic cohomology complex} for Lie conformal superalgebras and vertex algebras.

Under some conditions these complexes form a short exact sequence, which leads to a long exact sequence in cohomology. Also, basic cohomology complex allow Lie derivatives and contractions, so that the Lie derivatives are expressed via the contractions and the differential by Cartan's formula. These facts allowed to compute the cohomology of the most important Lie conformal algebras \cite{BKV99,BDSK20},
and of Poisson vertex algebras \cite{DSK13,BDSK20}.

In Sections \ref{sec:3} and \ref{sec:4} of the present paper we develop a theory of pro-conformal pseudo operads and the corresponding basic cohomology complexes.

In the paper \cite{BDSK21} we constructed a spectral sequence relating cohomology of vertex algebras to the classical cohomology of Poisson vertex algebras. Since for freely generated Poisson vertex algebras the latter coincide with variational Poisson cohomology \cite{BDSHKV21}, and there are well-developed methods to compute the latter \cite{DSK13,BDSK20}, this enabled us to get information on cohomology of many important freely generated vertex algebras. Unfortunately this method gives information only on the ``bounded'' part of the cohomology, and we do not know how to prove that this part coincides with the whole $j$-th cohomology for $j\geq2$.

For that reason it is important to develop an alternative approach to cohomology of vertex algebras,
which is to introduce the basic cohomology complex, for which the notion of a pro-conformal pseudo operad is essential.

A \emph{conformal pseudo operad} is a collection of vector superspaces $\widetilde{{\mathscr P}}(n),\,n\in\mb Z_{\geq0}$, endowed with even endomorphisms $\partial:\,\widetilde{{\mathscr P}}(n)\to\widetilde{{\mathscr P}}(n)$,
and parity preserving bilinear products $\circ^i_{\lambda}:\,\widetilde{{\mathscr P}}(n)\times \widetilde{{\mathscr P}}(m)\to\widetilde{{\mathscr P}}(n+m-1)[\lambda]$, for $i=1,\dots,n$, satisfying the sesquilinearity condisions
\begin{equation}\label{eq:intro2}
(\partial f)\circ^i_{\lambda} g = (\partial+\lambda)(f\circ^i_\lambda g)
\,\,,\,\,\,\,
g\circ^i_{\lambda}(\partial f) = -\lambda\, f\circ^i_\lambda g
\,,
\end{equation}
and the associativity conditions similar to \eqref{eq:circ-assoc1}-\eqref{eq:circ-assoc2}:
\begin{align}
(f\circ^i_{\lambda} g)\circ^j_{\mu} h
& =
(-1)^{p(g)p(h)}
(f\circ^j_{\mu} h)\circ^{\ell+i-1}_{\lambda} g  \qquad \text{ if } 1\leq j<i 
\,,\label{eq:intro3} \\
(f\circ^i_{\lambda} g)\circ^j_{\mu} h
& =
f\circ^i_{\lambda+\mu}(g\circ^{j-i+1}_{\mu} h) \qquad \text{ if } i\leq j<i+m 
\,,\label{eq:intro4}
\end{align}
where $f\in\widetilde{{\mathscr P}}(n),\,g\in\widetilde{{\mathscr P}}(m),\,h\in\widetilde{{\mathscr P}}(\ell)$.
Furthermore, it is required that $\widetilde{{\mathscr P}}(n)$ is endowed with a right linear action of the symmetric group $S_n$ commuting with $\partial$ and satisfying the equivariance condition similar to \eqref{eq:circ-equiv}
\begin{equation}\label{eq:circ-equiv-intro}
(f^\sigma)\circ_\lambda^i(g^{\tau})
=
\big(f\circ_\lambda^{\sigma(i)} g\big)^{\sigma\circ_i\tau}
\,,
\end{equation}
where the $\circ_i$-product $\sigma\circ_i\tau\in S_{m+n-1}$ of permutations is defined by \eqref{eq:operad19b}.

Note the similarity between the definitions of pseudo operads and conformal pseudo operads
is analogous to the similarity between the definitions of Lie algebras and Lie conformal algebras.
In particular, for the conformal pseudo operad $\widetilde{{\mathscr P}}$, the vector superspaces
${\mathscr P}(n)=\widetilde{{\mathscr P}}(n)/\partial \widetilde{{\mathscr P}}(n)$ form a pseudo operad with $f\circ^i g
=\tilde f\circ^i_0\tilde g\,\,$ mod$\langle\partial\rangle$ and the obvious action of $S_n$.

An important property of a (pseudo) operad ${\mathscr P}$, used in our papers, beginning with \cite{BDSHK19}, is that the vector superspaces
\begin{equation}\label{eq:intro6}
W({\mathscr P}) 
=
\bigoplus_{n\geq-1}W_n({\mathscr P})
\,,\,\,
\text{ where }
W_n({\mathscr P})={\mathscr P}(n+1)^{S_{n+1}}
\,,
\end{equation}
has a canonical structure of a $\mb Z$-graded Lie superalgebra with the bracket defined by
($f\in W_n({\mathscr P}),\,g\in W_m({\mathscr P})$):
\begin{equation}\label{eq:intro7}
[f,g]=f\Box g-(-1)^{p(f)p(g)}g\Box f
\,,\,\,\text{ where }
f\Box g
=
\sum_{\sigma\in S_{m+1,n}}(f\circ_1 g)^{\sigma^{-1}}
\,,
\end{equation}
where $S_{m,n}\subset S_{m+n}$ denotes the subset of $(m,n)$-shuffles.
The earliest references to this construction that we know of are \cite{Ger63} for non-symmetric
operads and \cite{Tam02} for symmetric ones,
the latter being the symmetrization of the former.

Likewise, for a conformal pseudo operad $\widetilde{{\mathscr P}}$ the vector superspace
\begin{equation}\label{eq:intro8}
\widetilde{W}(\widetilde{{\mathscr P}}) 
=
\bigoplus_{n\geq-1}\widetilde{W}_n(\widetilde{{\mathscr P}})
\,,\,\,
\text{ where }
\widetilde{W}_n(\widetilde{{\mathscr P}})=\widetilde{{\mathscr P}}(n+1)^{S_{n+1}}
\,,
\end{equation}
has a canonical structure of a $\mb Z$-graded Lie conformal superalgebra with 
the $\lambda$-bracket
($f\in \widetilde{W}_n(\widetilde{{\mathscr P}}),\,g\in \widetilde{W}_m(\widetilde{{\mathscr P}})$):
\begin{equation}\label{eq:intro9}
[f_\lambda g]=f\Box_\lambda g-(-1)^{p(f)p(g)}g\Box_{-\lambda-\partial} f
\,,\,\,\text{ where }
f\Box_\lambda g
=
\sum_{\sigma\in S_{m+1,n}}(f\circ^1_{-\lambda-\partial} g)^{\sigma^{-1}}
\,.
\end{equation}
Furthermore, we have a natural representation, denoted by $\ad$, 
of the Lie superalgebra $W(\widetilde{{\mathscr P}}/\partial\widetilde{{\mathscr P}})$ on 
the Lie conformal superalgebra $\widetilde{W}(\widetilde{{\mathscr P}})$
by its derivations, given by

\begin{equation}\label{eq:intro10}
(\ad {\bar f})(g)= f\Box_0 g-(-1)^{p(f)p(g)}g\Box_{-\partial} f
\,,
\end{equation}
where $\bar f\in W_m(\widetilde{{\mathscr P}}/\partial\widetilde{{\mathscr P}})$,
$g\in\widetilde{W}_n(\widetilde{{\mathscr P}})$, and $f$ is a preimage of $\bar f$
under the map
$\widetilde{W}(\widetilde{{\mathscr P}})\twoheadrightarrow W(\widetilde{{\mathscr P}}/\partial\widetilde{{\mathscr P}})$,
see Theorem \ref{20170603:thm2-conf}.

Given a (pseudo) operad ${\mathscr P}$,
one constructs a cohomology complex 
\begin{equation}\label{eq:intro12}
C({\mathscr P})
=
\big(\bigoplus_{j\geq0}{\mathscr P}(j),\ad X\big)
\,,
\end{equation}
for each odd element $X\in W_1({\mathscr P})\,\big(={\mathscr P}(2)^{S_2}\big)$,
such that $[X,X]=0$.
Similarly, given a conformal pseudo operad $\widetilde{{\mathscr P}}$,
one constructs a \emph{basic} cohomology complex 
\begin{equation}\label{eq:intro13}
\widetilde{C}(\widetilde{{\mathscr P}})
=
\big(\bigoplus_{j\geq0}\widetilde{P}(j),\ad X\big)
\,,
\end{equation}
for each odd element $X\in W_1(\widetilde{{\mathscr P}}/\partial\widetilde{{\mathscr P}})$
such that $[X,X]=0$, using the representation \eqref{eq:intro10}
of $W(\widetilde{{\mathscr P}}/\partial\widetilde{{\mathscr P}})$ on $\widetilde{W}(\widetilde{{\mathscr P}})$.

In order to apply the operadic construction described above to the definition of cohomology of an algebraic structure, one needs to find a linear operad which ``governs'' this algebraic structure.

The simplest example is the Lie superalgebra cohomology.
Given a vector superspace $V$ over a field $\mb F$, one considers the well-known linear operad 
${\mathscr P}=\mc Hom\,V$ (also denoted by $\mc End\, V$), for which
\begin{equation}\label{eq:intro14}
(\mc Hom\,V)(n)
=
\Hom_{\mb F}(V^{\otimes n},V)
\,.
\end{equation}
The (right) action of $S_n$ on $(\mc Hom\,V)(n)$ is defined via its natural (left) action on $V^{\otimes n}$
(taking into account the parity of $V$), and the $i$-th product $f\circ_i g$
of $f\in(\mc Hom\,V)(n)$ and $g\in(\mc Hom\,V)(m)$ is defined by ($i=1,\dots,n$):
\begin{equation}\label{eq:intro15}
\begin{split}
& (f\circ_i g)(v_1\otimes\dots\otimes v_{n+m-1}) \\
& \qquad =
\pm
f(v_1\otimes\dots\otimes v_{i-1}\otimes g(v_i\otimes\dots\otimes v_{i+m-1})\otimes v_{i+m}\otimes\dots\otimes v_{n+m-1})
\,,
\end{split}
\end{equation}
where the $\pm$ sign is given by the usual Koszul rule.
We consider the Lie superalgebra (see \eqref{eq:intro6})
\begin{equation}\label{eq:intro16}
W(\Pi V)
=
W(\mc Hom\,\Pi V)
=
\bigoplus_{j\geq-1}W_j(\Pi V)
\,,
\end{equation}
where $\Pi$ denotes the reversal of parity,
corresponding to the operad $\mc Hom\,\Pi V$,
associated to the vector superspace $\Pi V$.
This operad ``governs'' Lie superalgebras in the sense that the Lie superalgebra structures on the vector superspace $V$ are in bijective correspondence with the odd elements $X\in W_1(\Pi V)$
such that $[X,X]=0$, via the formula \cite{NR67,DSK13}
\begin{equation}\label{eq:intro17}
[a,b]=(-1)^{p(a)} X(a\otimes b)
\,,\,\, a,b\in V
\,.
\end{equation}
Then the cohomology complex of this Lie superalgebra with coefficients in the adjoint representation is
\begin{equation}\label{eq:intro18}
C(\Pi V)=\big(\bigoplus_{j\geq0}C^j(\Pi V),\ad X\big)
\,\,,\,\,\,\,\text{ where }
C^j(\Pi V)=W_{j-1}(\Pi V)
\,.
\end{equation}
The cohomology of $V$ with coefficients in any $V$-module $M$ can be obtained by a simple reduction
of the complex $C(\Pi(V\oplus M))$, where $V\oplus M$ is a semidirect sum of the Lie superalgebra $V$
and the $V$-module $M$, viewed as an abelian ideal.

The next in complexity example is the Lie conformal superalgebra cohomology,
reviewed in Section \ref{sec:chom}, \cite{BKV99,DSK09}.
The corresponding ``governing'' operad is $\mc Chom\,V$,
where $V$ is a vector superspace with an even endomorphism $\partial$.
Introduce the vector superspaces 
\begin{equation}\label{eq:intro19}
V_n=V[\lambda_1,\dots,\lambda_n]/\langle\partial+\lambda_1+\dots+\lambda_n\rangle
\,,
\end{equation}
where the $\lambda_i$'s are commuting even indeterminates and $\langle\Phi\rangle$
stands for the image of the endomorphism $\Phi$.
Then $(\mc Chom\,V)(n)$ consists of all linear maps
$f_{\lambda_1,\dots,\lambda_n}:\,V^{\otimes n}\to V_n$,
satisfying the sesquilinearity property
\begin{equation}\label{eq:intro20}
f_{\lambda_1,\dots,\lambda_n}(v_1\otimes\dots\otimes\partial v_i\otimes\dots\otimes v_n)
=-\lambda_i
f_{\lambda_1,\dots,\lambda_n}(v_1\otimes\dots\otimes v_n)
\,.
\end{equation}
The right action of $S_n$ on $(\mc Chom\,V)(n)$
is given by the simultaneous permutation of the factors of $V^{\otimes n}$ and of the $\lambda_i$'s.
The composition of maps is defined by \eqref{20170613:eq2},
and consequently the products $f\circ_i g$ are defined using \eqref{eq:operad8}.

In the same way as above, we consider the Lie superalgebra
\begin{equation}\label{eq:intro21}
W^\partial(\Pi V)
=
W(\mc Chom\,\Pi V)
=
\bigoplus_{j\geq-1}W^\partial_j(\Pi V)
\,,
\end{equation}
corresponding to the operad $\mc Chom\,\Pi V$,
and use the fact that 
the Lie conformal superalgebra structures on $V$ with the endomorphism $\partial$
are in bijective correspondence with the odd elements $X\in W^\partial_1(\Pi V)$
such that $[X,X]=0$, via the formula \cite{DSK13}
\begin{equation}\label{eq:intro22}
[a_\lambda b]=(-1)^{p(a)} X_{\lambda,-\lambda-\partial}(a\otimes b)
\,,\,\, a,b\in V
\,.
\end{equation}

The Lie superalgebra $W^{\partial,\text{ass}}(\Pi V)$ constructed in \cite[Sec.5.1]{DSK13}
is a subalgebra of $W^{\partial}(\Pi V)$ consisting of maps satisfying
the Leibniz rules, provided that $V$ carries the structure of a differential algebra.
It gives rise to the variational Poisson vertex algebra cohomology,
however it doesn't correspond to an operad.

The corresponding to $\mc Chom\,V$ conformal operad $\widetilde{\mc Chom}\,V$
is constructed as follows. Let (cf. \eqref{eq:intro19})
\begin{equation}\label{eq:intro23}
\widetilde{V}_n=V[\lambda_1,\dots,\lambda_n]
\,,
\end{equation}
and let $(\widetilde{\mc Chom}\,V)(n)$ denote the space of all linear maps
\begin{equation}\label{eq:intro24}
f_{\lambda_1,\dots,\lambda_n}:\,V^{\otimes n}\to \widetilde{V}_n
\,,
\end{equation}
satisfying the same sesquilinearity property \eqref{eq:intro20} as before.
We define an $\mb F[\partial]$-module structure on $(\widetilde{\mc Chom}\,V)(n)$ by letting
\begin{equation}\label{eq:intro25}
(\partial f)_{\lambda_1,\dots,\lambda_n}
=
(\partial+\lambda_1+\dots+\lambda_n)f_{\lambda_1,\dots,\lambda_n}
\,.
\end{equation}
The right action of $S_n$ on $(\widetilde{\mc Chom}\,V)(n)$ 
is defined in the same way as for $\mc Chom\,V$.
Finally, the $\circ^i_\lambda$-products are defined as in \eqref{20170613:eq2-pro}
in Section \ref{sec:5.1}.

If the $\mb F[\partial]$-module $V$ is finitely generated,
the values of the $\circ^i_\lambda$-products are polynomials in $\lambda$,
and in this case $\widetilde{\mc Chom}\,V$ is a conformal pseudo operad.
Otherwise we choose a filtration of $V$ by finitely generated $\mb F[\partial]$-submodules,
obtaining a \emph{pro-conformal} pseudo operad, as in \cite{Tam02}, see Section \ref{sec:4}.
Consequently, in general 
$\widetilde{W}(\widetilde{\mc Chom}\,V)$ is a pro-Lie conformal superalgebra.
As a result, for each $X\in W_1(\mc Chom\,\Pi V)$ such that $[X,X]=0$,
which defines a Lie conformal superalgebra structure on $V$ by \eqref{eq:intro22},
we obtain a cohomology complex, called the \emph{basic cohomology complex}
of this Lie conformal superalgebra.

Finally, the operad ``governing'' vertex algebras is called the \emph{chiral operad}
and denoted ${\mathscr P}^{\ch}$.
In their seminal book \cite{BD04}, Beilinson and Drinfeld generalized the notion of a vertex algebra,
introduced by Borcherds \cite{Bor86}, by defining a chiral algebra in the language of $\mc D$-modules
on any smooth algebraic curve,
so that a vertex algebra is a weakly translation covariant chiral algebra on the affine line.
In \cite{BDSHK19} we translated this construction for the affine line
to the purely algebraic language of vertex algebras.
The resulting operad ${\mathscr P}^{\ch}$, not surprisingly, 
turns out to be an extension of the operad $\mc Chom$,
in the same spirit as $\mc Chom$ is an extension of the operad $\mc Hom$.

In order to explain the construction of ${\mathscr P}^{\ch}$ (see Section \ref{sec:chiral}),
introduce, for each $n\in\mb Z_{\geq0}$,
the algebra $\mc O^{\star T}_n$ by letting $\mc O^{\star T}_0=\mc O^{\star T}_1=\mb F$,
and $\mc O^{\star T}_n=\mb F[(z_i-z_j)^{\pm1}\,|\,1\leq i<j\leq n]$ for $n\geq2$.
Given a vector superspace $V$ with an even endomorphism $\partial$,
let $({\mathscr P}^{\ch}V)(n)$ be the space of all linear maps
\begin{equation}\label{eq:intro26}
f_{\lambda_1,\dots,\lambda_n}^{z_1,\dots,z_n}:\,V^{\otimes n}\otimes\mc O^{\star T}_n\to V_n
\,,
\end{equation}
satisfying the following two sesquilinearity conditions:
\begin{align}
& f_{\lambda_1,\dots,\lambda_n}^{z_1,\dots,z_n}((\lambda_i+\partial_i) v\otimes p)
=
f_{\lambda_1,\dots,\lambda_n}^{z_1,\dots,z_n}(v\otimes\frac{\partial p}{\partial z_i})
\,, \label{eq:intro27} \\
& f_{\lambda_1,\dots,\lambda_n}^{z_1,\dots,z_n}(v\otimes (z_i-z_j)p)
=
\Big(
\frac{\partial}{\partial\lambda_j}-\frac{\partial}{\partial\lambda_i}
\Big)
f_{\lambda_1,\dots,\lambda_n}^{z_1,\dots,z_n}(v\otimes p)
\,, \label{eq:intro28}
\end{align}
for $v=v_1\otimes\dots\otimes v_n\in V^{\otimes n}$, $p=p(z_1,\dots,z_n)\in\mc O^{\star,T}_n$,
where we let $\partial_iv=v_1\otimes\dots\otimes\partial v_i\otimes\dots\otimes v_n$.
The symmetric group $S_n$ acts on $({\mathscr P}^{\ch}V)(n)$
by permuting simultaneously the factors $v_1,\dots,v_n$ of $V^{\otimes n}$, the indeterminates 
$\lambda_1,\dots,\lambda_n$ and the variables $z_1,\dots,z_n$.
The composition of maps in ${\mathscr P}^{\ch}V$ is more complicated, 
and it is explained in Section \ref{sec:chiral}.

In the same way as above, we consider the Lie superalgebra
\begin{equation}\label{eq:intro28b}
W^{\ch}(\Pi V)
=
W({\mathscr P}^{\ch}\Pi V)
=
\bigoplus_{j\geq-1}W^{\ch}_j(\Pi V)
\,,
\end{equation}
corresponding to the operad ${\mathscr P}^{\ch}V$, and use the fact that (non-unital)
vertex algebra structures on $V$ with endomorphism $\partial$
are in bijective correspondence with odd elements $X\in W_1^{\ch}(\Pi V)$
such that $[X,X]=0$, via the formula (cf. \eqref{eq:intro22}) \cite{BDSHK19}
\begin{equation}\label{eq:intro30}
\int^\lambda d\sigma
[a_\sigma b]=(-1)^{p(a)} X_{\lambda,-\lambda-\partial}^{z_1,z_2}(a\otimes b\otimes\frac1{z_2-z_1})
\,,\,\, a,b\in V
\,.
\end{equation}
Here we use the integral of $\lambda$-bracket definition of a vertex algebra,
developped in \cite{DSK06}, see also Section \ref{sec:chiral} of the present paper.

The corresponding to ${\mathscr P}^{\ch}V$ pro-conformal pseudo operad $\widetilde{{\mathscr P}}^{\ch}$
is constructed along the lines similar to $\widetilde{\mc Chom}\,V$.
We replace in \eqref{eq:intro26} the vector superspace $V_n$ by $\widetilde{V}_n$
(given by \eqref{eq:intro23}),
with the $\mb F[\partial]$-module structure \eqref{eq:intro25},
define the action of $S_n$ on $(\widetilde{{\mathscr P}}^{\ch}V)(n)$
in the same way as for  $({{\mathscr P}}^{\ch}V)(n)$,
and define the products $\circ^i_\lambda$ by \eqref{20170613:eq2-pro2}.
As for $\widetilde{\mc Chom}\,V$,
these products take values in formal power series of $\lambda$, not in polynomials, 
hence $\widetilde{{\mathscr P}}^{\ch}V$ is a pro-conformal operad,
like $\widetilde{\mc Chom}\,V$.

As for $\widetilde{\mc Chom}\,V$, we have a representation of the Lie superalgebra
$W^\ch({\mathscr P}^\ch V)$ on the pro-Lie conformal superalgebra $\widetilde{W}^\ch(\widetilde{{\mathscr P}}^\ch V)$
by its derivations. As a result, for each odd $X\in W_1^\ch({\mathscr P}^\ch V)$ such that $[X,X]=0$,
which defines a (non-unital) vertex algebra structure on $V$,
we obtain a cohomology complex, called the \emph{basic cohomology complex} of this vertex algebra,
see Section \ref{rec:chiral-conf}.

Unfortunately the important property that the canonical maps 
$\widetilde{W}(\widetilde{\mc Chom}\,V)\to W(\mc Chom\,V)$
and $\widetilde{W}(\widetilde{{\mathscr P}}^{\ch}V)\to W({\mathscr P}^{\ch}V)$
are surjective holds only for a free $\mb F[\partial]$-module $V$,
see Propositions \ref{prop:basic-cohom} and \ref{prop:basic-ch-cohom} respectively.
In the case of Lie conformal superalgebras, this is not a problem, since in this case torsion is central.
This has lead to successful application of the method of basic cohomology to the computation of cohomology of Lie conformal superalgebras and PVA \cite{BDSK20}.
However for the unital vertex algebras the map 
$\varphi:\,\widetilde{W}(\widetilde{{\mathscr P}}^{\ch}V)\to W({\mathscr P}^{\ch}V)$
is not surjective, as Example \ref{ex:V=C} shows.
Hence the corresponding to this map cohomology long exact sequence gives information 
on the cohomology of the subcomplex Im\,$\varphi\subset W({\mathscr P}^\ch V)$.
We are planning to study this problem in a subsequent publication.

In Section \ref{sec:7} we generalize the construction of the previous section to the general context 
of pseudo tensor categories.

Throughout the paper the base field $\mb F$ has characteristics 0,
and all Homs and tensor products $\otimes$ are over $\mb F$ unless otherwise specified.

\section{Linear (pseudo) operads and vertex algebra cohomology}\label{sec:2}

In the present Section we review the definition of a (linear, unital, symmetric, super) operad ${\mathscr P}$ 
and the construction of the Lie superalgebra $W({\mathscr P})$ associated to it, following \cite{BDSHK19}.
For extended reviews on the theory of operads, see e.g. \cite{LV12,MSS02,Yau16}.

\subsection{Definition of a linear operad}\label{sec:2.1}

Recall that a vector superspace is a $\mb Z/2\mb Z$-graded vector space $V=V_{\bar 0}\oplus V_{\bar 1}$.
We denote by $p(v)\in\mb Z/2\mb Z=\{\bar 0,\bar 1\}$ the parity of a homogeneous element $v\in V$.
We also let $\Pi V$ be the same vector space with the reversed parity $\bar p=1-p$.

Recall that 
a (linear, symmetric, super) \emph{operad} ${\mathscr P}$ is a collection 
of vector superspaces ${\mathscr P}(n)$, $n\geq0$, with parity $p$,
endowed with parity preserving linear maps, called \emph{composition maps},
associated to integers $n\geq1, m_1,\dots,m_n\geq0$, 
\begin{equation}\label{eq:operad1}
\begin{split}
{\mathscr P}(n) \otimes {\mathscr P}(m_1)\otimes\dots\otimes{\mathscr P}(m_n)\,\,&\to\,\,{\mathscr P}(M_n) \,, \\
f \otimes g_1 \otimes\dots\otimes g_n \,\, &\mapsto \,\, f(g_1\otimes\dots\otimes g_n) \,,
\end{split}
\end{equation}
where we denote $M_j:=\sum_{k=1}^jm_j$, $0\leq j\leq n$,
satisfying the following \emph{associativity axiom}:
\begin{equation}\label{eq:operad2}
f\big(
(g_1\otimes\dots\otimes g_n)
(h_1\otimes\dots\otimes h_{M_n})
\big)
=
\big(f
(g_1\!\otimes\dots\otimes\! g_n)
\big)
(h_1\otimes\dots\otimes h_{M_n})
\,\in{\mathscr P}\Big(\sum_{j=1}^{M_n}\ell_j\Big)
\,,
\end{equation}
for every $f\in{\mathscr P}(n)$, $g_i\in{\mathscr P}(m_i)$ for $i=1,\dots,n$,
and $h_{j}\in{\mathscr P}(\ell_{j})$ for $j=1,\dots,M_n$, where $\ell_j\geq0$.
In the left-hand side of \eqref{eq:operad2}
the argument of $f$ is the tensor product of the linear maps $g_1,\dots,g_n$
applied to the element
$H_1\otimes\dots\otimes H_n
=h_1\otimes\dots\otimes h_{M_n}$,
where $H_j=h_{M_{j-1}+1}\otimes\dots\otimes h_{M_j}$,
defined by the rule
\begin{equation}\label{20170804:eq1}
(g_1\otimes\dots\otimes g_n)(H_1\otimes\dots\otimes H_n)
=
(-1)^{\sum_{i<j}p(g_j)p(H_i)}g_1(H_1)\otimes\dots\otimes g_n(H_n)
\,.
\end{equation}
The signs are obtained by the usual Koszul-Quillen rule:
every time two odd elements are switched, we change the sign.

Furthermore, it is assumed that,
for each $n\geq1$, ${\mathscr P}(n)$ has a right action of the symmetric group $S_n$,
denoted $f^\sigma$, for $f\in{\mathscr P}(n)$ and $\sigma\in S_n$,
satisfying the following \emph{equivariance axiom}
($f\in{\mathscr P}(n)$, $g_1\in{\mathscr P}(m_1),\dots,g_n\in{\mathscr P}(m_n)$,
$\sigma\in S_n$, $\tau_1\in S_{m_1},\dots,\tau_n\in S_{m_n}$):
\begin{equation}\label{eq:operad4}
f^\sigma(g_1^{\tau_1}\otimes \dots\otimes g_n^{\tau_n})
=
\big(f(\sigma(g_1\otimes\dots\otimes g_n))\big)^{\sigma(\tau_1,\dots,\tau_n)}
\,,
\end{equation}
where the composition $\sigma(\tau_1,\dots,\tau_n)\in S_{m_1+\dots+m_n}$ 
of permutations is defined as follows.
First, recall that $\sigma\in S_n$ is can be described as a permutation map
on the tensor power $V^{\otimes n}$ of a vector superspace $V$ via
\begin{equation}\label{eq:operad6}
\sigma(v_1\otimes\dots\otimes v_n)
:=
\epsilon_v(\sigma) 
\,
v_{\sigma^{-1}(1)}\otimes\dots\otimes v_{\sigma^{-1}(n)}
\,,
\end{equation}
where 
\begin{equation}\label{eq:sign}
\epsilon_v(\sigma)=\prod_{i<j\,|\,\sigma(i)>\sigma(j)}(-1)^{p(v_i)p(v_j)}
\end{equation}
(again, we follow the Koszul-Quillen rule for the sign factor.)
More generally, $\sigma$ acts as a permutation on tensor product of vector superspaces:
\begin{equation}\label{eq:operad6b}
\sigma\colon
W_1\otimes \dots\otimes W_n
\stackrel{\sim}{\longrightarrow}
W_{\sigma^{-1}(1)}\otimes\dots\otimes W_{\sigma^{-1}(n)}\, ,
\end{equation}
by the same formula.
Then, the composition $\sigma(\tau_1,\dots,\tau_n)\in S_{M_n}$ is defined by acting on 
$v=v_1\otimes\dots\otimes v_{M_n}\in V^{\otimes(m_1+\dots+m_n)}$ 
by first applying each $\tau_i\in S_{m_i}$ to the vector 
$w_i=v_{M_{i-1}+1}\otimes\dots\otimes v_{M_i}\,\in V^{\otimes m_i}$,
and then by applying $\sigma\in S_n$ to $\tau_1(w_1)\otimes\dots\otimes \tau_n(w_n)$,
viewed as an element of $W_1\otimes\dots\otimes W_n$,
with $W_i=V^{\otimes m_i}$.
Summarizing this in a formula, we have:
\begin{equation}\label{eq:operad19}
(\sigma(\tau_1,\dots,\tau_n))(v)
=
\sigma\big(
\tau_1(v_{1}\otimes\dots\otimes v_{M_1})
\otimes\dots\otimes
\tau_n(v_{M_{n-1}+1}\otimes\dots\otimes v_{M_n})
\big)\,.
\end{equation}

An operad ${\mathscr P}$ is said to be \emph{unital}
if it is endowed with a \emph{unit} element $1\in{\mathscr P}(1)$
satisfying the following \emph{unity axiom}: 
\begin{equation}\label{eq:operad3}
f(1\otimes\dots\otimes 1)=1(f)=f
\,,\,\,\text{ for every }\,\,f\in{\mathscr P}(n)
\,.
\end{equation}

\subsection{Operads vs pseudo operads}\label{sec:2.1b}

Given a unital operad ${\mathscr P}$, one defines, for each $i=1,\dots,n$, 
the $\circ_i$-product
$\circ_i\colon{\mathscr P}(n)\times{\mathscr P}(m)\to{\mathscr P}(n+m-1)$
by insertion in position $i$, i.e.
\begin{equation}\label{eq:operad8}
f\circ_i g=f(
\overbrace{1\otimes\dots\otimes 1}^{i-1}
\otimes\stackrel{\vphantom{\Big(}i}{g}\otimes
\overbrace{1\otimes\dots\otimes1}^{n-i})\,.
\end{equation}
The associativity \eqref{eq:operad2} translates, in terms of the $\circ_i$-products, as:
\begin{align}
(f\circ_i g)\circ_j h
& =
(-1)^{p(g)p(h)}
(f\circ_j h)\circ_{\ell+i-1} g  \qquad \text{ if } 1\leq j<i 
\,,\label{eq:circ-assoc1} \\
(f\circ_i g)\circ_j h
& =
f\circ_i (g\circ_{j-i+1} h) \qquad \text{ if } i\leq j<i+m 
\,,\label{eq:circ-assoc2} \\
(f\circ_i g)\circ_j h
& =
(-1)^{p(g)p(h)}
(f\circ_{j-m+1} h)\circ_i g \qquad \text{ if } i+m\leq j<n+m
\,.\label{eq:circ-assoc3}
\end{align}
In fact, equation \eqref{eq:circ-assoc3} is obtained from \eqref{eq:circ-assoc1}
by exchanging $g$ with $h$ and by replacing $j-m+1$ by $i$ and $i$ by $j$. 
Furthermore, the equivariance axiom \eqref{eq:operad4-conf} of an operad ${{\mathscr P}}$
is translated to the following equivariance conditions for the $\circ_i$-products:
for $f\in{{\mathscr P}}(n)$, $g\in{{\mathscr P}}(m)$,
$\sigma\in S_n$, $\tau\in S_{m}$, we have
\begin{equation}\label{eq:circ-equiv}
(f^\sigma)\circ_i(g^{\tau})
=
\big(f\circ_{\sigma(i)} g\big)^{\sigma\circ_i\tau}
\,.
\end{equation}
Recall that the $\circ_i$-product $\sigma\circ_i\tau\in S_{m+n-1}$ of permutations is defined as follows \cite{LV12}, \cite[Sec.2]{BDSHK19}.
First, as described in Section \ref{sec:2.1}, a permutation $\sigma\in S_n$ is uniquely determined by its action on the 
tensor product of vector superspaces as in \eqref{eq:operad6b}.
Then, the $\circ_i$-product $\sigma\circ_i\tau\in S_{m+n-1}$ is defined by acting on 
$v=v_1\otimes\dots\otimes v_{m+n-1}\in V^{\otimes(m+n-1)}$ 
by first applying $\tau\in S_{m}$ to $w=v_{i}\otimes\dots\otimes v_{i+m-1}\,\in V^{\otimes m}\in W=V^{\otimes m}$,
and then by applying $\sigma\in S_n$ to $v_1\otimes\dots\stackrel{i}{\widecheck{\vphantom{\big(}\tau(w)}}\dots\otimes v_n$
viewed as an element of $V\otimes\dots\stackrel{i}{\widecheck{\vphantom{\big(}W}} \dots\otimes V$.
Summarizing this in a formula, we have:
\begin{equation}\label{eq:operad19b}
(\sigma\circ_i\tau)(v_1\otimes\dots\otimes v_{m+n-1})
=
\sigma\big(
v_{1}\otimes\dots\stackrel{i}{\widecheck{\vphantom{\big(}\tau(w)}}\dots\otimes v_{m+n-1}
\big)
\big)\,.
\end{equation}

Conversely, knowing all the $\circ_i$-products allows to reconstruct,
thanks to the associativity axiom \eqref{eq:operad2} ,
all the composition maps, by
\begin{equation}\label{eq:operad8a}
f(g_1,\dots,g_n)
=
(\cdots((f\circ_1 g_1)\circ_{m_1+1}g_2)\cdots)\circ_{m_1+\dots+m_{n-1}+1} g_n
\,.
\end{equation}
One can check that, if the $\circ_i$ products satisfy 
the associativity \eqref{eq:circ-assoc1}-\eqref{eq:circ-assoc2}, and the equivariance \eqref{eq:circ-equiv},
then the composition maps \eqref{eq:operad8a} satisfy the axioms of an operad,
i.e. the associativity \eqref{eq:operad2-conf} and the equivariance \eqref{eq:operad4-conf}, 
respectively \cite{Mar96}.

We then define a (symmetric, super) \emph{pseudo operad} \cite{Mar96} as
a collection of vector superspaces ${{\mathscr P}}(n)$, $n\geq0$, with parity $p$,
with parity preserving $\circ_i$-products, 
${{\mathscr P}}(n)\times{{\mathscr P}}(m)\to{{\mathscr P}}(n+m-1)$, for $i=1,\dots,n$,
satisfying 
the associativity conditions \eqref{eq:circ-assoc1}-\eqref{eq:circ-assoc2},
and endowed with a right action of the symmetric group $S_n$
satisfying the equivariance condition \eqref{eq:circ-equiv}.
By the above observations, a pseudo operad is automatically an operad 
with the composition maps \eqref{eq:operad8a},
and a unital operad is automatically a pseudo operad
with the $\circ_i$-products \eqref{eq:operad8}.

\subsection{$\mb Z_{\geq-1}$-graded Lie superalgebra associated to a pseudo operad and cohomology complexes}\label{sec:2.2}

Given a pseudo operad ${\mathscr P}$,
consider the $\mb Z_{\geq-1}$-graded vector superspace
$W({\mathscr P})=\bigoplus_{n\geq-1}W_n$,
where
\begin{equation}
W_n={\mathscr P}(n+1)^{S_{n+1}}
=\big\{f\in{\mathscr P}(n+1)\,\big|\,f^\sigma=f\,\forall \sigma\in S_{n+1}\big\}
\,.
\end{equation}

One defines a Lie superalgebra bracket on $W$ as follows. Recall that
a permutation $\sigma\in S_{m+n}$ is called an $(m,n)$-\emph{shuffle} if
\begin{equation}\label{eq:perm6}
\sigma(1)<\dots<\sigma(m)
\,,\,\,
\sigma(m+1)<\dots<\sigma(m+n)
\,.
\end{equation}
We denote by $S_{m,n}\subset S_{m+n}$ the subset of $(m,n)$-shuffles.
By definition, $S_{n,0}=S_{0,n}=\{1\}$ for every $n\geq0$
and, by convention, we let $S_{m,n}=\emptyset$ if either $m$ or $n$ is negative.

We then define the $\Box$-product of $f\in W_n$ and $g\in W_m$ by
\begin{equation}\label{eq:box}
f\Box g
=
\sum_{\sigma\in S_{m+1,n}}
(f\circ_1 g)^{\sigma^{-1}}
\,\in W_{m+n}\,.
\end{equation}
Note that $S_{m+1,-1}=\emptyset$, hence $f\Box g=0$ if $f\in W_{-1}$.
For example,
for $f,g\in W_1$, we have
\begin{equation}\label{eq:box1}
f\Box g = f\circ_1 g + (f\circ_1 g)^{(23)} + (f\circ_1 g)^{(132)}\,.
\end{equation}

\begin{theorem}{\cite[Thm.3.4]{BDSHK19}}\label{20170603:thm2}
\begin{enumerate}[(a)]
\item
For\/ $f\in W_n$ and\/ $g\in W_m$, we have\/ $f\Box g\in W_{n+m}$.
\item
The associator of the\/ $\Box$-product is right supersymmetric, i.e.
\begin{equation}\label{20170608:eq2}
(f\Box g)\Box h-f\Box(g\Box h)
=(-1)^{p(g)p(h)}
(f\Box h)\Box g-f\Box(h\Box g)
\,.
\end{equation}
\item
Consequently, $W$ is a\/ $\mb Z$-graded Lie superalgebra
with bracket given by\/ ($f\in W_n$, $g\in W_m$)
\begin{equation}\label{20170603:eq4}
[f,g]=f\Box g-(-1)^{p(f)p(g)}g\Box f
\,.
\end{equation}
\end{enumerate}
\end{theorem}

Note that, if $X\in W_1$ is odd then the condition that $X\Box X=0$,
is equivalent, by the Jacobi identity for the Lie superalgebra $W({\mathscr P})$,
to $(\ad X)^2=0$.
\begin{definition}\label{def:univ-lie}
\begin{enumerate}[(a)]
\item
The $\mb Z_{\geq-1}$-\emph{graded Lie superalgebra} associated to the pseudo operad ${\mathscr P}$ 
is the $\mb Z_{\geq-1}$-graded superspace $W({\mathscr P})=\bigoplus_{n\geq-1}W_n$ 
with bracket defined by \eqref{20170603:eq4}.
\item
The \emph{cohomology complex} associated to ${\mathscr P}$
and an odd element $X\in W_1$ such that $X\Box X=0$
is $(W({\mathscr P}),\ad X)$.
\end{enumerate}
\end{definition}

\subsection{The unital operad $\mc{H}om(\Pi V)$ and Lie algebra cohomology}\label{sec:Hom}

Given a vector superspace $V=V_{\bar 0}\oplus V_{\bar 1}$,
we denote by $\Pi V$ the same space with reversed parity: $(\Pi V)_{\bar i}=V_{\bar 1-\bar i}$,
$\bar i\in\{\bar 0,\bar 1\}$.
We then consider the unital operad $\mc{H}om(\Pi V)$, defined as the collection of superspaces
$$
\mc{H}om(\Pi V)(n)=\Hom((\Pi V)^{\otimes n},\Pi V)\,\,,\qquad n\geq0\,.
$$
The composition maps are defined as the usual compositions:
for $(f:\,(\Pi V)^{\otimes n}\to \Pi V)\in\mc{H}om(\Pi V)(n)$ 
and $(g_i:\, (\Pi V)^{\otimes m_i}\to \Pi V)\in\mc{H}om(\Pi V)(m_i)$, 
 $i=1,\dots,n$, we let
\begin{equation}\label{eq:operad24}
(f(g_1\otimes \dots\otimes g_n))
(v_1\otimes\dots\otimes v_{M_n})
:=
f\big((g_1\otimes\dots\otimes g_n)(v_1\otimes\dots\otimes v_{M_n})\big)
\,,
\end{equation}
where the tensor product of maps 
$g_1\otimes\dots\otimes g_n:\,(\Pi V)^{\otimes m_1}\otimes\dots\otimes (\Pi V)^{\otimes m_n}
\to (\Pi V)^{\otimes n}$
was defined in \eqref{20170804:eq1}.
The unity element of the operad $\mc{H}om(\Pi V)$ is $1=\id_{\Pi V}\in\End(\Pi V)=\mc{H}om(\Pi V)(1)$.
The right action of $S_n$ on $\mc{H}om(\Pi V) (n)=\Hom((\Pi V)^{\otimes n},\Pi V)$
is given by
\begin{equation}\label{eq:operad7}
f^\sigma(v_1\otimes\dots\otimes v_n)
=
f(\sigma(v_1\otimes\dots\otimes v_n))
\,,
\end{equation}
with the action of $\sigma\in S_n$ on $(\Pi V)^{\otimes n}$ given by \eqref{eq:operad6}
(for the parity $\bar p=1-p$ of $\Pi V$).

Consider the Lie superalgebra $W(\Pi V):=W(\mc{H}om(\Pi V))$ defined in Section \ref{sec:2.2}.
\begin{proposition}[\cite{NR67,DSK13}]\label{20170612:prop1}
\begin{enumerate}[(a)]
\item
We have a bijective correspondence between 
the odd elements\/ $X\in W_1(\Pi V)$ such that\/ $X\Box X=0$
and the Lie superalgebra brackets\/
$[\cdot\,,\,\cdot]\colon V\times V\to V$ on\/ $V$,
given by 
\begin{equation}\label{20170612:eq3}
[a,b]=(-1)^{p(a)}X(a\otimes b)
\,.
\end{equation}
\item
Let $(V,[\cdot\,,\,\cdot])$ be a Lie superalgebra and let $X\in W_1(\Pi V)$ be the corresponding odd element
from (a).
The cohomology complex $(W(\Pi V),\ad X)$ from Definition \ref{def:univ-lie}(b)
coincides with the Chevalley-Eilemberg cohomology complex of the Lie superalgebra $V$
with coefficients in the adjoint representation.
\end{enumerate}
\end{proposition}

In the following subsections, we will repeat the same line of reasoning 
for the cohomology theories of Lie conformal superalgebras
and vertex algebras:
after reviewing their definition,
we will construct, for each of them, an operad ${\mathscr P}$,
and we will describe their algebraic structures as an element $X\in W_1\subset W({\mathscr P})$
such that $X\Box X=0$.
In this way, we automatically get, for each algebraic structure of interest,
the corresponding cohomology complex $(W({\mathscr P}),\ad X)$.

\subsection{The Chom unital operad and Lie conformal superalgebra cohomology}\label{sec:chom}

Let $V$ be a vector superspace with a given even endomorphism $\partial\in\End(V)$.
Recall that a \emph{Lie conformal superalgebra} structure on $V$ is a bilinear map
$[\cdot\,_\lambda\,\cdot]\colon V\times V\to V[\lambda]$, called $\lambda$-bracket, satisfying 
sesquilinearity ($a,b\in V$):
\begin{equation}\label{20170612:eq6}
[\partial a_\lambda b]=-\lambda[a_\lambda b]
\,,\,\,
[a_\lambda \partial b]=(\lambda+\partial)[a_\lambda b]
\,,
\end{equation}
skewsymmetry ($a,b\in V$):
\begin{equation}\label{20170612:eq4}
[a_\lambda b]=-(-1)^{p(a)p(b)}[b_{-\lambda-\partial}a]
\,,
\end{equation}
and the Jacobi identity ($a,b,c\in V$):
\begin{equation}\label{20170612:eq5}
[a_{\lambda}[b_\mu c]]-(-1)^{p(a)p(b)}[b_\mu [a_\lambda ,b]]
=[[a_\lambda b]_{\lambda+\mu}c]
\,.
\end{equation}

The unital operad $\mc{C}hom(\Pi V)$, ``governing'' the Lie conformal superalgebra structures on $V$, is defined as 
the collection of superspaces
$$
\mc{C}hom(\Pi V)(n)\subset\Hom\big((\Pi V)^{\otimes n}
,\Pi V[\lambda_1,\dots,\lambda_n]\big/\big\langle\partial+\lambda_1+\dots+\lambda_n\big\rangle\big)
\,\,,\,\,\,\,
n\geq0
\,,
$$
consisting of all linear maps
$$
f_{\lambda_1,\dots,\lambda_n}\colon (\Pi V)^{\otimes n}\to
\Pi V[\lambda_1,\dots,\lambda_n]\big/\big\langle\partial+\lambda_1+\dots+\lambda_n\big\rangle
\,,
$$
satisfying the sesquilinearity conditions:
\begin{equation}\label{20170613:eq1}
f_{\lambda_1,\dots,\lambda_n}(v_1\otimes\,\cdots\partial v_i\cdots\,\otimes v_n)
=-\lambda_if_{\lambda_1,\dots,\lambda_n}(v_1\otimes\,\cdots\,\otimes v_n)
\,\,\text{ for all } i=1,\dots,n
\,.
\end{equation}
Hereafter
$\lambda_1,\dots,\lambda_k$ are commuting indeterminates of even parity
and $\langle\Phi\rangle$ denotes the image of an endomorphim $\Phi$.
In particular, $\mc{C}hom(\Pi V)(0)=\Pi V/\partial\Pi V$
and $\mc{C}hom(\Pi V)(1)=\End_{\mb F[\partial]}(V)$.

The unity in the operad $\mc{C}hom(\Pi V)$ is $1=\id_V\in\mc{C}hom(\Pi V)(1)=\End_{\mb F[\partial]} V$,
and the right action of $\sigma\in S_n$ on $\mc{C}hom(\Pi V)(n)$
is given by (cf. \eqref{eq:operad6} and \eqref{eq:operad7}):
\begin{equation}\label{20170613:eq3}
\begin{array}{l}
\displaystyle{
\vphantom{\Big(}
(f^\sigma)_{\lambda_1,\dots,\lambda_n}(v_{\sigma(1)}\otimes\dots\otimes v_{\sigma(n)})
=
f_{\sigma(\lambda_1,\dots,\lambda_n)}
(\sigma(v_{\sigma(1)}\otimes\dots\otimes v_{\sigma(n)}))
} \\
\displaystyle{
\vphantom{\Big(}
\,\,\,\,\,\,\,\,\,\,\,\,\,\,\,\,\,\,
=
\epsilon_v(\sigma) 
f_{\lambda_{\sigma^{-1}(1)},\dots,\lambda_{\sigma^{-1}(n)}}
(v_1\otimes\dots\otimes v_n)
\,,}
\end{array}
\end{equation}
where $\epsilon_v(\sigma)$ is given by \eqref{eq:sign} for the parity $\bar p=1-p$ of $\Pi V$.
The composition of $f\in\mc{C}hom(\Pi V)(n)$
and $g_1\in\mc{C}hom(\Pi V)(m_1),\dots,g_n\in\mc{C}hom(\Pi V)(m_n)$ is defined by
\begin{equation}\label{20170613:eq2}
\begin{array}{l}
\displaystyle{
\vphantom{\Big(}
\big(f(g_1\otimes\dots\otimes g_n)\big)_{\lambda_1,\dots,\lambda_{M_n}}
(v_1\otimes\dots\otimes v_{M_n})
} \\
\displaystyle{
\vphantom{\Big(}
:=
f_{\Lambda_1,\dots,\Lambda_n}
\big(
((g_1)_{\lambda_{1},\dots,\lambda_{M_1}}
\otimes\dots\otimes
(g_n)_{\lambda_{M_{n-1}+1},\dots,\lambda_{M_n}})
(v_1\otimes\dots\otimes v_{M_n})
\big)
\,,}
\end{array}
\end{equation}
where we let
$\Lambda_i=\sum_{j=M_{i-1}+1}^{M_i}\lambda_j$, $i=1,\dots,n$,
and, recalling \eqref{20170804:eq1}, we have
\begin{equation}\label{20170821:eq5}
\begin{array}{l}
\displaystyle{
\vphantom{\Big(}
((g_1)_{\lambda_{1},\dots,\lambda_{M_1}}
\otimes\dots\otimes
(g_n)_{\lambda_{M_{n-1}+1},\dots,\lambda_{M_n}})
(v_1\otimes\dots\otimes v_{M_n})
} \\
\displaystyle{
\vphantom{\Big(}
=\pm\,
(g_1)_{\lambda_{1},\dots,\lambda_{M_1}}\!\!(v_1\otimes\dots\otimes v_{M_1})
\otimes\dots\otimes
(g_n)_{\lambda_{M_{\!n\!-\!1}\!+1},\dots,\lambda_{M_n}}\!\!(v_{M_{\!n\!-\!1}\!+1}\otimes\dots\otimes v_{M_n})
\,,}
\end{array}
\end{equation}
where the sign $\pm$ is given by
\begin{equation}\label{20170821:eq5b}
\pm=(-1)^{\sum_{i<j} \bar p(g_j)(\bar p(v_{M_{i-1}+1})+\dots+\bar p(v_{M_i}))}
\end{equation}

The proof of the unity, equivariance and associativity axioms for the unital operad 
$\mc{C}hom(\Pi V)$ can be found in \cite[Sec.5.2]{BDSHK19}.

Consider the $\mb Z_{\geq-1}$-graded Lie superalgebra 
$W^\partial(\Pi V):=W(\mc{C}hom(\Pi V))$ from Defintion \ref{def:univ-lie}.
\begin{proposition}{\cite[Prop.5.1]{BDSHK19}}\label{20170612:prop2}
\begin{enumerate}[(a)]
\item
We have a bijective correspondence between 
the odd elements\/ $X\in W^\partial_1(\Pi V)$ such that\/ $X\Box X=0$
and the Lie conformal superalgebra\/ $\lambda$-brackets
$[\cdot\,_\lambda\,\cdot]\colon V\times V\to V[\lambda]$ on\/ $V$,
given by 
\begin{equation}\label{20170612:eq3a}
[a_\lambda b]=(-1)^{p(a)}X_{\lambda,-\lambda-\partial}(a\otimes b)
\,.
\end{equation}
\item
Let $(V,[\cdot\,_\lambda\,\cdot])$ be a Lie conformal superalgebra and 
let $X\in W^\partial_1(\Pi V)$ be the corresponding odd element
from (a).
The cohomology complex $(W^\partial(\Pi V),\ad X)$ from Definition \ref{def:univ-lie}(b)
coincides with the cohomology complex of the Lie conformal superalgebra $V$
introduced in \cite{BKV99,DSK09} with coefficients in the adjoint representation.
\end{enumerate}
\end{proposition}

\subsection{The chiral operad and vertex algebra cohomology}\label{sec:chiral}

As in the previous subsection, $V$ is a vector superspace over $\mb F$ 
with a given even endomorphism $\partial\in\End(V)$.

Following \cite{DSK06}, we review the ``fifth definition'' of a (non unital) \emph{vertex algebra} 
structure on $V$.
It is given by an {\itshape integral} $\lambda$-{\itshape bracket},
namely a linear map $V \otimes V \to \mb F [\lambda] \otimes V$, 
denoted by 
$$
u\otimes v\,\mapsto\,\int^\lambda d\sigma[u_\sigma v]
=
{:}uv{:}+\int_0^\lambda\!d\sigma\,[u_\sigma v]
\,,
$$
such that the following axioms hold ($u,v,w\in V$):
\begin{enumerate}[(i)]
\item
$\int^\lambda\!\! d\sigma\, [\partial u_\sigma v] = -\int^\lambda\!\! d\sigma\,\sigma[u_\sigma v]$,
$\int^\lambda\!\! d\sigma\,[u_\sigma \partial v] = \int^\lambda\!\!d\sigma\, (\partial+\sigma)[u_\sigma v]$,
\item 
$\int^\lambda\!\! d\sigma\,[v_\sigma u] =(-1)^{p(u)p(v)}\int^{-\lambda-\partial}\!\! d\sigma\,[u_\sigma v]$,
\item
$\int^\lambda\!\! d\sigma\!\! \int^\mu\!\! d\tau \Big(
[u_\sigma[v_\tau w]] - (-1)^{p(u)p(v)} [v_\tau[u_\sigma w]] - [[u_\sigma v]_{\sigma+\tau}w] 
\Big) = 0$.
\end{enumerate}
The axioms (i), (ii), (iii) are called the sesquilinearity, skewsymmetry and Jacobi identity axioms, respectively. 
In a \emph{unital} vertex algebra one also requires the existence of an element 
$\vac\in V_{\bar{0}}$ such that $\partial\vac=0$ and 
$\int^\lambda d\sigma[\vac_\sigma v]=\int^\lambda d\sigma[v_\sigma \vac]=v$.

The $\lambda$-\emph{bracket} $[u_\lambda v]$ is then
defined as the derivative by $\lambda$ of the polynomial $\int^\lambda\!d\sigma\,[u_\sigma v]$,
while the \emph{normally ordered product} ${:}uv{:}$ 
is defined as the constant term of $\int^\lambda\!d\sigma\,[u_\sigma v]$.
Axioms (i)-(iii) are a concise way to write more complicated relations
involving the normally ordered products ${:}uv{:}$ and the $\lambda$-brackets $[u_\lambda v]$.
We refer to \cite{DSK06} for a detailed discussion.

A left \emph{module} $M$ over a non-unital vertex algebra $V$ is a vector superspace 
with an even endomorphism $\partial^M\in\End(M)$, endowed with a
parity preserving \emph{integral} $\lambda$-\emph{action}
$V\otimes M\to\mb F[\lambda]\otimes M$, denoted by
$$
v\otimes m\,\mapsto\,\int^\lambda d\sigma(v_\sigma m)\,,
$$
such that the following axioms hold
($u,v\in V,\,m\in M$):
\begin{enumerate}[(i)]
\item
$\int^\lambda\!\! d\sigma\, (\partial v_\sigma v) = -\int^\lambda\!\! d\sigma\,\sigma(v_\sigma m)$,
$\int^\lambda\!\! d\sigma\,(v_\sigma \partial^Mm) = \int^\lambda\!\!d\sigma\, (\partial^M+\sigma)(v_\sigma m)$,
\item
$\int^\lambda\!\! d\sigma\!\! \int^\mu\!\! d\tau \Big(
(u_\sigma(v_\tau m)) - (-1)^{p(u)p(v)} (v_\tau(u_\sigma m)) - ([u_\sigma v]_{\sigma+\tau}m)
\Big) = 0$.
\end{enumerate}

Let $n$ be a non-negative integer.
We denote by $\mc O_{n}^{\star T}$ be the algebra of 
translation invariant polynomials in the even variables $z_1,\dots,z_n$ 
localized on the diagonals $z_i=z_j$ for $i\neq j$,
i.e., $\mc O_{0}^{\star T}=\mc O_{1}^{\star T}=\mb F$ and
\begin{equation}\label{eq:Ostarn}
\mc O_{n}^{\star T}
=\mb F[(z_i-z_i)^{\pm 1}]_{1\leq i<j\leq n}
\,\,,\qquad n\geq2.
\end{equation}
The unital operad ${\mathscr P}^{\ch}(\Pi V)$, governing (non-unital) vertex algebra structures on $V$,
is defined as the collection of superspaces
$$
\mc{P}^{\ch}(\Pi V)(n)\subset\Hom\big((\Pi V)^{\otimes n}\otimes\mc O_{n}^{\star T}
,\Pi V[\lambda_1,\dots,\lambda_n]\big/\big\langle\partial+\lambda_1+\dots+\lambda_n\big\rangle\big)
\,\,,\,\,\,\,
n\geq0
\,,
$$
consisting of all linear maps
$$
f_{\lambda_1,\dots,\lambda_n}\colon (\Pi V)^{\otimes n}\otimes\mc O_{n}^{\star T}\to
\Pi V[\lambda_1,\dots,\lambda_n]\big/\big\langle\partial+\lambda_1+\dots+\lambda_n\big\rangle
\,,
$$
satisfying the following two \emph{sesquilinearity} conditions:
\begin{equation}\label{20160629:eq4}
\begin{array}{l}
\displaystyle{
\vphantom{\Big(}
f_{\lambda_1,\dots,\lambda_n}(\partial_iv\otimes p)
+\lambda_if_{\lambda_1,\dots,\lambda_n}(v\otimes p)
=f_{\lambda_1,\dots,\lambda_n}\Bigl(v\otimes \frac{\partial p}{\partial z_i}\Bigr)
\,,} \\
\displaystyle{
\vphantom{\Big(}
f_{\lambda_1,\dots,\lambda_n}(v\otimes (z_i-z_j)p)
=\Bigl(\frac{\partial}{\partial\lambda_j}-\frac{\partial}{\partial\lambda_i}\Bigr)
f_{\lambda_1,\dots,\lambda_n}(v\otimes p)
\,,}
\end{array}
\end{equation}
for $v=v_1\otimes\dots\otimes v_n\in(\Pi V)^{\otimes n}$, $p=p(z_1,\dots,z_n)\in\mc O_{n}^{\star T}$,
where we let $\partial_iv=v_1\otimes\dots\partial v_i\dots\otimes v_n$.
In particular,
${\mathscr P}^\ch(\Pi V)(0) = \Pi V/\partial \Pi V$
and
${\mathscr P}^\ch(\Pi V)(1) = \End_{\mb F[\partial]} (V)$.
Sometimes, in order to specify the variables of the function $p\in\mc O_{n}^{\star T}$,
we will denote the image of the map $f$ as
\begin{equation}\label{20160709:eq1}
f^{z_1,\dots,z_n}_{\lambda_1,\dots,\lambda_1}(v_1\otimes\dots\otimes v_n\otimes p(z_1,\dots,z_n))
\,.
\end{equation}

The unity of the operad ${\mathscr P}^{\ch}(\Pi V)$ is the identity endomorphism 
$1=\id_V$, viewed as an element of ${\mathscr P}^\ch(\Pi V)(1)= \End_{\mb F[\partial]} (V)$.
The symmetric group $S_{n}$ has a right action on ${\mathscr P}^\ch(\Pi V)(n)$ by permuting simultaneously the vectors $v_1,\dots,v_n\in\Pi V$, the indeterminates $\lambda_1,\dots,\lambda_n$, 
and the corresponding variables $z_1,\dots,z_n$ in $p\in\mc O_n^{\star T}$. 
Explicitly, for $f\in{\mathscr P}^\ch(\Pi V)(n)$ and $\sigma \in S_{n}$, we let
\begin{equation}\label{20160629:eq5}
\begin{split}
& (f^\sigma)^{z_1,\dots,z_n}_{\lambda_1,\dots,\lambda_n}(v_1\otimes\dots\otimes v_n\otimes p(z_1,\dots,z_n))
\\
&=
\epsilon_v(\sigma)
f^{z_1,\dots,z_n}_{\lambda_{\sigma^{-1}(1)},\dots,\lambda_{\sigma^{-1}(n)}}
\big(v_{\sigma^{-1}(1)}\otimes\dots\otimes v_{\sigma^{-1}(n)};p(z_{\sigma^{-1}(1)},\dots,z_{\sigma^{-1}(n)})\big)
\,,
\end{split}
\end{equation}
where $\epsilon_v(\sigma)$ is given by \eqref{eq:sign} for the parity $\bar p=1-p$ of $\Pi V$.

To define the structure of an operad on ${\mathscr P}^{\ch}(\Pi V)$ 
we are left to specify how elements in ${\mathscr P}^\ch(\Pi V)$ are composed. 
Fix $n\geq1$ and $m_1,\dots,m_n\geq0$,.
Let $f_{\lambda_1,\dots,\lambda_n}\in{\mathscr P}^\ch(\Pi V)(n)$ 
and $(g_i)_{\lambda_1,\dots,\lambda_{m_i}}\in{\mathscr P}^\ch(\Pi V)(m_i)$, $i=1,\dots,n$.
In order to define their composition
$f(g_1\otimes\dots\otimes g_n)_{\lambda_1,\dots,\lambda_{M_n}}$,
where, as before, $M_n=\sum_{i=1}^nm_i$,
let $v=v_1\otimes\dots\otimes v_{M_n}\in(\Pi V)^{\otimes M_n}$ 
and $p(z_1,\dots,z_{M_n})\in\mc O_{M_n}^{\star T}$.
Let
\begin{equation}\label{circ3}
w_i=v_{M_{i-1}+1}\otimes\dots\otimes v_{M_i}\,\in V^{\otimes m_i}, \qquad i=1,\dots,n.
\end{equation}
We can decompose
\begin{equation}\label{circ4}
p(z_1,\dots,z_{M_n}) = h(z_1,\dots,z_{M_n}) \prod_{i=1}^n q_i(z_{M_{i-1}+1},\dots, z_{M_i}), 
\end{equation}
for some $q_i \in \mc O_{m_i}^{\star T}$ and $h\in\mc O_{M}^{\star T}$ which has no poles at
$z_k=z_l$ for any $k,l$ such that $M_{i-1}+1 \le k<l \le M_i$ ($1\le i\le n$).
Then, by definition,
\begin{equation}\label{circ5}
\begin{split}
&\bigl(f(g_1\otimes\dots\otimes g_n)\bigr)_{\lambda_1,\dots,\lambda_{M_n}}^{z_1,\dots,z_{M_n}} 
\bigl(v_1\otimes\dots\otimes v_{M_n}; p(z_1,\dots,z_{M_n})\bigr) \\
&= \pm f_{\Lambda_1,\dots,\Lambda_n}^{Z_1,\dots,Z_n} 
\Bigl(
(g_1)_{\lambda_1-\partial_{z_1},\dots,\lambda_{m_1}-\partial_{z_{m_1}}}^{z_1,\dots,z_{m_1}} 
(v_1\otimes\dots\otimes v_{m_1}\otimes q_1(z_1,\dots,z_{m_1}) )_{\to} \otimes \dots \\
& \dots \otimes
(g_n)_{\lambda_{M_{n\!-\!1}\!+\!1}\!-\partial_{z_{M_{n\!-\!1}\!+\!1}},\dots,\lambda_{M_n}\!-\partial_{z_{M_n}}}^{z_{M_{n\!-\!1}\!+\!1},\dots,z_{M_n}} 
(v_{M_{n\!-\!1}\!+\!1}\otimes\dots \ldots\otimes v_{M_n}\otimes q_n(z_{M_{n\!-\!1}\!+\!1},\dots \\
& \qquad \dots ,z_{M_n}))_{\to}
\otimes h(z_1,\dots,z_{M_n}) \big|_{z_{M_{i-1}+1}=\dots=z_{M_i}=Z_i \,\forall 1\le i\le n)} \Bigr),
\end{split}
\end{equation}
where $\Lambda_i=\sum_{j=M_{i-1}+1}^{M_i}\lambda_j$ and the sign $\pm$ is 
given by \eqref{20170821:eq5b}.
In the right-hand side of \eqref{circ5}, the arrows means that 
we first take the partial derivatives of $h$ and then we make the substitutions 
$z_{M_{i-1}+1}=\dots=z_{M_i}=Z_i$, $i=1,\dots,n$.

It is proved in \cite[Prop.6.7]{BDSHK19} that
the collection of vector superspaces ${\mathscr P}^\ch(n)$, $n\ge 0$, 
with the unity $\id_V$, the action of $S_n$ defined by \eqref{20160629:eq5} 
and the composition maps defined by \eqref{circ5}, is a unital operad.

Next, consider the $\mb Z_{\geq-1}$-graded Lie superalgebra 
$W^\ch(\Pi V):=W({\mathscr P}^{\ch})(\Pi V))$ from Defintion \ref{def:univ-lie}(a).
\begin{proposition}{\cite[Thm.6.12]{BDSHK19}}\label{20170612:prop2-ch}
We have a bijective correspondence between 
the odd elements\/ $X\in W^\ch_1(\Pi V)$ such that\/ $X\Box X=0$
and the structure of (non-unital) vertex algebras on $V$.
In particular, the integral $\lambda$-bracket associated to $X$ is given by
($u,v\in V$)
\begin{equation}\label{20160719:eq3}
\int^{\lambda}\!d\sigma[u_\sigma v]
=
(-1)^{p(u)}X^{z_1,z_2}_{\lambda,-\lambda-\partial}\bigl(u\otimes v\otimes \frac1{z_2-z_1}\bigr)
\,.
\end{equation}
\end{proposition}
\begin{definition}
Let $(V,\int^\lambda\!d\sigma[\cdot\,_\sigma\,\cdot])$ be a (non-unital) vertex algebra and 
let $X\in W^\ch_1(\Pi V)$ be the corresponding odd element from Proposition \ref{20170612:prop2-ch}.
The \emph{vertex algebra cohomology complex} of $V$ 
(with coefficients in $V$)
is the cohomology complex $(W^\ch(\Pi V),\ad X)$ from Definition \ref{def:univ-lie}(b).
\end{definition}

\begin{remark}\label{rem:27}
We have constructed cohomology complexes for Lie superalgebras, Lie conformal superalgebras
and vertex algebras, by considering respectively the operads 
$\mc Hom(\Pi V), \mc{C}hom (\Pi V)$ and ${\mathscr P}^{\ch}(\Pi V)$
associated to a given vector superspace $V$.
For this purpose we constructed the associated $\mb Z_{\geq-1}$-graded Lie superalgebras
$W(\Pi V), W^\partial(\Pi V)$ and $W^\ch(\Pi V)$,
such that there is a bijective correspondence between odd elements $X$ satisfying $[X,X]=0$
in $W_1(\Pi V), W_1^\partial(\Pi V)$ and $W_1^\ch(\Pi V)$
and the Lie superalgebra, Lie conformal superalgbera and vertex algebra structures on $V$, respectively.
As a result we constructed the corresponding  cohomology complexes with coefficients in the adjoint module.
This construction can be easily generalized to the case of an arbitrary $V$-module $M$,
by considering the extention of the Lie superalgebra, Lie conformal superalgebra and vertex algebra 
structure of $V$ to $V\oplus M$, by making $M$ to be an abelian ideal.
Then the natural reduction of the corresponding complexes produces
the cohomology complex of $V$ with coefficients in $M$.
See e.g. \cite{DSK13}.
\end{remark}

\begin{remark}\label{rem:classical}
An important role is also played by the classical operad ${\mathscr P}^{\cl}(\Pi V)$,
governing Poisson vertex algebra structures on $V$ and the corresponding classical PVA cohomology.
Such operad can be obtained as the associated graded of a filtered chiral operad.
The space ${\mathscr P}^{\cl}(n)$ consists of maps
$$
f_{\lambda_1,\dots,\lambda_n}\colon \mc G(n)\times (\Pi V)^{\otimes n}\to
\Pi V[\lambda_1,\dots,\lambda_n]\big/\big\langle\partial+\lambda_1+\dots+\lambda_n\big\rangle
\,,
$$
satisfying certain conditions,
where $\mc G(n)$ is the set of oriented graphs with $n$ vertices \cite{BDSHK19}.
Evaluating on the graph without edges produces the corresponding variational Poisson cohomology.
This restriction induces an isomorphism in cohomology,
provided that $V$, as differential algebra, is an algebra of differential polynomials \cite{BDSHKV21}.
\end{remark}

\section{Conformal operads and conformal pseudo operads}\label{sec:3}

\subsection{Definition of conformal operad and pseudo operad}

\begin{definition}\label{def:conf-operad}
A \emph{conformal} (symmetric, super) \emph{operad} $\widetilde{{\mathscr P}}$ 
is a collection of vector superspaces $\widetilde{{\mathscr P}}(n)$, $n\geq0$, with parity $p$,
endowed with an even endomorphism $\partial:\,\widetilde{{\mathscr P}}(n)\to\widetilde{{\mathscr P}}(n)$.
The collection $\widetilde{{\mathscr P}}$ is endowed
with parity preserving linear maps, called \emph{composition maps},
associated to integers $n\geq1, m_1,\dots,m_n\geq0$,
\begin{equation}\label{eq:conf-comp}
\begin{split}
\widetilde{{\mathscr P}}(n) \otimes \widetilde{{\mathscr P}}(m_1)\otimes\dots\otimes\widetilde{{\mathscr P}}(m_n)\,\,&
\to\,\,\widetilde{{\mathscr P}}(M_n)[\lambda_1,\dots,\lambda_n] \,, \\
f \otimes g_1 \otimes\dots\otimes g_n \,\, &\mapsto \,\, 
f_{\lambda_1,\dots,\lambda_n}(g_1\otimes\dots\otimes g_n) \,,
\end{split}
\end{equation}
satisfying the \emph{sesquilinearity conditions} ($f\in\widetilde{{\mathscr P}}(n)$, $g_i\in\widetilde{{\mathscr P}}(m_i)$ for $i=1,\dots,n$)
\begin{equation}\label{20170613:eq1-conf}
\begin{split}
& (\partial f)_{\lambda_1,\dots,\lambda_n}(g_1\otimes\,\cdots\,\otimes g_n)
=
(\partial+\lambda_1+\dots+\lambda_n)f_{\lambda_1,\dots,\lambda_n}(g_1\otimes\,\cdots\,\otimes g_n)
\,,\\
& f_{\lambda_1,\dots,\lambda_n}(g_1\otimes\,\cdots\partial g_i\cdots\,\otimes g_n)
=-\lambda_if_{\lambda_1,\dots,\lambda_n}(g_1\otimes\,\cdots\,\otimes g_n)
\,\,\text{ for } i=1,\dots,n
\,,
\end{split}
\end{equation}
and the \emph{associativity axiom} ($f\in\widetilde{{\mathscr P}}(n)$, $g_i\in\widetilde{{\mathscr P}}(m_i)$ for $i=1,\dots,n$,
and $h_{j}\in\widetilde{{\mathscr P}}(\ell_{j})$ for $j=1,\dots,M_n$)
\begin{equation}\label{eq:operad2-conf}
\begin{split}
& \big(f_{\lambda_1,\dots,\lambda_n}
(g_1\otimes\dots\otimes g_n)
\big)_{\mu_1,\dots,\mu_{M_n}}
(h_1\otimes\dots\otimes h_{M_n}) \\
& =
\pm
f_{\lambda_1+\Gamma_1,\dots,\lambda_n+\Gamma_n}\big(
(g_1)_{\mu_1,\dots,\mu_{m_1}}(H_1)
\otimes\dots\otimes 
(g_n)_{\mu_{M_{n-1}+1},\dots,\mu_{M_n}}(H_n)
\big)
\,,
\end{split}
\end{equation}
where
\begin{equation}\label{eq:Mj-Gammaj}
M_j:=\sum_{k=1}^jm_k
\,\,,\,\,\,\,
\Gamma_i=\sum_{j=M_{i-1}+1}^{M_i}\mu_j
\,\,,\,\,\,\,
H_i=h_{M_{i-1}+1}\otimes\dots\otimes h_{M_i}
\,,\qquad
i=1,\dots,n
\,,
\end{equation}
and where, as in \eqref{20170804:eq1}, the sign is $\pm=(-1)^{\sum_{i<j}p(g_j)p(H_i)}$.
Equation \eqref{eq:operad2-conf} is an identity in 
$\widetilde{{\mathscr P}}\Big(\sum_{j=1}^{M_n}\ell_j\Big)
[\lambda_1,\dots,\lambda_n,\mu_1,\dots\mu_{M_n}]$.
Furthermore, it is assumed that,
for each $n\geq1$, $\widetilde{{\mathscr P}}(n)$ has a right action of the symmetric group $S_n$,
which commutes with $\partial$,
denoted $f^\sigma$, for $f\in\widetilde{{\mathscr P}}(n)$ and $\sigma\in S_n$,
satisfying the following \emph{equivariance axiom}
($f\in\widetilde{{\mathscr P}}(n)$, $g_1\in\widetilde{{\mathscr P}}(m_1),\dots,g_n\in\widetilde{{\mathscr P}}(m_n)$,
$\sigma\in S_n$, $\tau_1\in S_{m_1},\dots,\tau_n\in S_{m_n}$):
\begin{equation}\label{eq:operad4-conf}
(f^\sigma)_{\lambda_1,\dots,\lambda_n}(g_1^{\tau_1}\otimes \dots\otimes g_n^{\tau_n})
=
\big(f_{\lambda_{\sigma^{-1}(1)},\dots,\lambda_{\sigma^{-1}(n)}}(\sigma(g_1\otimes\dots\otimes g_n))\big)^{\sigma(\tau_1,\dots,\tau_n)}
\,,
\end{equation}
where the composition $\sigma(\tau_1,\dots,\tau_n)\in S_{M_n}$ 
was defined by \eqref{eq:operad19}.
\end{definition}

\begin{definition}
A conformal operad $\widetilde{{\mathscr P}}$ is called \emph{unital}
if there is a \emph{unit} element $1\in\widetilde{{\mathscr P}}(1)/\partial\widetilde{{\mathscr P}}(1)$
such that, for any representative $\widetilde{1}\in\widetilde{{\mathscr P}}(1)$, we have
\begin{equation}\label{eq:1-conf}
\widetilde{1}_{-\partial}(f)=f
\,,\,\,
f_{0,\dots,0}(\widetilde{1}\otimes\dots\otimes\widetilde{1})=f
\,\,,\qquad f\in\widetilde{{\mathscr P}}(n)
\,.
\end{equation}
\end{definition}
Note that, by the sesquilinearity conditions \eqref{20170613:eq1-conf},
neither of the above expressions depend on the choice of the representative $\widetilde{1}\in\widetilde{{\mathscr P}}(1)$.

\begin{definition}\label{def:conf-pseudo-operad}
A \emph{conformal pseudo operad}
is a collection of vector superspaces $\widetilde{{\mathscr P}}(n)$, $n\geq0$, with parity $p$,
endowed with an even endomorphism $\partial:\,\widetilde{{\mathscr P}}(n)\to\widetilde{{\mathscr P}}(n)$,
with parity preserving $\circ^i_\lambda$-products, 
$\widetilde{{\mathscr P}}(n)\times\widetilde{{\mathscr P}}(m)\to\widetilde{{\mathscr P}}(n+m-1)[\lambda]$, for $i=1,\dots,n$,
satisfying the sesquilinearity conditions
\begin{equation}\label{eq:circ-sesq}
(\partial f)\circ^i_\lambda g
=
(\lambda+\partial)(f\circ^i_\lambda g)
\,\,,\qquad
f\circ^i_\lambda(\partial g)
=
-\lambda f\circ^i_\lambda g
\,,
\end{equation}
and the associativity conditions ($f\in\widetilde{{\mathscr P}}(n), g\in\widetilde{{\mathscr P}}(m), h\in\widetilde{{\mathscr P}}(\ell)$)
\begin{align}
(f\circ^i_{\lambda} g)\circ^j_{\mu} h
& =
(-1)^{p(g)p(h)}
(f\circ^j_{\mu} h)\circ^{\ell+i-1}_{\lambda} g  \qquad \text{ if } 1\leq j<i 
\,,\label{eq:circ-assoc1-conf} \\
(f\circ^i_{\lambda} g)\circ^j_{\mu} h
& =
f\circ^i_{\lambda+\mu}(g\circ^{j-i+1}_{\mu} h) \qquad \text{ if } i\leq j<i+m 
\,.\label{eq:circ-assoc2-conf}
\end{align}
By exchanging $\lambda$ with $\mu$ and $g$ with $h$,
and by replacing $j-m+1$ by $i$ and $i$ by $j$,
condition \eqref{eq:circ-assoc1-conf} can be equivalently written as
\begin{equation}\label{eq:circ-assoc3-conf}
(f\circ^i_{\lambda} g)\circ^j_{\mu} h
=
(-1)^{p(g)p(h)}
(f\circ^{j-m+1}_{\mu} h)\circ^i_{\lambda} g \qquad \text{ if } i+m\leq j<n+m
\,.
\end{equation}
Furthermore, we require that each $\widetilde{{\mathscr P}}(n)$ 
is endowed with a right action of the symmetric group $S_n$
commuting with $\partial$ and
satisfying the following equivariance condition
($f\in\widetilde{{\mathscr P}}(n)$, $g\in\widetilde{{\mathscr P}}(m)$, $\sigma\in S_n$, $\tau\in S_{m}$):
\begin{equation}\label{eq:circ-equiv-conf}
f^\sigma\circ^i_{\lambda}g^{\tau}
=
\big(f\circ^{\sigma(i)}_{\lambda}g\big)^{\sigma\circ_i\tau}
\,,
\end{equation}
where the $\circ_i$-product $\sigma\circ_i\tau\in S_{m+n-1}$ of permutations was defined by \eqref{eq:operad19b}.
\end{definition}

\begin{proposition}\label{prop:conf-pseudo}
\begin{enumerate}[(a)]
\item
A unital conformal operad $\widetilde{{\mathscr P}}$ is automatically a conformal pseudo operad
with the $\circ^i_{\lambda}$-products 
$\circ^i_\lambda\colon\widetilde{{\mathscr P}}(n)\times\widetilde{{\mathscr P}}(m)\to\widetilde{{\mathscr P}}(n+m-1)[\lambda]$
defined, for each $i=1,\dots,n$, by
\begin{equation}\label{eq:operad8-conf}
f\circ^i_\lambda g
=
f_{0,\dots\stackrel{i}{\widecheck{\lambda}}\dots,0}(\widetilde{1}\otimes\dots 
\stackrel{i}{\widecheck{\vphantom{\big(}g}}\dots\otimes\widetilde{1})
\,,
\end{equation}
where $\stackrel{i}{\widecheck{}}$ means insertion in position $i$.
\item
A pseudo conformal operad is naturally a conformal operad with the composition maps 
\begin{equation}\label{eq:operad8a-conf}
f_{\lambda_1,\dots,\lambda_n}(g_1,\dots,g_n)
=
(\dots((f\circ^1_{\lambda_1} g_1)\circ^{m_1+1}_{\lambda_2}g_2)\dots)\circ^{M_{n-1}+1}_{\lambda_n} g_n
\,.
\end{equation}
\end{enumerate}
\end{proposition}
\begin{proof}
The sesquilinearity axioms \eqref{20170613:eq1-conf}
translates, in terms of $\circ^i_\lambda$-products \eqref{eq:operad8-conf}, as \eqref{eq:circ-sesq}.
The associativity axiom \eqref{eq:operad2-conf} translates as \eqref{eq:circ-assoc1-conf}-\eqref{eq:circ-assoc2-conf}.
Indeed, by \eqref{eq:operad2-conf}, both sides of \eqref{eq:circ-assoc1-conf} are equal to
$$
(-1)^{p(g)p(h)}
f_{0,\dots\stackrel{j}{\widecheck{\mu}}\dots\stackrel{i}{\widecheck{\lambda}}\dots,0}
(\widetilde{1}\otimes\dots 
\stackrel{j}{\widecheck{\vphantom{\big(}h}}\dots
\stackrel{i}{\widecheck{\vphantom{\big(}g}}\dots\otimes\widetilde{1})
\,,
$$
and both sides of \eqref{eq:circ-assoc2-conf} are equal to
$$
f_{0,\dots\stackrel{i}{\widecheck{\lambda+\mu}}\dots,0}
(\widetilde{1}\otimes\dots 
\stackrel{i}{\widecheck{\vphantom{\big(}
g_{0,\dots\stackrel{j-i+1}{\widecheck{\mu}}\dots,0}
}}
(\widetilde{1}\otimes\dots 
\stackrel{j-i+1}{\widecheck{\vphantom{\big(}h}}\dots\otimes\widetilde{1})
\dots\otimes\widetilde{1})
\,.
$$
Furthermore, the equivariance axiom \eqref{eq:operad4-conf} of the unital conformal operad $\widetilde{{\mathscr P}}$
is translated to the equivariance \eqref{eq:circ-equiv-conf} for the $\circ^i_\lambda$-products.
This proves claim (a).

Conversely, assume $\widetilde{{\mathscr P}}$ is a conformal pseudo operad
and define the composition maps by \eqref{eq:operad8a-conf}.
The sesquilinearity \eqref{eq:circ-sesq},
the associativity \eqref{eq:circ-assoc1-conf}-\eqref{eq:circ-assoc2-conf}, and the equivariance \eqref{eq:circ-equiv-conf}
of the $\circ^i_{\lambda}$ products 
translate to the sesquilinearity \eqref{20170613:eq1-conf},
the associativity \eqref{eq:operad2-conf} and the equivariance \eqref{eq:operad4-conf}
of the composition maps \eqref{eq:operad8a-conf}.
\end{proof}

\subsection{$\mb Z_{\geq-1}$-graded Lie conformal superalgebra
associated to a conformal pseudo operad and basic cohomology complexes}\label{sec:lie-conf}

Given a conformal pseudo operad $\widetilde{{\mathscr P}}$,
consider the $\mb Z_{\geq-1}$-graded vector superspace
$\widetilde{W}(\widetilde{{\mathscr P}})=\bigoplus_{n\geq-1}\widetilde{W}_n$,
where
\begin{equation}\label{20170603:eq3}
\widetilde{W}_n=\widetilde{{\mathscr P}}(n+1)^{S_{n+1}}
=\big\{f\in\widetilde{{\mathscr P}}(n+1)\,\big|\,f^\sigma=f\,\forall \sigma\in S_{n+1}\big\}
\,.
\end{equation}

One defines a Lie conformal superalgebra structure on $\widetilde{W}$ as follows.
We define the $\Box_\lambda$-product of $f\in\widetilde{{\mathscr P}}(n+1)$ and $g\in\widetilde{{\mathscr P}}(m+1)$ by
\begin{equation}\label{eq:box-conf}
f\Box_\lambda g
=
\sum_{\sigma\in S_{m+1,n}}
(f\circ^1_{-\lambda-\partial} g)^{\sigma^{-1}}
\,\in\widetilde{{\mathscr P}}(m+n+1)[\lambda]\,,
\end{equation}
where, as in \eqref{eq:box}, $S_{m,n}\subset S_{m+n}$ denotes the subset of $(m,n)$-shuffles
and, as usual, in $f\circ^1_{-\lambda-\partial} g$, $\partial$ is moved to the left and applied to the coefficients.
\begin{theorem}\label{20170603:thm2-conf}
\begin{enumerate}[(a)]
\item
The $\Box_\lambda$-product \eqref{eq:box-conf} satisfies the following
sesquilinearity conditions:
$$
(\partial f)\Box_\lambda g=-\lambda f\Box_\lambda g
\,\,,\,\,\,\,
f\Box_\lambda(\partial g)=(\lambda+\partial) (f\Box_\lambda g)
\,.
$$
\item
Define the associator of the $\Box_\lambda$-product as
\begin{equation}\label{associator}
(f_\lambda g_\mu h)
=
(f\Box_\lambda g)\Box_{\lambda+\mu}h
-f\Box_\lambda (g\Box_\mu h)
\,.
\end{equation}
For $f\in\widetilde{W}_n$ we have the following symmetry relation:
\begin{equation}\label{20170608:eq2-conf}
(f_\lambda g_\mu h)
=
(-1)^{p(g)p(h)}
(f_\lambda h_{-\lambda-\mu-\partial} g)
\,,
\end{equation}
where $\partial$ is moved to the left and acts on all coefficients.
\item
For $f\in\widetilde{W}_n$ and $g\in\widetilde{W}_m$, 
we have $f\Box_\lambda g\in\widetilde{W}_{n+m}[\lambda]$.
\item
Consequently, $\widetilde{W}$ is a $\mb Z_{\geq-1}$-graded Lie conformal superalgebra
with $\lambda$-bracket given by ($f\in\widetilde{W}_n$, $g\in\widetilde{W}_m$)
\begin{equation}\label{20170603:eq4-conf}
[f_\lambda g]=f\Box_\lambda g-(-1)^{p(f)p(g)}g\Box_{-\lambda-\partial} f
\,.
\end{equation}
\end{enumerate}
\end{theorem}
\begin{proof}
Claim (a) follows immediately by the definition \eqref{eq:box-conf} of the $\Box_\lambda$-product
and by the sesquilinearity axioms \eqref{eq:circ-sesq} of the  $\circ^i_\lambda$-products.
Next, we prove claim (b), following closely the Proof of \cite[Thm.3.4]{BDSHK19}.
Let $f\in\widetilde{W}_n$, $g\in\widetilde{P}(m+1)$ and $h\in\widetilde{P}(\ell+1)$.
By the definition of the $\Box_\lambda$-product and the same computations 
as in \cite[Eq.(3.18)]{BDSHK19}, we have
\begin{equation}\label{eq:assoc1}
\begin{split}
f\Box_\lambda (g\Box_\mu h)
& =
\sum_{\sigma\in S_{m+\ell+1,n}}
\sum_{\tau\in S_{\ell+1,m}}
\big(
f\circ^1_{-\lambda-\partial} 
(g\circ^1_{-\mu-\partial}h)^{\sigma^{-1}}
\big)^{\tau^{-1}} \\
& =
\sum_{\sigma\in S_{\ell+1,m,n}}
\big(
f\circ^1_{-\lambda-\partial} 
(g\circ^1_{-\mu-\partial}h)
\big)^{\sigma^{-1}} 
\,,
\end{split}
\end{equation}
where $S_{\ell,m,n}\subset S_{\ell+m+n}$ denotes the subset of $(\ell,m,n)$-suffles,
i.e. permutations satisfying $\sigma(1)<\dots<\sigma(\ell)$, $\sigma(\ell+1)<\dots<\sigma(\ell+m)$
and $\sigma(\ell+m+1)<\dots<\sigma(\ell+m+n)$.
On the other hand, we have
\begin{equation}\label{eq:assoc2}
(f\Box_\lambda g)\Box_{\lambda+\mu} h
=
\sum_{\sigma\in S_{m+1,n}}
\sum_{\tau\in S_{\ell+1,m+n}}
\big(
(f\circ^1_{-\lambda-\partial} g)^{\sigma^{-1}}
\circ^1_{-\lambda-\mu-\partial}h
\big)^{\tau^{-1}}
\,.
\end{equation}
The sum over $\sigma$ in the RHS splits in two:
the sum over the shuffles $\sigma\in S_{m+1,n}$ such that $\sigma(1)=1$
and the sum over the shuffles such that $\sigma(m+2)=1$.
By the same computations as in \cite[Eq.(3.20)]{BDSHK19},
the first contribution gives
\begin{equation}\label{eq:assoc3}
\begin{split}
& \sum_{\substack{\sigma\in S_{m+1,n} \\ \text{s.t. } \sigma(1)=1}}
\sum_{\tau\in S_{\ell+1,m+n}}
\big(
(f\circ^1_{-\lambda-\partial} g)^{\sigma^{-1}}
\circ^1_{-\lambda-\mu-\partial}h
\big)^{\tau^{-1}} \\
& =
\sum_{\sigma\in S_{\ell+1,m,n}}
\big(
(f\circ^1_{-\lambda-\partial} g)
\circ^1_{-\lambda-\mu-\partial}h
\big)^{\sigma^{-1}} 
\,.
\end{split}
\end{equation}
Using the sesquilinearity \eqref{eq:circ-sesq} and the associativity axiom \eqref{eq:circ-assoc2-conf} 
(with $i=j=1$) for the $\circ^1_\lambda$-product,
it is not hard to check that the RHS of \eqref{eq:assoc1} coincides with the RHS of \eqref{eq:assoc3},
so they cancel in the expression of the associator \eqref{associator}.
Hence,
\begin{equation}\label{eq:assoc4a}
(f_\lambda g_\mu h)
=
\sum_{\substack{\sigma\in S_{m+1,n} \\ \text{s.t. } \sigma(m+2)=1}}
\sum_{\tau\in S_{\ell+1,m+n}}
\big(
(f\circ^1_{-\lambda-\partial} g)^{\sigma^{-1}}
\circ^1_{-\lambda-\mu-\partial}h
\big)^{\tau^{-1}} 
\end{equation}
Using the sesquilinearity conditions \eqref{eq:circ-sesq}
and following the same line of computations as in \cite[Eq.(3.21)-(3.22)]{BDSHK19}, 
the above identity can be rewritten as
\begin{equation}\label{eq:assoc4}
(f_\lambda g_\mu h)
=
\sum_{\sigma\in S_{m+1,\ell+1,n-1}}
\big(
(f\circ^1_{\mu} g) \circ^{m+2}_{-\lambda-\mu-\partial}h
\big)^{\sigma^{-1}} 
\,.
\end{equation}
By assumption $f$ lies in $\widetilde{W}_n$, so $f=f^{(12)}$.
By applying the equivariance axiom \eqref{eq:circ-equiv-conf} twice, we get
\begin{align*}
& (f\circ^1_{\mu} g) \circ^{m+2}_{-\lambda-\mu-\partial}h
=
(f^{(12)}\circ^1_{\mu} g) \circ^{m+2}_{-\lambda-\mu-\partial}h \\
& =
(f\circ^2_{\mu} g)^{(12)\circ_11_{m+1}} \circ^{m+2}_{-\lambda-\mu-\partial}h
=
\big(
(f\circ^2_{\mu} g) \circ^{1}_{-\lambda-\mu-\partial}h
\big)^{((12)\circ_11_{m+1})\circ_{m+2}1_{\ell+1}}
\,.
\end{align*}
By the definition of $\circ_i$-products in the permutation groups, we have (cf. \cite[Eq.(2.16)]{BDSHK19})
$$
((12)\circ_11_{m+1})\circ_{m+2}1_{\ell+1}
=
(12)(1_{m+1},1_{\ell+1},1_{n-1})
\,.
$$
Hence, \eqref{eq:assoc4} gives
$$
(f_\lambda g_\mu h)
=
\sum_{\sigma\in S_{m+1,\ell+1,n-1}}
\big(
\big(
(f\circ^2_{\mu} g) \circ^{1}_{-\lambda-\mu-\partial}h
\big)^{\big(\sigma\cdot
((12)(1_{m+1},1_{\ell+1},1_{n-1}))^{-1}
\big)^{-1}}
\,.
$$
By \cite[Prop.2.5(b)]{BDSHK19},
we have a bijective correspondence $S_{m+1,\ell+1,n-1}\stackrel{\sim}{\rightarrow}S_{\ell+1,m+1,n-1}$
mapping $\sigma\mapsto\sigma\cdot((12)(1_{m+1},1_{\ell+1},1_{n-1}))^{-1}$.
As a consequence,
$$
(f_\lambda g_\mu h)
=
\sum_{\sigma\in S_{\ell+1,m+1,n-1}}
\big(
(f\circ^2_{\mu} g) \circ^{1}_{-\lambda-\mu-\partial}h
\big)^{\sigma^{-1}}
\,.
$$
We next apply the sesquilinearity axioms \eqref{eq:circ-sesq}
and the first associativity condition \eqref{eq:circ-assoc1-conf}
to get
$$
(f\circ^2_{\mu} g) \circ^{1}_{-\lambda-\mu-\partial}h
=
(-1)^{p(g)p(h)}
(f \circ^{1}_{-\lambda-\partial} h) \circ^{\ell+2}_{\mu} g
\,.
$$
Hence,
\begin{equation}\label{eq:assoc5}
(f_\lambda g_\mu h)
=
(-1)^{p(g)p(h)}
\sum_{\sigma\in S_{\ell+1,m+1,n-1}}
\big(
(f \circ^{1}_{-\lambda-\partial} h) \circ^{\ell+2}_{\mu} g
\big)^{\sigma^{-1}}
\,.
\end{equation}
Comparing \eqref{eq:assoc4} and \eqref{eq:assoc5}
and using again the sesquilinearity conditions \eqref{eq:circ-sesq},
we get \eqref{20170608:eq2-conf}, proving (b).
Claim (c) can be proved in the same way as \cite[Thm.3.4(a)]{BDSHK19}.
Claim (d) is an obvious consequence of (a), (b) and (c).
\end{proof}
\begin{remark}\label{rem:reimundo}
Notice that the identity \eqref{20170608:eq2-conf}
is an identity satisfied by compositions only,
so it is a formal consequence of identities \eqref{eq:circ-sesq}-\eqref{eq:circ-equiv-conf}
(see also \eqref{eq:circ-assoc} in Section \ref{sec:operads}).
\end{remark}

\subsection{Relation between linear and conformal (pseudo) operads}

\begin{proposition}\label{prop:lin-conf}
\begin{enumerate}[(a)]
\item
Given a conformal operad $\widetilde{{\mathscr P}}$,
we have a natural structure of a linear operad on the collection of vector superspaces ${\mathscr P}=({\mathscr P}(n))_{n\geq0}$,
where ${\mathscr P}(n)=\widetilde{{\mathscr P}}(n)/\partial\widetilde{{\mathscr P}}(n)$,
where the compositions maps 
${\mathscr P}(n) \otimes {\mathscr P}(m_1)\otimes\dots\otimes{\mathscr P}(m_n)\,\to\,{\mathscr P}(M_n)$
are defined by
\begin{equation}\label{eq:operad1-quot}
\bar f(\bar g_1\otimes\dots\otimes \bar g_n)
=
f_{0,\dots,0}(g_1\otimes\dots\otimes g_n) 
\,,
\end{equation}
for $\bar f\in {\mathscr P}(n)$, $\bar g_i\in{\mathscr P}(m_i)$, $i=1,\dots,n$,
where $f\in\widetilde{{\mathscr P}}(n)$, $g_i\in\widetilde{{\mathscr P}}(m_i)$, $i=1,\dots,n$,
are any their representatives,
and where the (right) action of $S_n$ on ${\mathscr P}(n)$
is induced by its action on $\widetilde{{\mathscr P}}(n)$.
Furthermore, if $\widetilde{{\mathscr P}}$ is a unital conformal operad
then ${\mathscr P}$ is a unital operad, with the same unit element $1\in{\mathscr P}(1)=\widetilde{{\mathscr P}}(1)/\partial\widetilde{{\mathscr P}}(1)$.
\item
Likewise, given a conformal pseudo operad $\widetilde{{\mathscr P}}$,
we have a natural structure of a linear pseudo operad on the collection 
${\mathscr P}=(\widetilde{{\mathscr P}}(n)/\partial\widetilde{{\mathscr P}}(n))_{n\geq0}$,
where the $\circ_i$-products 
${\mathscr P}(n) \otimes {\mathscr P}(m)\,\to\,{\mathscr P}(m+n-1)$
are defined by
\begin{equation}\label{eq:operad1-quotb}
\bar f\circ_i\bar g
=
f\circ^i_0 g 
\,,
\end{equation}
for $\bar f\in {\mathscr P}(n)$, $g\in{\mathscr P}(m)$,
where $f\in\widetilde{{\mathscr P}}(n)$, $g\in\widetilde{{\mathscr P}}(m)$,
are any their representatives,
and where the (right) action of $S_n$ on ${\mathscr P}(n)$
is induced by its action on $\widetilde{{\mathscr P}}(n)$.
\item
If $\widetilde{{\mathscr P}}$ is a conformal pseudo operad,
there is a canonical surjective map of $\mb Z_{\geq-1}$-graded Lie superalgebras 
\begin{equation}\label{eq:quot}
\widetilde W(\widetilde{{\mathscr P}})/\partial\widetilde W(\widetilde{{\mathscr P}})
\twoheadrightarrow
W({\mathscr P})
\,,
\end{equation}
which restricts to an isomorphism $\widetilde{W}_n/\partial\widetilde{W}_n\simeq W_n$ on the degree $n$,
provided that $\partial$ has zero kernel on $\widetilde{{\mathscr P}}(n)$, or that $n=0$.
Moreover, we have a compatible (with the map \eqref{eq:quot}) Lie superalgebra 
action of $W({\mathscr P})$ on $\widetilde W(\widetilde{{\mathscr P}})$,
preserving the gradings, given by
\begin{equation}\label{eq:W-action}
[\bar f,g]
:=
\tilde{f}\Box_0 g-(-1)^{p(\tilde{f})p(g)}g\Box_{-\partial}\tilde{f}
\in\widetilde{W}_{m+n}
\,,
\end{equation}
for $\bar f\in W_m$, $g\in \widetilde{W}_n$
and any $\tilde{f}\in \widetilde{W}_m$ representative of a preimage of $\bar f$ via the map \eqref{eq:quot}.
\end{enumerate}
\end{proposition}
\begin{proof}
Assume that $\widetilde{{\mathscr P}}$ is a conformal operad.
The sesquilinearity conditions \eqref{20170613:eq1-conf}
guarantee that the right-hand side of \eqref{eq:operad1-quot}
does not depend on the choice of representatives,
hence the composition maps are well defined.
Moreover, the associativity axiom \eqref{eq:operad2}
follows by \eqref{eq:operad2-conf} setting all $\lambda_i$'s and $\mu_j$'s equal to $0$.
By assumption, the action of $S_n$ on $\widetilde{{\mathscr P}}(n)$ commutes with the action of $\partial$,
hence it induces a well defined action on $\widetilde{{\mathscr P}}(n)/\partial\widetilde{{\mathscr P}}(n)={\mathscr P}(n)$.
Setting all $\lambda_i$'s equal to $0$ in \eqref{eq:operad4-conf} we obtain \eqref{eq:operad4}
Moreover, the unit element $1\in{\mathscr P}(1)$ satisfies the unity axiom \eqref{eq:operad3}.
This proves claim (a).
The proof of claim (b) is similar.

Let us prove claim (c). 
We have the canonical maps
\begin{align*}
& \widetilde{W}_{n-1}/\partial\widetilde{W}_{n-1}
\simeq
\widetilde{{\mathscr P}}(n)^{S_{n}}/\partial\widetilde{{\mathscr P}}(n)^{S_{n}}
\stackrel{\alpha}{\longrightarrow}
\widetilde{{\mathscr P}}(n)^{S_{n}}\big/\big(\partial\widetilde{{\mathscr P}}(n)\cap\widetilde{{\mathscr P}}(n)^{S_{n}}\big) \\
& \stackrel{\beta}{\longrightarrow}
\Big(\widetilde{{\mathscr P}}(n)/\partial\widetilde{{\mathscr P}}(n)\Big)^{S_{n}}
\simeq
{\mathscr P}(n)^{S_{n}}=W_{n-1}\,.
\end{align*}
The map $\alpha$ is surjective since, obviously, 
$\partial\widetilde{{\mathscr P}}(n)^{S_{n}}
\subset\partial\widetilde{{\mathscr P}}(n)\cap\widetilde{{\mathscr P}}(n)^{S_{n}}$.
Moreover, if $\partial$ has zero kernel on $\widetilde{{\mathscr P}}(n)$,
then $\partial\widetilde{{\mathscr P}}(n)^{S_{n}}
=\partial\widetilde{{\mathscr P}}(n)\cap\widetilde{{\mathscr P}}(n)^{S_{n}}$,
so $\alpha$ is bijective.
The quotient map 
$\widetilde{{\mathscr P}}(n)\to\widetilde{{\mathscr P}}(n)/\partial\widetilde{{\mathscr P}}(n)$
restricts to
$\widetilde{{\mathscr P}}(n)^{S_{n}}\to\Big(\widetilde{{\mathscr P}}(n)/\partial\widetilde{{\mathscr P}}(n)\Big)^{S_{n}}$,
which has kernel $\partial\widetilde{{\mathscr P}}(n)\cap\widetilde{{\mathscr P}}(n)^{S_{n}}$,
so $\beta$ is clearly injective.
We claim that $\beta$ is surjective as well.
Indeed, let $U\subset\widetilde{{\mathscr P}}(n)$ be an $S_{n}$-submodule 
complementary to $\partial\widetilde{{\mathscr P}}(n)$.
For $\bar f\in\widetilde{{\mathscr P}}(n)/\partial\widetilde{{\mathscr P}}(n)$
let $u\in U\subset\widetilde{{\mathscr P}}(n)$ be its unique representative in $U$.
If $\bar f$ is fixed by the action of $S_n$, then so is $u$,
so that $\beta(\bar u)=\bar f$. Hence, $\beta$ is surjective, as claimed.
The fact that \eqref{eq:quot} is a Lie superalgebra homomoprhism
is an immediate consequence of the definition of the $\circ_i$-products \eqref{eq:operad1-quotb}
of ${\mathscr P}=\widetilde{{\mathscr P}}/\partial\widetilde{{\mathscr P}}$,
comparing \eqref{eq:box} with \eqref{eq:box-conf}.
We are left to prove the last assertion, 
that \eqref{eq:W-action} defines an action 
of the Lie superalgebra $W({\mathscr P})$ on $\widetilde{W}(\widetilde{{\mathscr P}})$, preserving the gradings.
By the above arguments, the kernel of the map \eqref{eq:quot} coincides with the kernel 
of the map $\alpha$ above. Hence another representative ${\widetilde f}_1$ 
of a preimage of $\bar f$ differs from $\widetilde f$ by an element 
$r\in\widetilde{{\mathscr P}}(n+1)^{S_{n+1}}\cap\partial\widetilde{{\mathscr P}}(n+1)$.
Then, since $r$ is in the image of $\partial$,
by the sesquilinearity of the $\Box_\lambda$ product in Theorem \ref{20170603:thm2-conf}(a),
we have that $[r,g]=0$.
Hence, the right-hand side of \eqref{eq:W-action}
does not depend on the choice of the representative $\widetilde f$
and the action $[\bar f,g]$ is well defined.
Moreover, consider, as above, an $S_{m+1}$-submodule $U\subset\widetilde{{\mathscr P}}(m+1)$
complementary to $\partial\widetilde{{\mathscr P}}(m+1)$.
There is a unique $u\in U$ representative of a preimage of $\bar f$,
and since $\bar f$ is fixed by the action of $S_n$, 
we have $u\in\widetilde{{\mathscr P}}(m+1)^{S_m+1}=\widetilde{W}_m$.
Hence, $[\bar f,g]=[u_\lambda g]\big|_{\lambda=0}$ which,
by Theorem \ref{20170603:thm2-conf}(c), lies in $\widetilde{W}_{m+n}$, as wanted.
The fact that \eqref{eq:W-action} defines an action of the Lie superalgebra $W({\mathscr P})$
on $\widetilde{W}(\widetilde{{\mathscr P}})$ follows from \eqref{20170608:eq2-conf},
the proof being similar to the proof of Theorem \ref{20170603:thm2-conf}(d).
\end{proof}

Note that, if $X\in W_1$ is odd then the condition that $X\Box X=0$
implies that $[X,[X,\,\cdot]]=0$ on $\widetilde W(\widetilde{{\mathscr P}})$.
\begin{definition}\label{def:univ-lie-conf}
Let $\widetilde{{\mathscr P}}$ be a conformal pseudo operad.
\begin{enumerate}[(a)]
\item
The $\mb Z_{\geq-1}$-\emph{graded Lie conformal superalgebra} 
associated to $\widetilde{{\mathscr P}}$ 
is the $\mb Z_{\geq-1}$-graded superspace $\widetilde{W}(\widetilde{{\mathscr P}})
=\bigoplus_{n\geq-1}\widetilde{W}_n$  with $\lambda$-bracket defined by \eqref{20170603:eq4-conf}.
\item
The \emph{basic cohomology complex}, associated to $\widetilde{{\mathscr P}}$
and an odd element $X\in W_1$ such that $X\Box X=0$,
is $(\widetilde{W}(\widetilde{{\mathscr P}}),[X,\,\cdot])$.
\end{enumerate}
\end{definition}

\section{Pro-conformal pseudo operads}\label{sec:4}

In order to include, in the following Section \ref{sec:basic-chom-chiral},
the basic Chom and chiral operads over an $\mb F[\partial]$-module $V$ of infinite rank,
we need a generalization of the notion of conformal pseudo operad,
which allows formal power series in $\lambda$.
This leads to the notion of a pro-conformal pseudo operad $\widetilde{{\mathscr P}}$.
The corresponding $\mb Z_{\geq-1}$-graded space $\widetilde{W}(\widetilde{{\mathscr P}})$
of invariants with respect to the action of the symmetric groups
will then have the structure of a pro-Lie conformal superalgebra.

\subsection{Projective systems}\label{sec:proj-system}

Recall that a \emph{projective system} in a category
is a collection of objects $(V_q)_{q\geq0}$,
endowed with morphisms 
$V_0\stackrel{\pi_{1,0}}{\leftarrow} V_1\stackrel{\pi_{2,1}}{\leftarrow} V_2 \leftarrow\dots$.
The \emph{projective limit} $V=\varprojlim V_q$, if it exists, is an object endowed with
a collection of morphisms $\pi_q:\,V\to V_q$, $q\geq0$,
such that the following diagram is commutative:
$$
\xymatrixrowsep{1pc}
\xymatrix{
& V \ar[d]_{\pi_{q}} \ar[rd]^{\pi_{q+1}} \\
\dots & \ar[l] V_q & \ar[l]^{\pi_{q+1,q}} V_{q+1} & \ar[l] \dots
}
$$
and it satisfies the universal property associated to this diagram.

We will be interested, in particular, in two categories.
First, the category $\mb F[\partial]$-mod 
of vector superspaces $V$ with an even endomorphism $\partial\in\End V$,
where the morphisms are parity preserving linear maps commuting with $\partial$.

Then, the category $\mb F[\partial]$-mod-$S_n$
of vector superspaces $V$ endowed with parity preserving
left action of $\mb F[\partial]$ and right action of $S_n$
commuting with each other: $\partial(v^\sigma)=(\partial v)^\sigma$,
for every $v\in V$ and $\sigma\in S_n$;
here the morphisms between two such spaces
are parity preserving linear maps commuting 
both with the action $\partial$ and with the action of the symmetric groups $S_n$.

We have a functor from the category $\mb F[\partial]$-mod-$S_n$ 
to the category $\mb F[\partial]$-mod,
obtained by taking the submodule of $S_n$-invariants: $P\mapsto P^{S_n}$.
\begin{lemma}\label{lem:fd-sn}
Let $(P_q)_{q\geq0}$ be a projective system in the category $\mb F[\partial]$-mod-$S_n$.
Then taking $S_n$-invariants we have a projective system $(P_q^{S_n})_{q\geq0}$
in the category $\mb F[\partial]$-mod.
Moreover, if $(P_q)_{q\geq0}$ admits a projective limit $P=\varprojlim P_q$,
then so does $(P_q^{S_n})_{q\geq0}$, and
$$
\varprojlim (P_q^{S_n})= P^{S_n}
\,.
$$
In other words, taking projective limit commutes with taking $S_n$-invariants.
\end{lemma}
\begin{proof}
It immediately follows by the universal properties of the projective limits,
by restricting maps to the subspaces of invariants.
\end{proof}
\begin{lemma}\label{lem:proj-pol}
Let $(P_q)_{p\geq0}$ be a projective system in the category $\mb F[\partial]$-mod.
Taking spaces of polynomials in the variable $\lambda$ commuting with $\partial$, 
we have a projective system 
$(P_q[\lambda])_{q\geq0}$ still in the category $\mb F[\partial]$.
Moreover, if $(P_q)_{q\geq0}$ admits a projective limit $P=\varprojlim P_q$,
then so does $(P_q[\lambda])_{q\geq0}$, and
$$
\varprojlim (P_q[\lambda])
=
\Big\{
a(\lambda)\in P[[\lambda]]
\,\Big|\,
\pi_q(a(\lambda))\in P_q[\lambda] \text{ for every } q\geq0
\Big\}
\subset P[[\lambda]]
\,.
$$
\end{lemma}
\begin{proof}
It immediately follows by the universal property of projective limits,
by looking at the coefficients of each power of $\lambda$.
\end{proof}

\subsection{Pro-Lie conformal superalgebras}

\begin{definition}\label{def:pro-lca}
A \emph{pro-Lie conformal superalgebra} $R$ is
a vector superspace with an even endomorphism $\partial\in\End V$
endowed with a parity-preserving (formal power series valued) $\lambda$-bracket
\begin{equation}\label{eq:pro-lambda}
R\otimes R\to R[[\lambda]]
\,\,,\,\,\,\,
a\otimes b\mapsto [a_\lambda b]
\,.
\end{equation}
We require that the $\lambda$-bracket satisfies
the sesquilinearity \eqref{20170612:eq6} and the Jacobi identity \eqref{20170612:eq5},
which are now understood as identities between formal power series:
\begin{equation}\label{eq:pro-lca1}
\begin{split}
& [\partial a_\lambda b]=-\lambda[a_\lambda b]
\,,\,\,
[a_\lambda \partial b]=(\lambda+\partial)[a_\lambda b]
\,\text{ in } R[[\lambda]]\,, \\
& [a_{\lambda}[b_\mu c]]-(-1)^{p(a)p(b)}[b_\mu [a_\lambda ,b]]
=[[a_\lambda b]_{\lambda+\mu}c]
\,\text{ in } R[[\lambda,\mu]]
\,.
\end{split}
\end{equation}
In order to make sense of the skewsymmetry,
we require that $R=\varprojlim R_q$
is the projective limit of a projective system $(R_q)_{q\geq0}$
in the category $\mb F[\partial]$-mod,
and that, 
for every $q\geq0$, the $\lambda$-barcket composed with $\pi_q$ 
has polynomial values:
\begin{equation}\label{eq:lambda-proj}
\xymatrixcolsep{2pc}
\xymatrixrowsep{0.5pc}
\xymatrix{
& R[[\lambda]] \ar[rd]^{\pi_q} \\
R\otimes R \ar@{.>}[rd]_{\exists!} \ar[ru]^{[\cdot\,_\lambda\,\cdot]}  && R_q[[\lambda]] \\
& R_q[\lambda]\ar@{^{(}->}[ur] 
}
\end{equation}
We then require that the skewsymmetry condition \eqref{20170612:eq4}
holds, after composing with $\pi_q$, in $R_q[\lambda]$, for every $q\geq0$:
\begin{equation}\label{eq:pro-lca2}
\pi_q\big([a_\lambda b])
=
-(-1)^{p(a)p(b)}
\pi_q([b_{-\lambda-\partial}a])
\,\text{ in } R_q[\lambda]
\,.
\end{equation}
\end{definition}
\begin{remark}\label{rem:pro-lca}
By Lemma \ref{lem:proj-pol}, condition \eqref{eq:lambda-proj}
can be equivalently formulated by requesting that the $\lambda$-bracket \eqref{eq:pro-lambda}
is in fact a map
$$
[\cdot\,_\lambda\,\cdot]\,:\,\,
R\otimes R\to
\varprojlim(R_q[\lambda])
\,\subset\,R[[\lambda]]
\,.
$$
\end{remark}
\begin{example}\label{ex:cend}
Let $V$ be an $\mb F[\partial]$-module,
with a filtration 
$0=F^0V\subset F^1V\subset F^2V\subset\dots
\subset V=\bigcup_{q\geq0}F^qV$, consisting of $\mb F[\partial]$-submodules of finite rank.
Let $\Cend(V)$ be the space of conformal endomorphisms of $V$, i.e. linear maps
$\varphi_\sigma:\,V\to V[\sigma]$ satisfying the following sesquilinearity condition:
$$
\varphi_\sigma(\partial v)=(\partial+\sigma)\varphi_\sigma(v)
\,.
$$
The space $\Cend(V)$ then acquires the structure of a $\mb F[\partial]$-module,
by letting $\partial:\,\Cend(V)\to\Cend(V)$ act as 
$(\partial\varphi)_\sigma(v)=-\sigma\varphi_\sigma(v)$.
Moreover, we have a (formal power series valued) $\lambda$-bracket
$[\cdot\,_\lambda\,\cdot]:\,\Cend(V)\otimes\Cend(V)\to\Cend(V)[[\lambda]]$,
$\varphi\otimes\psi\mapsto[\varphi_\lambda\psi]$,
defined by the following identity:
\begin{equation}\label{eq:cend}
[\varphi_\lambda\psi]_{\sigma}(v)
=
\varphi_\lambda(\psi_{\sigma-\lambda}(v))
-
\psi_{\sigma-\lambda}(\varphi_\lambda(v))
\,.
\end{equation}
If $V$ is a finite rank module over $\mb F[\partial]$ , then $\Cend(V)$
is a Lie conformal algebra, see \cite[Sec.2.10]{Kac98}.
In general,
we consider the following projective system of $\mb F[\partial]$-modules:
$R_q$ is the space of linear maps 
$\bar\varphi_\sigma:\,F^qV\to V[\sigma]$ 
satisfying $\bar\varphi_\sigma(\partial v)=(\partial+\sigma)\bar\varphi_\sigma(v)$,
the action of $\partial_q:\,R_q\to R_q$ is given, as before, 
by $(\partial_q\bar\varphi)_\sigma(v)=-\sigma\bar\varphi_\sigma(v)$,
and the maps $\pi_{q+1,q}:\,R_{q+1}\to R_q$ are given by restrictions.
Then,
it is immediate to check that $\Cend(V)$ is indeed the projective limit of this system:
$\Cend(V)=\varprojlim R_q$.
Moreover, given $q\geq0$, let $\{u_i\}_{i=1}^N$ be a finite set of generators 
of the $\mb F[\partial]$-module $F^qV$.
Then, any element of $F^qV$ has the form $v=\sum_{i=1}^NP_i(\partial)u_i$,
where $P_i(\partial)\in\mb F[\partial]$, so, by the sesquilinearity, equation \eqref{eq:cend} gives
$$
[\varphi_\lambda\psi]_{\sigma}(v)
=
\sum_{i=1}^NP_i(\partial+\sigma)\big(
\varphi_\lambda(\psi_{\sigma-\lambda}(u_i))
-
\psi_{\sigma-\lambda}(\varphi_\lambda(u_i))
\big)
\,.
$$
In particular, the degree in $\lambda$ of all polynomials in the RHS is uniformly bounded 
with respect to $v\in F^qV$,
so that $\pi_q[\varphi_\lambda\psi]$ lies in $R_q[\lambda]$, as desired.
The proof of sesquilinearity, skew-symmetry and Jacobi identity
for the $\lambda$-bracket in $\Cend(V)$ is the same as for the finite rank case.
\end{example}
\begin{remark}\label{rem:rcend}
One may also consider the space $\RCend(V)$ of right conformal endomorphisms of $V$,
i.e. linear maps $\varphi_\sigma:\,V\to V[\sigma]$ 
such that $\varphi_\sigma(\partial v)=-\sigma\varphi_\sigma(v)$.
Then $\partial$ acts on $\RCend(V)$ by $(\partial\varphi)_\sigma(v)=(\partial+\sigma)\varphi_\sigma(v)$,
and there is a natural isomorphism $\Cend(V)\simeq\RCend(V)$,
given by $\varphi_\sigma\mapsto\varphi_{-\sigma-\partial}$.
\end{remark}
\begin{lemma}\label{lem:rmoddr}
If $R$ is a pro-Lie conformal superalgebra,
then $R/\partial R$ is a Lie superalgebra, with Lie bracket 
$[\bar a,\bar b]=\overline{[a_\lambda b]}\big|_{\lambda=0}$,
for $\bar a,\bar b\in R/\partial R$ and $a,b\in R$ any their representatives.
Moreover, we have a representation of the Lie superalgebra $R/\partial R$ on $R$,
given by $\bar a(b)=[a_\lambda b]\big|_{\lambda=0}$.
\end{lemma}
\begin{proof}
Same as in the ``non-pro'' case.
\end{proof}

\subsection{Pro-conformal pseudo operads}

\begin{definition}\label{def:topological}
A \emph{pro-conformal pseudo operad} $\widetilde{{\mathscr P}}$ is a collection 
of vector superspaces $\widetilde{{\mathscr P}}(n)$, $n\geq0$,
with parity $p$, endowed with an even endomorphism 
$\partial:\,\widetilde{{\mathscr P}}(n)\to\widetilde{{\mathscr P}}(n)$, 
with a right action of the symmetric group $S_n$ preserving 
the parity and commuting with $\partial$,
denoted $f^\sigma$, for $f\in\widetilde{{\mathscr P}}(n)$ and $\sigma\in S_n$,
and with parity preserving (formal power series valued) $\circ^i_\lambda$-products 
\begin{equation}\label{eq:operad1-top}
\widetilde{{\mathscr P}}(n) \otimes \widetilde{{\mathscr P}}(m)
\to\widetilde{{\mathscr P}}(m+n-1)[[\lambda]] \,,\qquad 
f \otimes g \mapsto f\circ^i_{\lambda} g 
\,,
\end{equation}
satisfying the same \emph{sesquilinearity}, \emph{associativity} and \emph{equivariance} conditions 
as in \eqref{eq:circ-sesq}, \eqref{eq:circ-assoc1-conf}-\eqref{eq:circ-assoc2-conf} and \eqref{eq:circ-equiv-conf}, respectively.
(All these conditions make sense for formal power series.)
We require moreover that each $\widetilde{{\mathscr P}}(n)=\varprojlim\widetilde{{\mathscr P}}_q(n)$ 
is the projective limit of a projective system $(\widetilde{{\mathscr P}}_q(n))_{q\geq0}$ 
in the category $\mb F[\partial]$-mod-$S_n$,
and that, for every $q$, the composition of \eqref{eq:operad1-top} with $\pi_q$ 
has image in the space of polynomials:
\begin{equation}\label{eq:operad1b-top}
\pi_q(f\circ^i_{\lambda}g)
\,\in\,
\widetilde{{\mathscr P}}_q(m+n-1)[\lambda]
\,,\qquad f\in\widetilde{{\mathscr P}}(n),\,g\in \widetilde{{\mathscr P}}(m)
\,.
\end{equation}
\end{definition}
\begin{remark}\label{rem:pro-co}
By Lemma \ref{lem:proj-pol} (rather its generalization to polynomials with $n$ variables), 
condition \eqref{eq:operad1b-top}
can be equivalently formulated by requesting that for the $\circ^i_\lambda$-products \eqref{eq:operad1-top}
we have
$$
f\circ^i_{\lambda}g
\,\in\,
\varprojlim \big(
\widetilde{{\mathscr P}}_q(m+n-1)[\lambda]
\big)
\,\subset\,
\widetilde{{\mathscr P}}_q(m+n-1)[[\lambda]]
\,.
$$
\end{remark}

\subsection{$\mb Z_{\geq-1}$-graded pro-Lie conformal superalgebra 
$\widetilde{W}(\widetilde{{\mathscr P}})$}

Let $\widetilde{{\mathscr P}}$ be a pro-conformal pseudo operad.
We define the $\Box_\lambda$-product
\begin{equation}\label{eq:box-pro}
\Box_\lambda\colon \widetilde{{\mathscr P}}(n)\otimes \widetilde{{\mathscr P}}(m)
\to \widetilde{{\mathscr P}}(n+m-1)[[\lambda]]
\,,
\end{equation}
by formula \eqref{eq:box-conf}.
We observe that, by definition of pseudo conformal operad,
composing both $\circ^i_\lambda$ and $\Box_\lambda$
with the restriction maps 
$\pi_q:\,\widetilde{{\mathscr P}}(n+m-1)\to\widetilde{{\mathscr P}}_q(n+m-1)$,
we get images in the space of polynomials $\widetilde{{\mathscr P}}_q(n+m-1)[\lambda]$.
Hence, by Lemma \ref{lem:proj-pol}, the map \eqref{eq:box-pro}
has in fact images in
$$
\varprojlim(\widetilde{{\mathscr P}}_q(n+m-1)[\lambda])
\,\,\Big(
\subset
\widetilde{{\mathscr P}}(n+m-1)[[\lambda]]
\Big)\,.
$$
In particular, while $\Box_{-\lambda-\partial}$ a priori may not be defined,
for every $q\geq0$ it makes sense to consider the map
$$
\pi_q\circ \Box_{-\lambda-\partial}
\,:\,\,
\widetilde{{\mathscr P}}(n)\otimes \widetilde{{\mathscr P}}(m)
\to
\widetilde{{\mathscr P}}_q(n+m-1)[\lambda]
\,,
$$
and by the universal property of the projective limit, this collection of maps
induces a map, which makes sense to denote $\Box_{-\lambda-\partial}$
\begin{equation}\label{eq:box-pro2}
\Box_{-\lambda-\partial}
\,:\,\,
\widetilde{{\mathscr P}}(n)\otimes \widetilde{{\mathscr P}}(m)
\to
\varprojlim(\widetilde{{\mathscr P}}_q(n+m-1)[\lambda])
\subset
\widetilde{{\mathscr P}}(n+m-1)[[\lambda]]
\,.
\end{equation}
We may then define the $\lambda$-bracket 
\begin{equation}\label{eq:lambda-pro}
[\cdot\,_\lambda\,\cdot]
\,:\,\,
\widetilde{{\mathscr P}}(n)\otimes \widetilde{{\mathscr P}}(m)
\to
\varprojlim(\widetilde{{\mathscr P}}_q(n+m-1)[\lambda])
\subset
\widetilde{{\mathscr P}}(n+m-1)[[\lambda]]
\,,
\end{equation}
by the same formula \eqref{20170603:eq4-conf}.

The following is the analogue of Theorem \ref{20170603:thm2-conf} in the ``pro'' setting:
\begin{theorem}\label{20170603:thm2-pro}
Consider the $\mb Z_{\geq-1}$-graded vector superspace
of invariants with respect to the symmetric groups:
$$
\widetilde{W}(\widetilde{{\mathscr P}})=\bigoplus_{n\geq-1}\widetilde{W}_n
\,\,,\,\,\,\,
\widetilde{W}_n=\widetilde{{\mathscr P}}(n+1)^{S_{n+1}}
\,.
$$
Then $\widetilde{W}(\widetilde{{\mathscr P}})$ has the structure of a pro-Lie conformal superalgebra
with the $\lambda$-bracket defined in \eqref{eq:lambda-pro} 
(restricted to $\widetilde{W}(\widetilde{{\mathscr P}})$).
\end{theorem}
\begin{proof}
By assumption each $\widetilde{{\mathscr P}}(n+1)=\varprojlim \widetilde{{\mathscr P}}_q(n+1)$ 
is the projective limit of a projective system $(\widetilde{{\mathscr P}}_q(n+1))_{q\geq0}$ in the category $\mb F[\partial]$-mod-$S_n$.
Hence, by Lemma \ref{lem:fd-sn},
$\widetilde{W}_n$ is the projective limit of the projective system $(\widetilde{{\mathscr P}}_q(n+1))^{S_{n+1}}$,
in the category $\mb F[\partial]$-mod.
The Lie conformal algebra axioms can be proved in the same way as for the proof of Theorem \ref{20170603:thm2-conf}:
the difference here is that $\lambda$-brackets take value in formal power series,
but applying the projection maps $\pi_q$ we reduce each time to polynomials.
\end{proof}

\subsection{Relation between linear and pro-conformal pseudo operads}

As in the ``non-pro'' case, given a pro-conformal pseudo operad $\widetilde{{\mathscr P}}$,
the quotients ${\mathscr P}(n)=\widetilde{{\mathscr P}}(n)/\partial\widetilde{{\mathscr P}}(n)$ form a linear pseudo operad,
with the action of $S_n$ induced by its action on $\widetilde{{\mathscr P}}(n)$
and the $\circ_i$-products defined by \eqref{eq:operad1-quotb}.
Then the ``pro'' analogue of Proposition \ref{prop:lin-conf} holds:
\begin{proposition}\label{prop:univ-lie-pro}
Let $\widetilde{{\mathscr P}}$ be a pro-conformal pseudo operad
and consider 
the linear pseudo operad ${\mathscr P}=\widetilde{{\mathscr P}}/\partial\widetilde{{\mathscr P}}$ defined above,
the $\mb Z_{\geq-1}$ Lie superalgebra $W({\mathscr P})$,
the $\mb Z_{\geq-1}$ pro-Lie conformal superalgebra $\widetilde{W}(\widetilde{{\mathscr P}})$
from Theorem \ref{20170603:thm2-pro},
and the $\mb Z_{\geq-1}$ Lie superalgebra 
$\widetilde{W}(\widetilde{{\mathscr P}})/\partial \widetilde{W}(\widetilde{{\mathscr P}})$
(cf. Lemma \ref{lem:rmoddr}).
There is a canonical surjective map of $\mb Z_{\geq-1}$-graded Lie superalgebras 
$\widetilde W(\widetilde{{\mathscr P}})/\partial\widetilde W(\widetilde{{\mathscr P}})
\twoheadrightarrow W({\mathscr P})$.
This map restricts to an isomorphism $\widetilde{W}_n/\partial\widetilde{W}_n\simeq W_n$ on degree $n$,
provided that $\partial$ has zero kernel on $\widetilde{{\mathscr P}}(n)$, or that $n=0$.
We have a compatible Lie superalgebra action of $W({\mathscr P})$ on $\widetilde W(\widetilde{{\mathscr P}})$,
preserving the gradings, given by \eqref{eq:W-action},
where $\Box_{-\partial}$ is understood via the universal property of the projective limit,
as for \eqref{eq:box-pro2}.
\end{proposition}
\begin{proof}
The same arguments as for the proof of Proposition \ref{prop:lin-conf} apply in this case.
Passing to the projective limite is straightforward.
\end{proof}
\begin{definition}\label{def:univ-lie-pro}
The \emph{basic cohomology complex}, associated to a pro-conformal operad $\widetilde{{\mathscr P}}$
and an odd element $X\in W_1$ such that $X\Box X=0$, is $(\widetilde{W}(\widetilde{{\mathscr P}}),[X,\,\cdot])$.
\end{definition}

\section{Chom and Chiral pro-conformal pseudo operads}\label{sec:basic-chom-chiral}

\subsection{The Chom pro-conformal pseudo operad}\label{sec:5.1}

Let $V$ be a vector superspace with a given even endomorphism $\partial\in\End(V)$.
The \emph{Chom pro-conformal pseudo operad} $\widetilde{\mc{C}hom}(\Pi V)$ is defined as 
the collection of superspaces
$\widetilde{\mc{C}hom}(\Pi V)(n)$, $n\geq0$,
consisting of all linear maps
$$
f_{\lambda_1,\dots,\lambda_n}\colon (\Pi V)^{\otimes n}\to
\Pi V[\lambda_1,\dots,\lambda_n]
\,,
$$
satisfying the same sesquilinearity conditions as in \eqref{20170613:eq1},
except that now both sided are considered in $\Pi V[\lambda_1,\dots,\lambda_n]$
(and not in its quotient by $\langle\partial+\lambda_1\dots+\lambda_n\rangle$).
In particular, $\widetilde{\mc{C}hom}(\Pi V)(0)=\Pi V$
and $\widetilde{\mc{C}hom}(\Pi V)(1)=\RCend(V)$ (cf. Remark \ref{rem:rcend}).
We define an $\mb F[\partial]$-module structure on $\widetilde{\mc{C}hom}(\Pi V)(n)$
by letting
\begin{equation}\label{eq:dchom}
(\partial f)_{\lambda_1,\dots,\lambda_n}
=
(\partial+\lambda_1+\dots+\lambda_n)\,f_{\lambda_1,\dots,\lambda_n}
\,.
\end{equation}
The right action of $\sigma\in S_n$ on $\widetilde{\mc{C}hom}(\Pi V)(n)$
is given by \eqref{20170613:eq3},
where now both sides of the equality are considered in $\Pi V[\lambda_1,\dots,\lambda_n]$.
Finally, the $\circ^i_\lambda$product of $f\in\widetilde{\mc{C}hom}(\Pi V)(n)$
and $g\in\widetilde{\mc{C}hom}(\Pi V)(m)$ 
is defined by (cf. \eqref{eq:operad8-conf}):
\begin{equation}\label{20170613:eq2-pro}
\begin{array}{l}
\displaystyle{
\vphantom{\Big(}
(f\circ^i_\lambda g)_{\mu_1,\dots,\mu_{m+n-1}}
(v_1\otimes\dots\otimes v_{m+n-1})
:=
\pm
f_{\mu_1,\dots,\stackrel{i}{\widecheck{\vphantom{\big(}\lambda}}+\mu_i+\dots+\mu_{i+m-1},\dots,\mu_{m+n-1}}
\big(
} \\
\displaystyle{
\vphantom{\Big(}
\qquad v_1\otimes\dots\otimes
\stackrel{i}{\widecheck{\vphantom{\big(}g}}_{\mu_i,\dots,\mu_{i+m-1}}(v_i\otimes\dots\otimes v_{i+m-1})
\otimes\dots\otimes v_{m+n-1}
\big)
\,,}
\end{array}
\end{equation}
where the sign $\pm$ is $(-1)^{\bar p(g)(\bar p(v_1)+\dots+\bar p(v_{i-1}))}$.
\begin{remark}
If $V$ is free as $\mb F[\partial]$-module,
then $\widetilde{\mc{C}hom}(\Pi V)$ is in fact a unital pro-conformal operad,
the unity being
\begin{equation}\label{eq:1chom}
1=\id_V\in\widetilde{\mc{C}hom}(\Pi V)(1)/\partial\widetilde{\mc{C}hom}(\Pi V)(1)\simeq\End_{\mb F[\partial]} V
\,.
\end{equation}
If $V$ has torsion, $\id_V$ may not lie 
in $\widetilde{\mc{C}hom}(\Pi V)(1)/\partial\widetilde{\mc{C}hom}(\Pi V)(1)$.
For example, for $V=\mb F$, we have $\widetilde{\mc{C}hom}(\Pi V)(n)=0$ for every $n$.
\end{remark}

The novelty here is that,
even though both sides of \eqref{20170613:eq2-pro}
lie in the space of polynomials
$\Pi V[\lambda,\mu_1,\dots,\mu_{m+n-1}]$
for every $v_1\otimes\dots\otimes v_{m+n-1}\in (\Pi V)^{\otimes(m+n-1)}$,
equation \eqref{20170613:eq2-pro} 
defines in general a formal powers series valued $\circ^i_\lambda$-product:
$$
f\circ^i_{\lambda}g
\,\in\,\widetilde{\mc{C}hom}(\Pi V)(m+n-1)[[\lambda]]
\,,
$$
since the degree in the variable $\lambda$ of the polynomials in the right-hand side
of \eqref{20170613:eq2-pro} may not be uniformly bounded with respect to the choice of the vectors 
$v_1,\dots,v_{m+n-1}$.

We therefore assume that $V$ has as increasing exhaustive filtration 
by finite rank $\mb F[\partial]$-submodules compatible with the $\mb Z/2\mb Z$-grading of $V$:
\begin{equation}\label{eq:filtrV}
0=F^0V\subset F^1V\subset F^2V\subset\dots\subset V=\bigcup_qF^qV
\,.
\end{equation}
To such filtration, we associate projective systems $\widetilde{\mc{C}hom}(\Pi V)_q(n)$, $q\geq0$,
defined as the spaces of maps 
$$
\bar f_{\lambda_1,\dots,\lambda_n}\colon (\Pi F^qV)^{\otimes n}\to
\Pi V[\lambda_1,\dots,\lambda_n]
\,,
$$
satisfying the same sesquilinearity conditions \eqref{20170613:eq1}.
The maps 
\begin{equation}\label{eq:chom-proj-system}
\pi_{q+1,q}:\,\widetilde{\mc{C}hom}(\Pi V)_{q+1}(n)\to\widetilde{\mc{C}hom}(\Pi V)_q(n),\,q\geq0\,,
\end{equation}
are given by restrictions.

\begin{proposition}\label{prop:chom-pro}
The supserspaces $\widetilde{\mc{C}hom}(\Pi V)(n)$, $n\geq0$,
form a pro-conformal pseudo operad,
with the $\mb F[\partial]$-module structure defined by \eqref{eq:dchom},
the action of $S_n$ defined by \eqref{20170613:eq3}, 
the $\circ^i_\lambda$-products defined by \eqref{20170613:eq2-pro},
and the projective systems \eqref{eq:chom-proj-system}.
\end{proposition}
\begin{proof}
Since the filtration \eqref{eq:filtrV} is exhaustive,
it is immediate to check that
$$
\varprojlim \widetilde{\mc{C}hom}(\Pi V)_q(n)
=
\widetilde{\mc{C}hom}(\Pi V)(n) 
\,,
$$
with the restriction maps $\pi_{q}:\,\widetilde{\mc{C}hom}(\Pi V)(n)\to\widetilde{\mc{C}hom}(\Pi V)_q(n)$.
Consider now the composition of $f\circ^i_{\lambda}g$ with the restriction map $\pi_q$.
Namely, we apply this composition to vectors $v_1,\dots,v_{m+n-1}\in F^qV$.
Let $(u_\alpha)_{\alpha\in I}$ be a finite set of generators of the $\mb F[\partial]$-module $F^qV$.
We thus have $v_j=\sum_{\alpha\in I}P_{j,\alpha}(\partial)u_\alpha$ for some polynomials $P_{j,\alpha}(\partial)\in\mb F[\partial]$,
$\alpha\in I,\,j=1,\dots,m+n-1$.
By sesquilinearity, the right-hand side of \eqref{20170613:eq2-pro} becomes
\begin{align*}
& \pm
\sum_{\alpha_1,\dots,\alpha_{m+n-1}\in I}
P_{1,\alpha_1}(-\mu_1)\dots P_{m+n-1,\alpha_{m+n-1}}(-\mu_{m+n-1})
\\
&\qquad f_{\mu_1,\dots,\stackrel{i}{\widecheck{\vphantom{\big(}\lambda}}+\mu_i+\dots+\mu_{i+m-1},\dots,\mu_{m+n-1}}
\big(
u_{\alpha_1}\otimes\dots
\\ & \qquad\qquad
\dots \otimes
\stackrel{i}{\widecheck{\vphantom{\big(}g}}_{\mu_i,\dots,\mu_{i+m-1}}(u_{\alpha_i}\otimes\dots\otimes u_{\alpha_{i+m-1}})
\otimes\dots\otimes u_{\alpha_{m+n-1}}
\big)
\,.
\end{align*}
This is a combination of finitely many polynomials in $\lambda$
whose degrees depend on the generators $u_\alpha,\,\alpha\in I$,
but not on the vectors $v_1,\dots,v_{M_n}$.
As a consequence, the degree in $\lambda$ is uniformly bounded with respect to
the vectors $v_1,\dots,v_{M_n}$, so that
$$
\pi_q(f\circ^i_{\lambda}g))
\,\in\,\widetilde{\mc{C}hom}(\Pi V)_q(m+n-1)[\lambda]
$$
as desired.

The proof of the unity, equivariance and associativity axioms for the 
$\circ^i_\lambda$-products of $\widetilde{\mc{C}hom}(\Pi V)$ 
are the same as in the ``non-pro'' situation, cf. Section \ref{sec:chom}.
\end{proof}

Note that the quotient 
$\widetilde{\mc{C}hom}(\Pi V)/\partial \widetilde{\mc{C}hom}(\Pi V)$
does not coincide in general with $\mc{C}hom(\Pi V)$.
Indeed, we have a canonical map
\begin{equation}\label{eq:quotient-chom}
\widetilde{\mc{C}hom}(\Pi V)/\partial \widetilde{\mc{C}hom}(\Pi V)
\to
\mc{C}hom(\Pi V)
\,,
\end{equation}
defined by composing a map 
$f_{\lambda_1,\dots,\lambda_n}:\,(\Pi V)^{\otimes n}\to\Pi V[\lambda_1,\dots,\lambda_n]$ 
from $\widetilde{\mc{C}hom}(\Pi V)$
with the quotient map 
$\Pi V[\lambda_1,\dots,\lambda_n]\twoheadrightarrow 
\Pi V[\lambda_1,\dots,\lambda_n]/\big\langle\partial+\lambda_1+\dots+\lambda_n\big\rangle$.
And it is easy to check that this map is always injective.
But in general it may not be surjective.
In fact, for $n=0$ it is always bijective,
while for $n\geq1$ is it surjective under the assumption
that the $\mb F[\partial]$-module $V$ is free (cf. \cite[Rem.6.6]{DSK13}).
As a consequence, we cannot apply directly 
Proposition \ref{prop:univ-lie-pro} and Definition \ref{def:univ-lie-pro}
to define the basic cohomology complex attached to a Lie conformal superalgebra $V$.

On the other hand, we have the following
\begin{proposition}\label{prop:basic-cohom}
\begin{enumerate}[(a)]
\item
For an $\mb F[\partial]$-module $V$, we have a canonical injective Lie superalgebra homomorphism
\begin{equation}\label{eq:quotient-chom2}
\widetilde{W}(\widetilde{\mc{C}hom}(\Pi V))/
\partial\widetilde{W}(\widetilde{\mc{C}hom}(\Pi V))
\to
W(\mc{C}hom(\Pi V))
\,,
\end{equation}
induced by the map \eqref{eq:quotient-chom}.
This map 
is always injective, 
in degree $n=-1$ it is always bijective,
while in degree $n\geq0$ it is surjective provided that $V$ is free as $\mb F[\partial]$-module.
\item
We have a representation of the $\mb Z_{\geq-1}$-graded Lie superalgebra $W(\mc{C}hom(\Pi V))$
on $\widetilde{W}(\widetilde{\mc{C}hom}(\Pi V))$,
compatible with the map \eqref{eq:quotient-chom2}.
\end{enumerate}
\end{proposition}
\begin{proof}
The proof of both statements is straightforward. It and can be found in \cite[Prop.6.5 and Rem.6.6]{DSK13}.
\end{proof}
\begin{definition}\label{def:basic-lca-coohom}
Let $V$ be a Lie conformal superalgebra
and consider the odd element 
$X\in W_1(\mc{C}hom(\Pi V))$ such that $X\Box X=0$
associated to the $\lambda$-bracket on $V$ by Proposition \ref{20170612:prop2}(a).
Its action on $\widetilde{W}(\widetilde{\mc{C}hom}(\Pi V))$
given by Proposition \ref{prop:basic-cohom}(b) defines a cohomology complex,
which we call 
the \emph{basic Lie conformal superalgebra cohomology complex}
of $V$ (with coefficients in $V$).
\end{definition}
\begin{remark}\label{rem:basic-lca}
For a Lie conformal superalgebra which is finite rank as an $\mb F[\partial]$-module,
this definition coincides with the basic Lie conformal algebra cohomology complex
introduced in \cite{BKV99}.
\end{remark}

\subsection{The chiral pro-conformal operad}\label{rec:chiral-conf}

As in the previous subsection, $V$ is a vector superspace over $\mb F$ 
with a given even endomorphism $\partial\in\End(V)$,
and we assume that $V$ has as increasing exhaustive filtration \eqref{eq:filtrV}
by finite rank $\mb F[\partial]$-submodules compatible with the $\mb Z/2\mb Z$-grading of $V$.

Recall, from Section \ref{sec:chiral}, that $\mc O_{n}^{\star T}$ denotes the algebra of translation invariant polynomials
in $n$ variables $z_1,\dots,z_n$, localized on the diagonals $z_i-z_j$ for $i\neq j$, cf. \eqref{eq:Ostarn}.
The \emph{Chiral pro-conformal operad} $\widetilde{\mc{P}}^{\ch}(\Pi V)$ is defined as 
the collection of superspaces
$\widetilde{\mc{P}}^{\ch}(\Pi V)(n)$, $n\geq0$,
consisting of all linear maps
$$
f_{\lambda_1,\dots,\lambda_n}\colon (\Pi V)^{\otimes n}\otimes\mc O_{n}^{\star T} \to
\Pi V[\lambda_1,\dots,\lambda_n]
\,,
$$
satisfying the same sesquilinearity conditions as in \eqref{20160629:eq4},
except that now both sides are considered in $\Pi V[\lambda_1,\dots,\lambda_n]$
(and not in its quotient by $\langle\partial+\lambda_1\dots+\lambda_n\rangle$).
In particular, $\widetilde{\mc{P}}^{\ch}(\Pi V)(0)=\Pi V$
and $\widetilde{\mc{P}}^{\ch}(\Pi V)(1)=\RCend(V)$.
We define an $\mb F[\partial]$-module structure on $\widetilde{\mc{P}}^{\ch}(\Pi V)(n)$
by the same equation as \eqref{eq:dchom}, where now $f$ lies in $\widetilde{\mc{P}}^{\ch}(\Pi V)(n)$.
%
%
The right action of $\sigma\in S_n$ on $\widetilde{\mc{P}}^{\ch}(\Pi V)(n)$
is given by \eqref{20160629:eq5},
where now both sides of the equality are considered in $\Pi V[\lambda_1,\dots,\lambda_n]$.

As for the Chom pro-conformal operad,
in order to define $\circ^i_\lambda$-products we need to realize $\widetilde{\mc{P}}^{\ch}(\Pi V)(n)$
as projective limit of a projective system.
Indeed, to the filtration \eqref{eq:filtrV}
we associate the projective systems $\widetilde{\mc{P}}^{\ch}(\Pi V)_q(n)$, $q\geq0$,
defined as the spaces of maps 
$$
\bar f_{\lambda_1,\dots,\lambda_n}\colon (\Pi F^qV)^{\otimes n}\otimes\mc O_{n}^{\star T}
\to
\Pi V[\lambda_1,\dots,\lambda_n]
\,,
$$
satisfying the same sesquilinearity conditions \eqref{20160629:eq4}.
The maps 
\begin{equation}\label{eq:Pch-proj-system}
\pi_{q+1,q}:\,\widetilde{\mc{P}}^{\ch}(\Pi V)_{q+1}(n)\to\widetilde{\mc{P}}^{\ch}(\Pi V)_q(n),\,q\geq0\,,
\end{equation}
are given by restrictions.
Readily, $\widetilde{\mc{P}}^{\ch}(\Pi V)(n)=\varprojlim \widetilde{\mc{P}}^{\ch}(\Pi V)_q(n)$.

Then, the $\circ^i_\lambda$-product of $f\in\widetilde{\mc{P}}^{\ch}(\Pi V)(n)$
and $g\in\widetilde{\mc{P}}^{\ch}(\Pi V)(m)$ 
is defined by
\begin{equation}\label{20170613:eq2-pro2}
\begin{split}
& \big(f\circ^i_{\lambda}g\big)_{\mu_1,\dots,\mu_{m+n-1}}^{z_1,\dots,z_{m+n-1}}
\big(v_1\otimes\dots\otimes v_{m+n-1}\otimes p(z_1,\dots,z_{m+n-1})\big) \\
& :=
\pm
f_{\mu_1,\dots,\stackrel{i}{\widecheck{\vphantom{\big(}\lambda}}+\mu_i+\dots+\mu_{i+m-1},\dots,\mu_{m+n-1}}^{
z_1,\dots z_{i-1},\stackrel{i}{\widecheck{\vphantom{\big(}Z}},z_{i+m},\dots,z_{m+n-1}
}
\Big( v_1\otimes\dots
\\
&\dots\otimes
\stackrel{i}{\widecheck{\vphantom{\big(}\phantom{g}}}\!\!\!\!\!\!
g_{\mu_i-\partial_{z_i},\dots,\mu_{i+m-1}-\partial_{z_{i+m-1}}}^{z_i,\dots,z_{i+m-1}}
(v_i\otimes\dots\otimes v_{i+m-1}\otimes q(z_i\otimes\dots\otimes z_{i+m-1}))_\to
\otimes\dots
\\
&\dots\otimes v_{m+n-1}
\otimes
h(z_1,\dots,z_{m+n-1})\big|_{z_i=\dots=z_{m+i-1}=Z}
\big)
\,,
\end{split}
\end{equation}
where the sign $\pm$ is the same as in \eqref{20170613:eq2-pro}
and $q\in\mc O^{\star,T}_m$ and $h\in\mc O^{\star,T}_{m+n-1}$ are such that
$p(z_1,\dots,z_{m+n-1})
=
q(z_i,\dots,z_{i+m-1})
h(z_1,\dots,z_{m+n-1})$,
and $h$ has no poles at $z_k=z_l$ such that $i\leq k<l\leq i+m-1$.
As in \eqref{circ5}, in the right-hand side the arrows means that 
we first take the partial derivatives of $h$ and then we make the substitutions 
$z_{i}=\dots=z_{i+m-1}=Z$.

As for the Chom pro-conformal operad,
even though both sides of \eqref{20170613:eq2-pro2}
lie in the space of polynomials
$\Pi V[\lambda,\mu_1,\dots,\mu_{m+n-1}]$
for every $v_1\otimes\dots\otimes v_{m+n-1}\in (\Pi V)^{\otimes m+n-1}$,
equation \eqref{20170613:eq2-pro2} 
defines in general a formal powers series valued $\circ^i_\lambda$-product:
$$
f\circ^i_{\lambda}g
\,\in\,\widetilde{\mc{P}}^{\ch}(\Pi V)(m+n-1)[[\lambda]]
\,,
$$
since the degree in the variable $\lambda$ of the polynomials in the right-hand side
of \eqref{20170613:eq2-pro2} may not be uniformly bounded with respect to the choice of the vectors 
$v_1,\dots,v_{m+n-1}$.

\begin{proposition}\label{prop:Pch-pro}
The supserspaces $\widetilde{\mc{P}}^{\ch}(\Pi V)(n)$, $n\geq0$,
form a pro-conformal pseudo operad,
with the $\mb F[\partial]$-module structure defined by \eqref{eq:dchom},
the action of $S_n$ defined by \eqref{20160629:eq5}, 
the $\circ^i_\lambda$-products defined by \eqref{20170613:eq2-pro2},
and the projective systems \eqref{eq:Pch-proj-system}.
\end{proposition}
\begin{proof}
Same as for Proposition \ref{prop:chom-pro}.
\end{proof}

The ``chiral'' version of Proposition \ref{prop:basic-cohom} holds as well:
\begin{proposition}\label{prop:basic-ch-cohom}
Let $V$ be an $\mb F[\partial]$-module.
\begin{enumerate}[(a)]
\item
We have a canonical injective map 
\begin{equation}\label{eq:quotient-ch-chom}
\widetilde{\mc{P}}^{\ch}(\Pi V)(n)/
\partial\widetilde{\mc{P}}^{\ch}(\Pi V)(n)
\to
\mc{P}^{\ch}(\Pi V)(n)
\,,
\end{equation}
which is bijective for $n=0$ while, for $n\geq1$, is surjective provided that $V$ is free as an $\mb F[\partial]$-module.
\item
We have a canonical injective Lie superalgebra homomorphism
\begin{equation}\label{eq:quotient-ch-chom2}
\widetilde{W}(\widetilde{\mc{P}}^{\ch}(\Pi V))/
\partial\widetilde{W}(\widetilde{\mc{P}}^{\ch}(\Pi V))
\to
W(\mc{P}^{\ch}(\Pi V))
\,,
\end{equation}
induced by the map \eqref{eq:quotient-ch-chom}.
This map 
is always injective, 
in degree $n\geq0$ it is surjective provided that $V$ is free as $\mb F[\partial]$-module,
while in degree $n=-1$ it is bijective for arbitrary module $V$.
\item
We have a representation of the $\mb Z_{\geq-1}$-graded Lie superalgebra $W(\mc{P}^{\ch}(\Pi V))$
on $\widetilde{W}(\widetilde{\mc{P}^{\ch}}(\Pi V))$,
compatible with the map \eqref{eq:quotient-ch-chom2}.
\end{enumerate}
\end{proposition}
\begin{proof}
(a) For $n=0$ and $1$ the chiral operad coincides with the Chom operad,
hence \eqref{eq:quotient-ch-chom} reduces to \eqref{eq:quotient-chom},
which is always an isomorphism if $n=0$, 
and it is an isomorphism for a free $\mb F[\partial]$-module if $n=1$.
If $n\geq2$, we prove surjectivity of \eqref{eq:quotient-ch-chom} provided that $V=\mb F[\partial]U$
is free.
Let $f_{\lambda_1,\dots,\lambda_n}:\, (\Pi V)^{\otimes n}\otimes\mc O^{\star,T}_n\to (\Pi V)[\lambda_1,\dots,\lambda_n]$
be an element of ${\mathscr P}^{\ch}(\Pi V)(n)$.
We construct a preimage $\widetilde{f}\in\widetilde{{\mathscr P}}^{\ch}(\Pi V)(n)$ as follows.
Let $u_1,u_2,\dots$ be a basis of the vector superspace $U$. We let
\begin{equation}\label{20240524:eq1}
\widetilde{f}_{\lambda_1,\dots,\lambda_n}(u_{i_1}\otimes\dots\otimes u_{i_n}\otimes p)
=\frac1n\sum_{\ell=1}^n
\iota_\ell\big(
f_{\lambda_1,\dots,\lambda_n}(u_{i_1}\otimes\dots\otimes u_{i_n}\otimes p)
\big)
\,,
\end{equation}
for every $p\in\mc O^{\star,T}_n$, where $\iota_\ell$ is the map
$$
\iota_\ell:\,
(\Pi V)[\lambda_1,\dots,\lambda_n]/\langle\partial+\lambda_1+\dots+\lambda_n\rangle
\simeq
(\Pi V)[\lambda_1,
\stackrel{\ell}{\check{\dots}}
,\lambda_n]
\subset
(\Pi V)[\lambda_1,\dots,\lambda_n]
\,,
$$
and the first isomorphism is obtained by substituting $\lambda_\ell$ by 
$-\partial-\lambda_1-\stackrel{\ell}{\check{\dots}}-\lambda_n$.
It is straightforward to check that, since $f$ satisfies the sesquilinearity conditions \eqref{20160629:eq4} 
in $\Pi V[\lambda_1,\dots,\lambda_n]\big/\big\langle\partial+\lambda_1+\dots+\lambda_n\big\rangle$,
then $\widetilde{f}$, evaluated on generators in $U^{\otimes n}$, 
satisfies the second sesquilinearity condition in $\Pi V[\lambda_1,\dots,\lambda_n]$:
$$
\widetilde{f}_{\lambda_1,\dots,\lambda_n}(u_{i_1}\otimes\dots\otimes u_{i_n}\otimes z_{ij}p)
=
\Bigl(\frac{\partial}{\partial\lambda_j}-\frac{\partial}{\partial\lambda_i}\Bigr)
f_{\lambda_1,\dots,\lambda_n}(u_{i_1}\otimes\dots\otimes u_{i_n}\otimes p)
\,.
$$
This is because, for each $i,j$ and $\ell$, 
the operator $\frac{\partial}{\partial\lambda_j}-\frac{\partial}{\partial\lambda_i}$
commutes with the map $\iota_\ell$.
Then, we extend $\widetilde{f}$ inductively (on the powers of $\partial$) to $V^{\otimes n}\otimes\mc O^{\star,T}_n$
by applying the first sesquilinearity condition.
Explicitly, we let
\begin{equation}\label{20240524:eq2}
\begin{split}
& \widetilde{f}_{\lambda_1,\dots,\lambda_n}
(h_1(\partial)u_{i_1}\otimes\dots\otimes h_n(\partial)u_{i_n}\otimes p) \\
& =
\widetilde{f}_{\lambda_1,\dots,\lambda_n}
\Big(u_{i_1}\otimes\dots\otimes u_{i_n}
\otimes h_1\big(-\lambda_1+\frac{\partial}{\partial z_1}\big)\dots h_n\big(-\lambda_n+\frac{\partial}{\partial z_n}\big)
p\Big) 
\end{split}
\end{equation}
for every $p\in\mc O^{\star,T}_n$
and every polynomials $h_1(\partial),\dots,h_n(\partial)\in\mb F[\partial]$.
It is clear that such $\widetilde{f}$ is uniquely defined by the above formula.
One can check that it satisfies both sesquilinearity conditions \eqref{20160629:eq4}
and that, under the map \eqref{eq:quotient-ch-chom} it maps back to $f$.

For part (b) we just need to observe that the map \eqref{eq:quotient-ch-chom} defined by
\eqref{20240524:eq1} and \eqref{20240524:eq2}
commutes with the right action of the symmetric group $S_n$.
Claim (c) is also straightforward.
\end{proof}
\begin{example}\label{ex:V=C}
When $V$ is not free as a $\mb F[\partial]$-module, the map \eqref{eq:quotient-ch-chom}
does not need to be surjective.
As an example, consider the case when $V=\mb F$, the one dimensional space,
considered as a $\mb F[\partial]$-module with $\partial=0$.
In this case $\widetilde{\mc{P}}^{\ch}(n)=0$ for every $n\geq1$.
Indeed, an element $\tilde f\in\widetilde{\mc{P}}^{\ch}(n)$
is a map $\tilde f_{\lambda_1,\dots,\lambda_n}:\,\mc O^{\star,T}_n\to\mb F[\lambda_1,\dots,\lambda_n]$
satisfying the sesquilinearity conditions \eqref{20160629:eq4}:
\begin{equation}\label{20240523:eq1}
\begin{split}
& \lambda_i\tilde f_{\lambda_1,\dots,\lambda_n}(p)
=
\tilde f_{\lambda_1,\dots,\lambda_n}\Bigl(\frac{\partial p}{\partial z_i}\Bigr)
\,, \\
& \tilde f_{\lambda_1,\dots,\lambda_n}(z_{ij}p)
=
\Bigl(\frac{\partial}{\partial\lambda_j}-\frac{\partial}{\partial\lambda_i}\Bigr)
\tilde f_{\lambda_1,\dots,\lambda_n}(p)
\,,
\end{split}
\end{equation}
for every $p\in\mc O^{\star,T}_n$.
Summing the first such condition over $i$, we get
$$
(\lambda_1+\dots+\lambda_n)
\tilde f_{\lambda_1,\dots,\lambda_n}(p)=0
\,,
$$
which obviously implies $\tilde f=0$.
On the other hand, 
${\mc{P}}^{\ch}(n)$
consists of maps 
$$
f_{\lambda_1,\dots,\lambda_{n-1}}:\,\mc O^{\star,T}_n
\to\mb F[\lambda_1,\dots,\lambda_n]/(\lambda_1+\dots+\lambda_n)
\simeq\mb F[\lambda_1,\dots,\lambda_{n-1}]
$$
satisfying the same sesquilinearity conditions \eqref{20240523:eq1}.
In particular, ${\mathscr P}^{\ch}(1)=\mb F$,
and it is not hard to check that also ${\mathscr P}^{\ch}(2)$ is one dimensional,
spanned by the element defined by
$$
f_\lambda(z^n)
=\left\{
\begin{array}{ll}
0 & \text{ if } n\geq0\,, \\
\frac{(-\lambda)^{-n-1}}{(-n-1)!} &\text{ if } n\leq-1\,.
\end{array}
\right.
$$
\end{example}
\begin{definition}\label{def:basic-va-coohom}
Let $V$ be a vertex algebra
and consider the odd elements 
$X\in W_1(\mc{P}^{\ch}(\Pi V))$ such that $X\Box X=0$
associated to the vertex algebra structure of $V$ by Proposition \ref{20170612:prop2-ch}.
Its action on $\widetilde{W}(\widetilde{\mc{P}}^{\ch}(\Pi V))$
given by Proposition \ref{prop:basic-ch-cohom}(c) defines a cohomology complex,
which we call 
the \emph{basic vertex algebra cohomology complex}
of $V$ (with coefficients in $V$).
\end{definition}
For an arbitrary $V$-module $M$ the definition is extended as explained in Remark \ref{rem:27}.
Here is an explicit description of this complex:
$$
\big(\widetilde{\mc C}^{\ch}(V,M),\tilde d_X\big)
=
\Big(
\sum_{n\geq0}\widetilde{\mc C}^n(V,M),\tilde d_X\big)
\,,
$$
where
$\widetilde{\mc C}^n(V,M)$ is the space of linear maps
$$
\tilde f_{\lambda_1,\dots,\lambda_n}^{z_1,\dots,z_n}
\colon (\Pi V)^{\otimes n}\otimes\mc O_{n}^{\star T}
\to
\Pi M[\lambda_1,\dots,\lambda_n]
\,,
$$
satisfying the two sesquilinearity conditions \eqref{20160629:eq4},
and the symmetry condition, see \eqref{20160629:eq5}:
$\tilde{f}^\sigma=\tilde f$.
The differential $\tilde d_X$ is given by \cite[Eq.(7.6)]{BDSHK19},
or in a different form in \cite[Eq.(4.12)]{BDSK21}.
\begin{example}\label{ex:ciao}
We have $\widetilde{\mc C}^0(V,M)=M$.
Then, for $m\in M$, $\tilde d_Xm:\,V\to M[\lambda]$ is defined by
$$
(\tilde d_Xm)_\lambda v
=
-(-1)^{(1+p(v))p(m)}
v_{-\lambda-\partial}m
\,.
$$
In particular, for $M=V$ we have $(\tilde d_Xm)_\lambda v
=(-1)^{p(m)}m_\lambda v$.
Hence, 
$$
\widetilde{H}^0(V,V)
=
\big\{v\in V\,\big|\,
[V_\lambda v]=0\big\}
\,,
$$
is the center of $V$.
\end{example}
\begin{example}
Let $\bar R=\mb F[\partial]\otimes\mf g$ and $R=\bar R\oplus\mb FC$ be Lie conformal algebras,
where $\mf g$ is a simple Lie algebra with a symmetric invariant bilinear form $(\cdot\,,\,\cdot)$,
with $\lambda$-brackets, respectively ($a,b\in\bar R$)
$$
[a_\lambda b]=[a,b]
\,,\qquad
[a_\lambda b]=[a,b]
+\lambda(a,b)C
\,,\,\,
C\text{ central}.
$$
Let $\mc V(\bar R)$ and $\mc V(R)$ be the corresponding Poisson vertex algebras
and, for $k\in\mb F$, let $V^k(R)=V(R)/(C-k)$, be the universal affine vertex algebra of level $k$.
By \cite[Prop.2.11(c) and Thm.4.10]{BDSHK20} we have
$$
H^n_{\text{PV}}(\mc V(\bar R),\mc V(\bar R))\simeq H^n_{\mc Chom}(\bar R,\mb F)
\,,\qquad
\text{ for } n=0,1\,.
$$
Consequently, by \cite[Thm.6.12]{BDSK21}, we have
$$
H^n_{\ch}(V^k(\mf g),V^k(\mf g))
\simeq
H^n_{\mc Chom}(\bar R,\mb F)
\,,\qquad
\text{ for } n=0,1\,.
$$
Hence, by \cite[Thm.4.10]{BDSK20}, we obtain
$$
H^0_{\ch}(V^k(\mf g),V^k(\mf g))=\mb F\vac
\,,\qquad
H^1_{\ch}(V^k(\mf g),V^k(\mf g))=0
\,.
$$
Consequently all Casimirs of $V^k(\mf g)$ are trivial and all derivations of $V^k(\mf g)$ are inner.
\end{example}

\section{Operads in pseudo tensor categories}\label{sec:7}

In this section we generalize the constructions of the previous sections to the
general context of pseudo tensor categories. 
%

\subsection{Pseudo Tensor Categories}\label{sec:6.1}

For a general study of pseudo tensor categories we refer the readers to
\cite[Sec.1.1]{BD04} or \cite[Sec.3]{BDAK01}, these are also known in the literature
as \emph{multicategories} \cite{Lam69},
or as \emph{colored operads} \cite{Yau16}. As in the previous sections, 
all categories and operads are $\mb F$-linear.

\begin{definition}\label{defn:pseudo-tensor}
A \emph{pseudo tensor category} consists of the following data:
\begin{enumerate}
\item a class ${\mc C}$ of objects;
\item for $n\geq0$ and a collection $\left\{ L_1,\cdots,L_n,L \right\}$ of $n+1$ objects in ${\mc C}$, an $\mb F$-vector space ${\mc C}(L_1,\cdots,L_n;L)$;
\item for $n\geq1$ and a permutation $\sigma \in S_n$, an isomorphism of vector spaces 
\begin{equation}\label{equation1b} 
{\mc C}(L_1,\cdots,L_n; L) \stackrel{\sim}{\longrightarrow} 
{\mc C}(L_{\sigma(1)},\cdots,L_{\sigma(n)}, L)
\,, \qquad 
\varphi \mapsto \varphi^{\sigma}\,;
\end{equation}
\item for any object $L$ of ${\mc C}$ a vector $\id_L \in {\mc C}(L;L)$;
\item for a collection $\left\{ m_1,\cdots,m_n \right\} \subset \mathbb{Z}_{\geq1}$, 
denote $M_0=0$ and  $M_i = \sum_{j=1}^{i} m_j$ for $1 \leq i \leq n$. 
Let $K_j$, $1 \leq j \leq M_n$ be objects. We have composition maps 
\begin{equation}\label{equation1} 
{\mc C}(L_1,\cdots,L_n;L) \otimes \left( \bigotimes_{i=1}^n {\mc C} \left( K_{M_{i-1}+1},\cdots,K_{M_i};L_i \right)\right) 
\rightarrow
{\mc C} \left( K_{1}, \cdots, K_{M_n} ; L\right),
\end{equation}
denoted $\psi \otimes \left( \otimes_{i=1}^n \varphi_{i} \right) \mapsto \psi(\varphi_1 \otimes\ldots \otimes \varphi_n)$. 
\end{enumerate}
These data are subject to the following axioms:
\begin{enumerate}[(i)]
\item
The isomorphisms $\varphi \mapsto \varphi^{\sigma}$ in \eqref{equation1b}
define a right linear action of $S_n$, in the sense that the following associativity holds: 
$(\varphi^\sigma)^\tau=\varphi^{\sigma\tau}$.
\item 
The composition maps \eqref{equation1} are associative in the sense that,
for every $\psi\in{\mc C}(L_1,\cdots,L_n;L)$, 
for $\varphi_i\in{\mc C} ( K_{M_{i-1}+1},\cdots,K_{M_i};L_i ),\,i=1,\dots,n$, 
and for $\pi_j\in{\mc C} ( H_{\sum_{k=1}^{j-1}\ell_k+1},\cdots,H_{\sum_{k=1}^{j}\ell_k};K_j )
,\,j=1,\dots,M_n$, we have
\begin{align*}
& \big(\psi(\varphi_1\otimes\dots\otimes\varphi_n)\big)(\pi_1\otimes\dots\otimes\pi_{M_n}) \\
& \qquad =
\psi\big(\varphi_1(\pi_1\otimes\dots\otimes\pi_{M_1})
\otimes\dots\otimes
\varphi_n(\pi_{M_{n-1}+1}\otimes\dots\otimes\pi_{M_n})
\big)\,.
\end{align*}
\item 
The composition map \eqref{equation1} are compatible 
with the action of the symmetric groups \eqref{equation1b} as follows. 
For $\sigma \in S_n$ and $\tau_i \in S_{m_i}$, $i=1,\dots,n$, we have 
$$
\psi^\sigma \left( \varphi_{\sigma(1)}^{\tau_{\sigma(1)}} 
\otimes \cdots \otimes 
\varphi_{\sigma(n)}^{\tau_{\sigma(n)}} \right) 
= 
\left( \psi(\varphi_1 \otimes \cdots \otimes \varphi_n) \right)^{\sigma(\tau_1,\dots,\tau_n)}\,,
$$
where $\sigma(\tau_1,\dots,\tau_n)$ is the composition of permutations defined by \eqref{eq:operad19}.
\item 
The elements $\id_L\in\mc C(L,L)$ act as the identities under the composition, in the sense that,
for $\varphi\in{\mc C}(L_1,\cdots,L_n;L)$, 
$$
\varphi(\id_{L_1}\otimes\dots\otimes\id_{L_n})
=
\id_L(\varphi)=\varphi
\,.
$$
\end{enumerate}
\end{definition}

\begin{remark}\label{rem:operad-1}
A particular case is when $L_1 = L_2 = \cdots = L_n=L$.
In this case 
the collection of vector spaces 
${\mathscr P}(n):={\mc C}(L, \cdots, L;L)$, $n\geq0$, is automatically a linear symmetric operad,
as defined in Section \ref{sec:2.1}.
Indeed, ${\mc C}(L, \cdots, L;L)$ is a right $S_n$-module with the action given by \eqref{equation1b},
and the axioms (ii), (iii), (iv) restrict, in this case, to the axioms \eqref{eq:operad2}, \eqref{eq:operad4}
and \eqref{eq:operad3} of the compositions in a unital operad.
In particular, a (linear symmetric) operad is the same as a pseudo tensor category with a single object.
\end{remark}

\begin{remark}\label{rem:augmentation}
Our Definition \ref{defn:pseudo-tensor} in fact corresponds 
to what Beilinson and Drinfeld call in \cite{BD04} an \emph{augmented} pseudo tensor category,
as we are including the case $n=0$.
\end{remark}

\begin{example}\label{ex:Hom}
A symmetric monoidal category ${\mc C}$ is a pseudo tensor category with 
$$
{\mc C} (L_1,\cdots,L_n;L) := \Hom_{{\mc C}} ( L_1 \otimes\cdots \otimes L_n, L
), \qquad n > 0,
$$
and 
\[ 
{\mc C}(\left\{  \right\}; L) := \Hom_{ {\mc C}} (\id, L),
\]
where $\id$ is the unit object of ${\mc C}$. Given a vector space $V$, we then recover, by Remark \ref{rem:operad-1},
the $\mc Hom$-operad of Section \ref{sec:Hom}:
$$
(\mc Hom\,V)(n)=\mc C(V,\dots,V;V)
\,,\qquad n\geq0\,.
$$
\end{example}
\begin{example}\label{ex:cHom}
The category of $\mb F[\partial]$-modules has a pseudo tensor structure, 
that we call the \emph{conformal pseudo tensor structure} ${\mc C}^*$, where 
$$
{\mc C}^*(L_1,\cdots,L_n; L) \subset 
\Hom_{\mb F}\big( L_1 \otimes\cdots \otimes L_n, 
L[\lambda_1,\cdots,\lambda_n]/\langle \partial+ \lambda_1 + \cdots + \lambda_n  \rangle \big)
\,,
$$
consists of maps 
$$
l_1 \otimes \cdots \otimes l_n \mapsto X_{\lambda_1, \ldots,\lambda_n} (l_1 \otimes \cdots \otimes l_n)
\,,
$$
satisfying the following sesquilinearity condition for all $i=1,\dots,n$:
\begin{equation} \label{eq:sesqui1}
X_{\lambda_1,\ldots,\lambda_n} (l_1  \otimes  \cdots  \otimes \partial l_i \otimes \cdots  \otimes  l_n) 
= - \lambda_i X_{\lambda_1, \ldots,\lambda_n} (l_1  \otimes  \cdots  \otimes  l_n)
\,. 
\end{equation}
For $n=0$ the empty tensor product is, by convention, the basis field $\mb F$, so
$\mc C^*(\left\{  \right\};L)=L/\partial L$.
The action of the symmetric group $S_n$ in \eqref{equation1b} is given as follows. 
For $X \in {\mc C}^*(L_1,\cdots,L_n;L)$ and $\sigma \in S_n$ we have
\begin{equation} \label{eq:symmetric*}
(X^\sigma)_{\lambda_1,\ldots,\lambda_n} (l_1\otimes \cdots\otimes l_{n}) 
:= X_{\lambda_{\sigma^{-1}(1)},\ldots,\lambda_{\sigma^{-1}(n)} } 
(l_{\sigma^{-1}(1)}\otimes\cdots\otimes l_{\sigma^{-1}(n)}).
\end{equation}
Composition is defined as follows. Let $m_1,\ldots,m_n \in \mathbb{Z}_{\geq0}$ 
and define $M_0=0$ and  $M_i = \sum_{j=1}^{i} m_j$ for $1 \leq i \leq n$.  
Let $\Lambda_i= \sum_{j=M_{i-1}+1}^{M_i} \lambda_j$ for $1 \leq i \leq n$. 
For $X \in {\mc C}^*(L_1,\ldots,L_n;L)$ 
and $Y_i \in {\mc C}^*(K_{M_{i-1}+1},\ldots,K_{M_{i}};L_i)$, we let 
$X(Y_1 \otimes \cdots \otimes Y_n) \in {\mc C}^*(K_1,\cdots,K_{M_n}; L)$ be defined by
\begin{equation}\label{eq:composition*}
\begin{split}
& \big(X(Y_1 \otimes \cdots  \otimes  Y_n) \big)_{\lambda_1,\ldots,\lambda_{M_n}} 
( k_1  \otimes  \cdots  \otimes  k_{M_n} ) \\
& =   
X_{\Lambda_1,\ldots,\Lambda_n} \big( 
(Y_1)_{\lambda_1,\ldots,\lambda_{m_1}}(k_1  \otimes  \cdots  \otimes  k_{m_1})  
\otimes  \cdots  \\
&\qquad\qquad \dots\otimes
(Y_n)_{\lambda_{M_{n-1}+1},\ldots,\lambda_{M_n}}
(k_{M_{n-1}+1} \otimes  \ldots  \otimes  k_{M_n} ) \big).
\end{split}
\end{equation}
Compare the above formulas \eqref{eq:sesqui1}, \eqref{eq:symmetric*} and \eqref{eq:composition*},
with \eqref{20170613:eq1}, \eqref{20170613:eq3} and \eqref{20170613:eq2} respectively.
As a result, for a given $\mb F[\partial]$-module $V$,
we recover, by Remark \ref{rem:operad-1}, the $\mc Chom$ operad of Section \ref{sec:chom}:
$$
(\mc Chom\,V)(n)=\mc C^*(V,\dots,V;V)
\,,\quad n\geq0
\,.
$$
\end{example}
\begin{example}\label{ex:ch}
We next take the same class of objects of $\mb F[\partial]$-modules, but with a different
pseudo tensor structure, obtaining the \emph{vertex pseudo tensor category} $\mc C^\ch$. 
We let ${\mc O}^{\star T}_n$ be as in \eqref{eq:Ostarn}.
We then set 
\begin{align*}
& {\mc C}^{ch}(L_1, \cdots,L_n; L) \\
&\quad 
\subset \Hom
\big( L_1 \otimes\cdots \otimes L_n \otimes {\mc O}^{\star T}_{n}, 
L[\lambda_1,\cdots,\lambda_n]/ \langle \partial+ \lambda_1 + \cdots \lambda_n \rangle \big) 
\end{align*}
be the subspace consisting of maps 
$$
l_1 \otimes \cdots \otimes l_n \otimes f(z_1,\dots,z_n)
\mapsto 
X_{\lambda_1,\ldots,\lambda_n}^{z_1,\dots,z_n} ( l_1\otimes \cdots\otimes l_n\otimes f )
\,,
$$
(as for the chiral operad, when not necessary we remove the upper indices $z_1,\dots,z_n$)
satisfying the following sesquilinearity conditions
\begin{equation} \label{eq:sesqui2} 
\begin{split}
& X_{\lambda_1,\ldots,\lambda_n} 
\big( l_1\otimes  \cdots \otimes(\partial + \lambda_i) l_i \otimes  \cdots \otimes l_n \otimes f \big) 
= 
X_{\lambda_1,\ldots,\lambda_n} 
\Big( l_1\otimes\cdots\otimes l_n\otimes \frac{\partial f}{\partial z_i} \Big)\,, \\
& X_{\lambda_1,\ldots,\lambda_n} 
\big( l_1\otimes\cdots\otimes l_n\otimes (z_i-z_j)f \big) 
=
\Big( \frac{\partial}{\partial \lambda_j} - \frac{\partial}{\partial \lambda_i} \Big) 
X_{\lambda_1,\ldots,\lambda_n} ( l_1\otimes\cdots\otimes l_n\otimes f )\,. 
\end{split}
\end{equation}
The action of $S_n$ in \eqref{equation1b} is given as follows. 
For $X \in {\mc C}^{ch}(L_1,\cdots,L_n;L)$ and $\sigma \in S_n$ we have
\begin{equation} \label{eq:symmetric-ch} 
\begin{split}
& (X^\sigma)^{z_1,\ldots,z_n}_{\lambda_1,\ldots,\lambda_n} 
\big(l_1\otimes \cdots\otimes l_{n}\otimes f(z_1,\cdots,z_n)\bigr) \\
& := X^{z_1, \ldots, z_n}_{\lambda_{\sigma^{-1}(1)},\ldots,\lambda_{\sigma^{-1}(n)} } 
\big( l_{\sigma^{-1}(1)}\otimes\cdots\otimes l_{\sigma^{-1}(n)}\otimes 
f(z_{\sigma^{-1}(1)},\cdots,z_{\sigma^{-1}(n)}) \big).
\end{split}
\end{equation}
Furthermore, compositions are defined as follows. 
Let $X \in {\mc C}^{ch}(L_1,\cdots,L_n;L)$ 
and $Y_i \in {\mc C}^{ch}(K_{M_{i-1}+1},\cdots,K_{M_i}; L_i)$ 
for $1 \leq i \leq n$ and $m_1,\cdots,m_n \in \mathbb{Z}_{\geq0}$ as above. 
Recall that a function $p(z_1,\dots,z_n) \in {\mc O}^{\star T}_{M_n}$ can be decomposed
as in \eqref{circ4}.
We then define, using the same notation as in \eqref{eq:composition*},
\begin{equation}\label{eq:composition-ch}
\begin{split}
& \big(X ( Y_1 \otimes \ldots \otimes Y_n ) \big)_{\lambda_1,\ldots,\lambda_{M_n}}^{z_1,\ldots,z_{M_n}} 
\big( k_1\otimes\ldots\otimes k_{M_n}\otimes p(z_1,\dots,z_{M_n}) \big) \\
& =
X_{\Lambda_1,\ldots,\Lambda_n}^{z_{M_1},\ldots,z_{M_n}} 
\Big( 
(Y_1)_{\lambda_1 - \partial_{z_1},\ldots,\lambda_{m_1} - \partial_{z_{m_1}}}^{z_1,\ldots,z_{m_1}}
\big(k_1 \otimes\ldots\otimes k_{m_1}\otimes q_1(z_1,\dots,z_{m_1})\big)_{\rightarrow} \otimes \ldots \\ 
& \dots\otimes\!
(Y_n)_{\lambda_{M_{n\!-\!1}\!+\!1} \!-\! \partial_{z_{M_{n\!-\!1}\!+\!1}},\dots, 
\lambda_{M_n} \!\!-\!\partial_{z_{M_n}}}^{z_{M_{n\!-\!1}\!+\!1},\ldots, z_{M_n}} 
\!\big(
k_{M_{n\!-\!1}\!+\!1}\otimes \dots\otimes k_{M_n} \!\otimes q_n(z_{M_{n\!-\!1}\!+\!1},\dots,z_{M_n})
\big)_{\rightarrow} \\
&\qquad\qquad
\otimes h(z_1,\dots,z_{M_n})|_{z_{M_{i -1}+1} = \ldots = z_{M_i}\,\forall i}  
\Big)\,,
\end{split}
\end{equation}
where the arrow $\rightarrow$ means that we apply the derivatives $\partial_{z_j}$ 
to $h$ before setting $z_{M_{i -1}+1} = \ldots = z_{M_i}$ for all $i$.
Compare the above formulas \eqref{eq:sesqui2}, \eqref{eq:symmetric-ch} and \eqref{eq:composition-ch},
with \eqref{20160629:eq4}, \eqref{20160629:eq5} and \eqref{circ5} respectively.
As a result, for a given $\mb F[\partial]$-module $V$,
we recover, by Remark \ref{rem:operad-1}, the chiral operad ${\mathscr P}$ of Section \ref{sec:chiral}:
$$
({\mathscr P}^\ch V)(n)=\mc C^\ch(V,\dots,V;V)
\,,\quad n\geq0
\,.
$$
\end{example}
\begin{example}\label{ex:cl}
Again, we consider the class of objects of $\mb F[\partial]$-modules,
and we associate to it yet another structure of pseudo tensor category,
which we call the \emph{classical pseudo tensor category} $\mc C^{\cl}$.
In order to define it, let $\mc G(n)$, for $n\geq0$,
be the set of all oriented graphs on the set of vertices $\{1,\dots,n\}$ with no tadpoles.
Then we let ${\mc C}^{cl}( L_1,\cdots,L_n;L )$ be the space of all maps 
$$
f: \mathcal{G}(n) \times L_1 \otimes \cdots \otimes L_n \rightarrow 
L[\lambda_1,\cdots,\lambda_n] /\langle \partial+ \lambda_1 + \cdots + \lambda_n \rangle
$$
denoted
$$
(\Gamma,l_1\otimes\dots\otimes l_n)
\mapsto 
f^\Gamma_{\lambda_1,\dots,\lambda_n}(l_1\otimes\dots\otimes l_n)
\,,
$$
which are linear in the $L_i$ factors and 
satisfy the cycle relations \cite[(10.4)]{BDSHK19} 
and the sesquilinearity condition \cite[(10.6)-(10.7)]{BDSHK19}.
The right action of $S_n$ in \eqref{equation1b} is given as follows. 
For $X \in {\mc C}^{cl}(L_1,\cdots,L_n;L)$, 
$\Gamma \in \mathcal{G}(n)$, and $\sigma \in S_n$, we have
\begin{equation} \label{eq:symmetric-cl}
(f^\sigma)^{\Gamma}_{\lambda_1,\ldots,\lambda_n} (l_1\otimes\cdots\otimes l_{n}) 
:= 
f^{\sigma^{-1}\Gamma}_{\lambda_{\sigma^{-1}(1)},\ldots,\lambda_{\sigma^{-1}(n)} } 
(l_{\sigma^{-1}(1)}\otimes\cdots\otimes l_{\sigma^{-1}(n)}),
\end{equation} 
where $\sigma^{-1}\Gamma$ is the graph obtained from $\Gamma$ by relabeling the vertex $i$  to have label $\sigma^{-1}(i)$. 
Composition is defined as follows. For a graph $\Gamma \in \mathcal{G}(M_n)$ 
we consider its cocomposition $\Delta^{m_1\ldots m_n}(\Gamma)$ defined in \cite[9.1]{BDSHK19}. 
In particular we have graphs $\Delta_0(\Gamma) \in \mathcal{G}(n)$ 
and $\Delta_i(\Gamma) \in \mathcal{G}(m_i)$, for $1 \leq i \leq n$. 
Let $\Lambda_i = \sum_{j=M_{i-1}+1}^{M_i} \lambda_j$, $i=1,\cdots,n$. 
Recall from \cite[(9.5)]{BDSHK19} that for a set of variables $x_1,\cdots,x_n$ we let 
$$
X(k) = X(\Gamma,m_1,\cdots,m_n;k) = \sum_{j \in \mathcal{E}(k)} x_j,
$$
the sum being over the set of all $j \in \left\{ 1,\cdots,n \right\}$ externally connected to the vertex $k$. 
Then, for $f \in {\mc C}^{cl}(L_1,\cdots,L_n;L)$ 
and $g^i \in {\mc C}^{cl}(K_{M_{i-1}+1},\cdots,K_{M_i}; L_i)$, $1 \leq i \leq n$, 
we define their composition $f (g^1\otimes \cdots\otimes g^n) \in {\mc C}^{cl}(K_1,\cdots,K_{M_n}; L)$ 
by (we use the same notation as in \eqref{eq:composition*})
\begin{equation}\label{eq:lambda-circ-i-first-cl-b}
\begin{split}
& (f (g^1\otimes\ldots\otimes g^n))^{\Gamma}_{\lambda_1,\ldots,\lambda_{M_n}} 
(k_1 \otimes \cdots \otimes k_{M_n}) \\
& =
f^{\Delta_0(\Gamma)}_{\Lambda_1,\ldots,\Lambda_n} 
\Big(   
\Big|_{x_1 = \Lambda_1 + \partial} 
(g^1)^{\Delta_1(\Gamma)}_{\lambda_1 + X(1),\ldots,\lambda_{m_1} + X(m_1)} 
(k_1 \otimes \ldots \otimes k_{m_1})
\otimes\dots \\
&\quad \dots\otimes
\Big|_{x_n = \Lambda_n + \partial} 
(g^n)^{\Delta_n(\Gamma)}_{\lambda_{M_{n-1}+1} + X(M_{n-1}+1),\ldots,\lambda_{M_n} + X(M_n)} (k_{M_{n-1}+1}\otimes\ldots\otimes k_{M_n}) 
\Big)
\,. 
\end{split}
\end{equation}
In this case, for a given $\mb F[\partial]$-module $V$,
we recover, by Remark \ref{rem:operad-1}, the classical operad ${\mathscr P}^\cl$ 
of \cite{BDSHK19}:
$$
({\mathscr P}^\cl V)(n)=\mc C^\cl(V,\dots,V;V)
\,,\quad n\geq0
\,.
$$
\end{example}
\begin{example}\label{ex:hopf}
Let $H$ be a cocommutative bialgebra.
Then the category $\mc C^H$ of $H$-modules has a pseudo tensor structure 
$$
\mc C^H(L_1,\dots,L_n;L)
=
\Hom_{H^{\otimes n}}(L_1\otimes\dots\otimes L_n,H^{\otimes n}\otimes_H L)
\,,
$$
with the natural action of $S_n$ and composition maps generalizing \eqref{eq:composition*} 
\cite{BD04,BDAK01}.
\end{example}

Notice that (see \cite[Lem.1.2.6]{BD04}), 
having included the case $n=0$ in Definition \ref{defn:pseudo-tensor}(2),
gives us a functor 
$$
h:\,\mc C\to\Vect_{\mb F}\,,\quad \text{ mapping } L\mapsto h(L)=\mc C(\left\{
 \right\};L)\,,
$$
called the \emph{augmentation functor},
and, for $i=1,\dots,n$, maps
\begin{equation} \label{eq:augmentation} 
\begin{split}
& h_i:\,
h(L_i)\otimes {\mc C}(L_1,\cdots,L_n;L) 
\rightarrow 
{\mc C}(L_1,\stackrel{i}{\check{\dots}},L_n;L)\,, \\
& \xi\otimes\varphi
\mapsto
h_i(\xi\otimes\varphi)=\varphi(\id_{L_1}\otimes\dots\otimes
\stackrel{i}{\widehat{\xi}}\otimes\dots\otimes\id_{L_n})
\,,
\end{split}
\end{equation}
called the \emph{augmentation maps}.
As usual, $\stackrel{i}{\check{\dots}}$ means that the $i$-th term is skipped
and $\stackrel{i}{\widehat{\xi}}$ means insertion in position $i$.

\begin{example}
For the conformal, chiral and classical pseudo tensor categories 
${\mc C}^*$, ${\mc C}^{ch}$ and ${\mc C}^{cl}$ of Examples \ref{ex:cHom}, \ref{ex:ch} and \ref{ex:cl},
the augmentation functor is given by $h(L) = L/\partial L$. 
For example, for the conformal pseudo tensor category $\mc C^*$,
given $\phi \in {\mc C}^{*}(L_1,..,L_n;L)$ and $\bar{l_i} \in L_i/\partial L_i=h(L_i)$, 
we let $l_i \in L_i$ be any pre-image, and the corresponding augmentation map \eqref{eq:augmentation} 
is defined by
$$
\big(h_i(\bar l_i\otimes\phi)\big)_{\lambda_1,\stackrel{i}{\check{\dots}},\lambda_n}
(l_1,\stackrel{i}{\check{\dots}},l_n)
:= 
\phi_{\lambda_1,\dots,\stackrel{i}{\widehat{0}},\dots,\lambda_n}(l_1,\cdots,l_n)\,,
$$
where $\stackrel{i}{\widehat{0}}$ means that the variable $\lambda_i$ is set equal to $0$.
$h_i(\bar l_i\otimes\phi)$ is indeed well defined 
(i.e. it is independent of the choice of $l_i\in L_i$) 
since by sesquilinearity we have 
$\phi_{\lambda_1,\dots,\stackrel{i}{\widehat{0}},\dots,\lambda_n}(l_1,\cdots,\partial l_i,\cdots,l_n) = 0$.
The ${\mc C}^{ch}$ and ${\mc C}^{cl}$ cases are treated similarly.
\end{example}

\begin{definition}\label{def:functor}
Let ${\mc C}, {\mc P}$ be two pseudo tensor categories. A \emph{pseudo tensor functor} 
$F: {\mc P} \rightarrow {\mc C}$ is an assignment of an object $F(M)$ of ${\mc C}$ 
for any object $M$ of ${\mc P}$, together with linear maps
$$
F:\, {\mc P} ( M_1,\cdots,M_n;L ) 
\rightarrow 
{\mc C} ( F(M_1),\cdots,F(M_n);F(L) )
\,,
$$
such that $F(\id_{M}) = \id_{F(M)}$ and $F$ is compatible with the composition maps. 
In the particular case when ${\mc P}$ is an operad, 
namely a pseudo tensor category with a single object $*$,
$F(*)$ is called a ${\mc P}$-\emph{algebra in} ${\mc C}$. 
\end{definition}

\begin{remark}\label{rem:algebra}
Let ${\mc C}$ be a pseudo tensor category. 
Then the augmentation map $h:\,\mc C\to\Vect_{\mb F}$ is a pseudo tensor functor 
(where $\mathrm{Vect}_{\mb F}$ has its usual tensor structure $\otimes_{\mb F}$). 
If ${\mathscr P}$ is an operad, and $F: {\mathscr P} \rightarrow {\mc C}$ 
is a ${\mathscr P}$-algebra in ${\mc C}$, 
then by composition we have that $h\circ F$ is a ${\mathscr P}$-algebra in $\mathrm{Vect}_{\mb F}$.
\end{remark}

As a special case of Definition \ref{def:functor}, consider the operad $\mc Lie$,
defined by letting $\mc Lie(1)=\mb F1$,
$\mc Lie(2)=\mb F\beta$ such that $\beta^{(1,2)}=-\beta$,
and letting all other spaces $\mc Lie(n)$ generated by composing this element $\beta$,
subject to the Jacobi identity 
$$
\beta(1\otimes\beta)=\beta(\beta\otimes 1)+(\beta(1\otimes\beta))^{(1,2)}\,.
$$
Then, by definition, a $\mc Lie$-\emph{algebra in} $\mc C$
is a pseudo tensor functor $\mc Lie\to\mc C$.
Concretely, it is one object $L\in\mc C$ and en element 
$[\cdot\,,\,\cdot]\in\mc C(L,L;L)$ satisfying skew-symmetry 
$$
[\cdot\,,\,\cdot]^{(1,2)}=-[\cdot\,,\,\cdot]\,,
$$
and Jacobi identity
$$
[\cdot\,,\,\cdot](\id_L\otimes[\cdot\,,\,\cdot]) = 
[\cdot\,,\,\cdot]([\cdot\,,\,\cdot]\otimes \id_L)
+ 
([\cdot\,,\,\cdot](\id_L\otimes[\cdot\,,\,\cdot]))^{(1,2)}\,.
$$
\begin{example}\label{ex:Lie1}
In particular, a $\mc Lie$-algebra in $Vect_{\mb F}$, viewed as a pseudo tensor category,
is a usual Lie algebra over $\mb F$.
\end{example}
\begin{example}\label{ex:Lie2}
A $\mc Lie$-algebra in the conformal pseudo tensor category $\mc C^*$ of Example \ref{ex:cHom}
is the same as a Lie conformal algebra. 
A $\mc Lie$-algebra in the pseudo tensor category $\mc C^H$ of Example \ref{ex:hopf}
is called a Lie pseudoalgebra in \cite{BDAK01}.
Simple Lie pseudoalgebras of finite rank over $H$ where classified there.
\end{example}
\begin{example}\label{ex:Lie3}
A $\mc Lie$-algebra in the chiral pseudo tensor category $\mc C^\ch$ of Example \ref{ex:ch}
is the same as a non unital vertex algebra. 
\end{example}
\begin{example}\label{ex:Lie4}
A $\mc Lie$-algebra in the classical pseudo tensor category $\mc C^\cl$ of Example \ref{ex:cl}
is the same as a non unital Poisson vertex algebra. 
\end{example}

\begin{remark}
By Remark \ref{rem:algebra},
given a $\mc Lie$-algebra in an augmented pseudo tensor category $\mc C$,
if we compose the functor $\mc Lie\to\mc C$ with the augmentation functor $h:\,\mc C\to Vect_{\mb F}$,
we obtain a (usual) Lie algebra over $\mb F$.
\end{remark}
\begin{example}\label{ex:lie-alg}
In the particular cases of Examples \ref{ex:Lie2}, \ref{ex:Lie3} and \ref{ex:Lie4},
we recover the facts that, if $L$ is a Lie conformal algebra, a (non unital) vertex algebra 
or a (non unital) Poisson vertex algebra,
then $h(L)=L/\partial L$ is canonically a Lie algebra over $\mb F$.
\end{example}

In what follows we will need the following particular case of the general construction in \cite[1.2.8]{BD04}. 
Let ${\mc C}$ be a pseudo tensor category, $L$ be a $\mc Lie$-algebra in ${\mc C}$ 
and let $h(L)$ be the corresponding Lie algebra over $\mb F$ defined in Remark \ref{rem:algebra}. 
Then $L$ is naturally a $h(L)$-\emph{module}, namely, there exists a canonical map 
$$
\ad : h(L) \rightarrow {\mc C}(L, L), \qquad l \mapsto ad_l\,,
$$
satisfying the usual identity:
$$
\ad_{[l_1,l_2]} = \ad_{l_1} \circ \ad_{l_2} - \ad_{l_2} \circ \ad_{l_1}\,.
$$
The map $\ad_l$ is defined as follows. If we let $[\cdot\,,\,\cdot] \in {\mc C}(L,L;L)$ be the image 
of the Lie bracket in $\beta\in\mc Lie(2)$, then 
\begin{equation}\label{no:lie-action}
\ad_l=h_1(l\otimes[\cdot\,,\,\cdot])\,,
\end{equation}
where $h_1$ is the augmentation map in \eqref{eq:augmentation}.
Moreover, this action is by derivations of the $\mc Lie$-algebra $L$, that is
$$
\ad_l \circ [\cdot\,,\,\cdot] = [\cdot\,,\,\cdot] \circ (\ad_l\otimes \id_L) + [\cdot\,,\,\cdot] \circ (\id_L\otimes\ad_l)
\,.
$$
\begin{example}\label{ex:lie-mod}
In the particular cases of Examples \ref{ex:Lie2}, \ref{ex:Lie3} and \ref{ex:Lie4},
we recover the facts that, if $L$ is a Lie conformal algebra, a (non unital) vertex algebra 
or a (non unital) Poisson vertex algebra,
then it is canonically a module over the Lie algebra $L/\partial L=h(L)$.
This action is by derivations of the $\lambda$-bracket. 
\end{example}

The vector space $\mb F = \mb F |0\rangle$ is naturally a vertex algebra (resp. a Poisson vertex algebra)
with zero $\lambda$-bracket. 
It is immediate to check that a  \emph{unital vertex algebra} $V$ is the same as 
a $\mc Lie$-algebra in the pseudo tensor category $\mathcal{C}^{ch}$, 
together with a morphism of non-unital vertex algebras $\mb F |0\rangle \rightarrow V$ 
such that the restriction of the binary operation $X_{\lambda_1, \lambda_2}$,
image of $\beta\in\mc Lie{(2)}$, to $\vac \otimes V$ satisfies:
\begin{equation}\label{eq:unit-vac}
X_{\lambda_1, \lambda_2}^{z_1,z_2} \Big(\vac, v; \frac{1}{z_2-z_1} \Big) = v
\,.
\end{equation}

\subsection{Operads in pseudo tensor categories}\label{sec:operads}

\begin{definition}\label{def:operad-ptc1}
Let ${\mc C}$ be a pseudo tensor category. 
An (non-necessarily unital) operad ${\mathscr P}$ in ${\mc C}$ consists of the following data:
a collection of objects ${\mathscr P}(n)$ of ${\mc C}$, $n\geq0$;
a right action of the symmetric group $S_n$ on ${\mathscr P}(n)$,
i.e.  a group homomorphism 
\begin{equation}\label{eq:gp-hom}
S_n \mapsto \Aut({\mathscr P}(n))^{op} \subset {\mc C}({\mathscr P}(n), {\mathscr P}(n))^{op}
, \qquad 
\sigma \mapsto \rho_\sigma \in {\mc C}\big( {\mathscr P}(n), {\mathscr P}(n) \big)
\,,
\end{equation}
where ${\mc C}\big( {\mathscr P}(n), {\mathscr P}(n) \big)$ has the associative product
given by the composition \eqref{equation1},
$\Aut({\mathscr P}(n))\subset {\mc C}\big( {\mathscr P}(n), {\mathscr P}(n) \big)$
is the group of invertible elements with respect to this product,
and $op$ denotes the group with respect to the opposite product;
a collection of \emph{composition maps}
\begin{equation} \label{eq:composition1} 
\circ_{(m_1,\ldots,m_n)} \in {\mc C} \big({\mathscr P}(n),{\mathscr P}(m_1),\cdots,{\mathscr P}(m_n); {\mathscr P}(M_n) \big)
\,,
\end{equation}
where we denote, as usual, $M_n=\sum_{i=1}^nm_i$.
These data are subject to the following conditions:
\begin{enumerate}[(i)]
\item 
The composition maps are compatible with the actions of the symmetric groups as follows. 
Given $\sigma \in S_n$ and $\tau_i \in S_{m_i}$, $i=1,\dots,n$, 
we consider the composition of permutations $\sigma(\tau_1,\dots,\tau_n) \in S_{m_1+\ldots +m_n}$ 
defined by \eqref{eq:operad19}. 
Then
\begin{equation}\label{20240606:eq1}
	\circ_{(m_1,\ldots,m_n)}
\big( \rho_{\sigma} \otimes \rho_{\tau_1} \otimes\cdots \otimes \rho_{\tau_n} \big) 
=
\rho_{\sigma(\tau_1, \ldots, \tau_n)} \Big( \big( \circ_{({m_{\sigma^{-1}(1)}},\ldots,m_{\sigma^{-1}(n)})} \big)^{\tilde\sigma} \Big)
\,,
\end{equation}
where $\tilde\sigma\in\text{Perm}\{0,1,\dots,n\}=S_{n+1}$
is the permutation which fixes $0$ and permutes $1,\dots,n$ via $\sigma$.
\item 
Recall that $\id_{{\mathscr P}(n)}\in\mc C({\mathscr P}(n);{\mathscr P}(n))$ denotes the unit vector
of Definition \ref{defn:pseudo-tensor}(4), associated to ${\mathscr P}(n)$ viewed as an object in $\mc C$.
The composition maps are associative in the following sense. 
Given $m_1,\cdots,m_n\in\mb Z_{\geq0}$,
and $l_{1},\dots,l_{M_n}\in\mb Z_{\geq0}$,
\begin{equation}\label{20240606:eq2}
\begin{split} 
& \Big[
\circ_{\big(\sum_{j=1}^{M_1} l_j , \ldots, \sum_{j=M_{n-1}+1}^{M_n} l_j\big)}  
\Big( 
\id_{{\mathscr P}(n)} 
\otimes \circ_{(l_1,\ldots,l_{M_1})} 
\otimes \dots 
\otimes \circ_{(l_{M_{n-1}+1},\ldots,l_{M_n})} 
\Big)
\Big]^\sigma  \\ 
& =
\circ_{(l_1, \ldots,l_{M_n})} 
\Big( 
\circ_{(m_1,\ldots,m_n)} 
\otimes \id_{{\mathscr P}(l_1)} 
\otimes \dots
\otimes \id_{{\mathscr P}(l_{M_n})} 
\Big)
\,,
\end{split}
\end{equation}
where $\sigma \in S_{M_n+n+1}$ is the unique shuffle of type $(n+1,M_n)$, given by
$\sigma(1)=1$, $\sigma(i) = i+M_{i-2}$, $2 \leq i \leq n+1$.
\end{enumerate}
\end{definition}
The above identities \eqref{20240606:eq1} and \eqref{20240606:eq2} need some explanation.
All compositions are in the pseudo tensor category $\mc C$, 
defined according to Definition \ref{defn:pseudo-tensor}(5).
In equation \eqref{20240606:eq1} the left-hand side  is obtained by applying the composition map 
\begin{align*}
& \mc C\Big(
{\mathscr P}(n),{\mathscr P}(m_1),\dots,{\mathscr P}(m_n);{\mathscr P}(M_n)
\Big) \\
&\qquad \otimes
\mc C\big(
{\mathscr P}(n);{\mathscr P}(n)
\big)
\otimes
\mc C\big(
{\mathscr P}(m_1);{\mathscr P}(m_1)
\big)
\otimes\dots\otimes
\mc C\big(
{\mathscr P}(m_n);{\mathscr P}(m_n)
\big)
\,,
\end{align*}
hence it lies in
\begin{equation}\label{eq:B}
{\mc C}\Big(
{\mathscr P}(n), {\mathscr P}(m_1),\cdots, {\mathscr P}(m_n) ; {\mathscr P}(M_n) 
\Big)
\,.
\end{equation}
In the right-hand side, $\circ_{({m_{\sigma^{-1}(1)}},\ldots,m_{\sigma^{-1}(n)})}$
lies in 
$$
{\mc C}\Big(
{\mathscr P}(n), {\mathscr P}(m_{\sigma^{-1}(1)}),\cdots, {\mathscr P}(m_{\sigma^{-1}(n)}) ; 
{\mathscr P}(M_n) \Big)
\,.
$$
Then, we need to right-act on it with $\tilde\sigma$, according to \eqref{equation1b},
ending up in \eqref{eq:B}.
If we then right-compose it with 
$\rho_{\sigma(\tau_1, \ldots, \tau_n)}\in\mc C\big({\mathscr P}(M_n);{\mathscr P}(M_n)\big)$,
the result is still in \eqref{eq:B}.
As for equation \eqref{20240606:eq2}, 
the right-hand side is obtained by applying the composition map to
\begin{align*}
& \mc C\Big(
{\mathscr P}(M_n),{\mathscr P}(l_1),\dots,{\mathscr P}(l_{M_n});{\mathscr P}\big(\sum_{j=1}^{M_n}l_j\big)
\Big) \\
&\qquad \otimes
\mc C\big(
{\mathscr P}(n),{\mathscr P}(m_1),\dots,{\mathscr P}(m_n);{\mathscr P}(M_n)
\big) \\
&\qquad \otimes
\mc C\big(
{\mathscr P}(l_1);{\mathscr P}(l_1)
\big)
\otimes\dots\otimes
\mc C\big(
{\mathscr P}(l_{M_n});{\mathscr P}(l_{M_n})
\big)
\,,
\end{align*}
hence it lies in
\begin{equation}\label{eq:A}
{\mc C}\Big(
{\mathscr P}(n), {\mathscr P}(m_1),\cdots, {\mathscr P}(m_n),
{\mathscr P}(l_1), \cdots,{\mathscr P}(l_{M_n});  {\mathscr P}\big(\sum_{j=1}^{M_n} l_j\big) 
\Big)
\,.
\end{equation}
Before applying $\sigma$, the left-hand side is obtained by applying the composition map
\begin{align*}
& \mc C\Big(
{\mathscr P}(n),{\mathscr P}\big(\sum_{j=1}^{M_1}l_j\big),\dots,
{\mathscr P}\big(\sum_{j=M_{n-1}+1}^{M_n}l_j\big);{\mathscr P}\big(\sum_{j=1}^{M_n}l_j\big)
\Big) \\
&\qquad \otimes
\mc C\Big(
{\mathscr P}(m_1),{\mathscr P}(l_1),\dots,{\mathscr P}(l_{M_1});{\mathscr P}\big(\sum_{j=1}^{M_1}l_j\big)
\Big) 
\otimes\dots \\
&\qquad \dots \otimes
\mc C\Big(
{\mathscr P}(m_n),{\mathscr P}(l_{M_{n-1}+1}),\dots,{\mathscr P}(l_{M_n});{\mathscr P}\big(\sum_{j=M_{n-1}+1}^{M_n}l_j\big)
\Big) \\
&\qquad \longrightarrow
{\mc C}\Big(
{\mathscr P}(n), {\mathscr P}(m_1),{\mathscr P}(l_1), \cdots,{\mathscr P}(l_{M_1})
,\dots \\
& \qquad\qquad
\dots,
{\mathscr P}(m_n),{\mathscr P}(l_{M_{n-1}+1}), \cdots,{\mathscr P}(l_{M_n});  
{\mathscr P}\big(\sum_{j=1}^{M_n} l_j\big) 
\Big)
\,.
\end{align*}
Hence, if we then apply the permutation $\sigma$, we obtain an element of \eqref{eq:A}.

\begin{example}\label{ex:symmetric}
When ${\mc C}$ is a symmetric monoidal category, then an operad in ${\mc C}$ is given by a collection of objects ${\mathscr P}(n)$ together with linear maps
\[ \circ_{(m_1,\ldots,m_n)} : {\mathscr P}(n) \otimes {\mathscr P}(m_1) \otimes \cdots \otimes {\mathscr P}(m_n) \rightarrow {\mathscr P}(m_1 +\cdots + m_n), \]
satisfying the usual associativity and symmetry conditions.
In particular, an operad in the category $\Vect_{\mb F}$ of vector spaces
is the same as a usual operad.
\end{example}

\begin{remark}\label{rem:not-in-c}
If ${\mc C}$ is a pseudo tensor category and $M \in {\mc C}$ is an object, 
then the collection of vector spaces ${\mathscr P} = {\mc C}(M,\cdots,M;M)$ 
is an operad in the category of vector spaces, cf. Remark \ref{rem:operad-1}, 
but \textbf{it is not} an operad in ${\mc C}$.
\end{remark}

\begin{example}\label{ex:asdqwe}
Suppose that the symmetric monoidal category ${\mc C}$ has an \emph{inner Hom} functor $\mc Hom$.
This is the case, for example, for the category $\mc C=Vect_{\mb F}$ of vector spaces.
In this case given any object $M$ in ${\mc C}$ we obtain an operad 
$\mc End_M$ in ${\mc C}$ given by 
$$
\mc End_M(n) := \mc Hom(M^{\otimes n}, M)
\,.
$$
\end{example}

\begin{definition}\label{def:operad-ptc2}
An operad in a pseudo tensor category $\mc C$ is \emph{unital} if
${\mathscr P}$ has an identity $1 \in h({\mathscr P}(1))$ such that 
$$
h_1(1 \otimes \circ_{(n)}
=
h_{n+1}(1^{\otimes n}\otimes \circ_{(1,\dots,1)})
=
\id_{{\mathscr P}(n)}
\in
\mc C({\mathscr P}(n),{\mathscr P}(n))
\,,
$$
where $h_i$ are the augmentation maps \eqref{eq:augmentation}.
\end{definition}

Given an unital operad ${\mathscr P}$ in an augmented pseudo tensor category ${\mc C}$, 
for any two natural numbers $n,m$ and for $1 \leq i \leq n$ we define the $\circ_i$ compositions as
\begin{equation} \label{eq:circ_i}
\circ_i = \circ_i^{n,m} 
:=
h_i(1^{\otimes(n-1)} \otimes \circ_{(1,\ldots,m,\ldots,1)})
\in 
{\mc C}( {\mathscr P}(n), {\mathscr P}(m); {\mathscr P}(n+m-1)) 
\end{equation}
(in the indices of $\circ$ in te RHS there are $n-1$ elements $1$ and $m$ is in the $i$-th position).
They satisfy the following conditions (cf. \eqref{eq:circ-assoc1-conf}-\eqref{eq:circ-assoc3-conf}):
\begin{equation} \label{eq:circ-assoc}
\circ_j^{n+m-1,l} \bigl( \circ_i^{n,m} \otimes \id_{{\mathscr P}(l)} \bigr) = 
\begin{cases}
\circ_{l+i-1}^{n+l-1,m} \bigl( \circ_j^{n,l} \otimes \id_{{\mathscr P}(m)} \bigr) & 1 \leq j < i, \\ 
\circ_i^{n,m+l-1} \bigl( \id_{{\mathscr P}(n)} \otimes \circ_{j-i+1}^{m,l} \bigr) & i \leq j < i+m, \\ 
\circ_i^{(n+l-1,m)} \bigl( \circ_{j-m+1}^{n,l} \otimes\id_{{\mathscr P}(m)} \bigr) & i+m \leq j < n+m.
\end{cases}
\end{equation}
which is an identity in 
${\mc C}\bigl({\mathscr P}(n), {\mathscr P}(m), {\mathscr P}(l); {\mathscr P}(n+m+l-2)\bigr)$.
Note that, given the objects ${\mathscr P}(n)$ for every $n \geq 1$ and the $\circ_i$ products,
we recover the composition morphisms $\circ_{(m_1,\ldots,m_n)}$ as 
\begin{equation} \label{eq:recover-circ} 
\circ_{(m_1,\ldots,m_n)} = \circ_{M_{n-1}+1}^{n+M_{n-1}-1,m_n} \Bigl( \cdots
\circ_{M_1+1}^{n+M_1-1,m_2} \Bigl( \circ^{n,M_1}_1 \otimes \id_{{\mathscr P}(m_2)} \Bigr)
\otimes \cdots \id_{{\mathscr P}(m_n)} \Bigr).  
\end{equation}

\begin{definition}\label{def:pseudo-operad}
A \emph{pseudo operad} in a pseudo tensor category ${\mc C}$ consists of a
collection ${\mathscr P}(n)$ of objects in $\mc C$,
endowed with actions of the symmetric groups $S_n$,
and with $\circ_i$ compositions \eqref{eq:circ_i} 
which are equivariant with respect to the actions of $S_n$,
and satisfy conditions \eqref{eq:circ-assoc}.
\end{definition}

\begin{remark}
It follows from this definition and \eqref{eq:recover-circ} that every unital
operad in a pseudo tensor category ${\mc C}$ gives rise to a pseudo operad in ${\mc C}$.
The converse is not true, as the examples of the conformal and vertex operads
from the previous sections are not unital operads, however the $\circ_i$
operations can be explicitly defined, endowing them with the structure of a
pseudo operad.
\end{remark}

Let ${\mathscr P}$ be a unital operad in an augmented pseudo tensor category ${\mc C}$. 
Let (cf. \eqref{eq:box})
\begin{equation} \label{eq:box-pseudo} 
\Box = \frac{1}{n} \sum_{i=1}^{n} \sum_{\sigma \in S_{m,n-1}}  
( \circ_i^{n,m} )^{\sigma^{-1}} 
\in 
{\mc C}({\mathscr P}(n), {\mathscr P}(m); {\mathscr P}(n+m-1)).
\end{equation}
We extend it linearly to a degree $0$ operation $\Box \in {\mc C} (W, W; W)$ where
\begin{equation} \label{eq:W-definition-inv}
W = \oplus_{k \geq 0} {\mathscr P}(k)^{S_k}
\,,
\end{equation}
which is a $\mathbb{Z}$-graded object of ${\mc C}$ where ${\mathscr P}(k)$ is in degree $k-1$. 
The $S_k$-invariants ${\mathscr P}(k)$ mean skew-symmetric invariants.
\begin{remark}
In all examples above, the objects ${\mathscr P}(n)$ are vector superspaces,
hence the notion of $S_n$-ivariants is clear. In general, however,
given an action 
$\rho: S_n \rightarrow \mathrm{Aut}({\mathscr P}(n))$, 
the skew symmetric action $\rho^\epsilon$ is obtained by defining
$$
\rho^\epsilon(\sigma) = (-1)^\sigma \id_{{\mathscr P}(n)} \circ \rho(\sigma).
$$
We also require that the category ${\mathcal C}$ has kernels in order to define
invariants, and that it has cokernels in order to define coinvariants. 
\end{remark}
Finally consider 
\begin{equation} \label{eq:bracket1}
[,] = \Box - \Box^{\sigma_{12}},
\end{equation} 
where $\sigma_{12} \in S_2$ is the non-trivial element and $\Box^{\sigma_{12}}$ 
is defined by Remark \ref{rem:operad-1}. 

\begin{theorem}\label{thm:main}
Let ${\mc C}$ be pseudo tensor category and let ${\mathscr P} = \left\{ {\mathscr P}(n)
\right\}$ be a pseudo operad in ${\mc C}$. Then the object $W$ is a $\mathbb{Z}_{\geq -1}$-graded Lie algebra in ${\mc C}$.
\end{theorem}
\begin{proof}
It follows by Remark \ref{rem:reimundo} since the $\circ_i$ products satisfy identities
\eqref{eq:circ-assoc}, or equivalently \eqref{eq:circ-assoc1-conf}-\eqref{eq:circ-assoc3-conf}.
\end{proof}
\begin{remark}
	If instead of taking invariants in \eqref{eq:W-definition-inv} we would
	have considered coinvariants, then the sum over shuffles in
	\eqref{eq:box-pseudo} is not necessary. $W$ is also a Lie algebra in
	${\mathcal C}$ in this case. When the characteristic of $\mathbb{F}$ is
	zero, these two definitions agree. 
\end{remark}

\subsection{Conformal operads as operads in the pseudo tensor category $\mc C^*$}

\begin{proposition}
A conformal operad $\widetilde{\mathscr P}$,
as defined in Definition \ref{def:conf-operad},
is the same as an operad in the pseudo tensor category $\mc C^*$ of Example \ref{ex:cHom}.
\end{proposition}
\begin{proof}
Indeed, $\widetilde{\mathscr P}(n)$ is endowed with an even endomorphism,
making it an $\mb F[\partial]$-module, namely an object of $\mc C^*$.
Moreover, by assumption we have a right action of $S_n$ on $\widetilde{\mathscr P}(n)$,
which defines the group homomorphism 
$S_n \mapsto \Aut({\mathscr P}(n))^{op}$ in \eqref{eq:gp-hom}.
The composition map \eqref{eq:conf-comp} defines a map
\begin{align*}
& \circ_{(m_1,\dots,m_n)}:\,
\widetilde{\mathscr P}(n)\otimes\widetilde{\mathscr P}(m_1)\otimes\dots\otimes\widetilde{\mathscr P}(m_n)
\to \\
&\qquad
\widetilde{\mathscr P}(M_n)[\lambda_1,\dots,\lambda_n]
\simeq
\widetilde{\mathscr P}(M_n)[\lambda_0,\dots,\lambda_n]/
\langle\partial+\lambda_0+\dots+\lambda_n\rangle
\,,
\end{align*}
given by
$$
(\circ_{(m_1,\dots,m_n)})_{\lambda_0,\lambda_1,\dots,\lambda_n}
(f\otimes g_1\otimes\dots\otimes g_n)
=
f_{\lambda_1,\dots,\lambda_n}(g_1\otimes\dots\otimes g_n)
\,.
$$
The sesquilinearity conditions \eqref{20170613:eq1-conf}
guarantee that
$$
\circ_{(m_1,\dots,m_n)}
\in
\mc C^*\big(\widetilde{\mathscr P}(n),\widetilde{\mathscr P}(m_1),\dots,\widetilde{\mathscr P}(m_n);
\widetilde{\mathscr P}(M_n)\big)
\,,
$$
as required in \eqref{eq:composition1}.

We want to check that
the compatibility \eqref{20240606:eq1} of the composition maps with the symmetric group action
boils down, for the pseudo tensor category $\mc C^*$,
to the equivariance condition \eqref{eq:operad4-conf}.
Indeed,  by the definition \eqref{eq:composition*} of the composition in the pseudo tensor
category $\mc C^*$, and of the right-action \eqref{eq:symmetric*} of the symmetric groups,
we have
\begin{align*}
& \Big(
\circ_{(m_1,\ldots,m_n)}
\big( \rho_{\sigma} \otimes \rho_{\tau_1} \otimes\cdots \otimes \rho_{\tau_n} \big) 
\Big)_{\lambda_0,\lambda_1,\dots,\lambda_n}
(f\otimes g_1\otimes\dots\otimes g_n) \\
& =
\big(\circ_{(m_1,\ldots,m_n)}\big)_{\lambda_0,\lambda_1,\dots,\lambda_n}
\Big(
(\rho_{\sigma})_{\lambda_0}(f) \otimes (\rho_{\tau_1})_{\lambda_1}(g_1) \otimes\cdots \otimes (\rho_{\tau_n})_{\lambda_n}(g_n)  
\Big) \\
& =
\big(\circ_{(m_1,\ldots,m_n)}\big)_{\lambda_0,\lambda_1,\dots,\lambda_n}
\Big(
f^\sigma \otimes g_1^{\tau_1} \otimes\cdots \otimes g_n^{\tau_n}  
\Big) \\
& =
(f^\sigma)_{\lambda_1,\dots,\lambda_n}
( g_1^{\tau_1} \otimes\cdots \otimes g_n^{\tau_n} )
\,,
\end{align*}
and
\begin{align*}
& \Big(
\rho_{\sigma(\tau_1, \ldots, \tau_n)} \Big( \big( \circ_{({m_{\sigma^{-1}(1)}},\ldots,m_{\sigma^{-1}(n)})} \big)^{\tilde\sigma} \Big)
\Big)_{\lambda_0,\lambda_1,\dots,\lambda_n}
(f\otimes g_1\otimes\dots\otimes g_n) \\
& =
(\rho_{\sigma(\tau_1, \ldots, \tau_n)})_{\lambda_0+\dots+\lambda_n} 
\Bigg(
\Big( \big( \circ_{({m_{\sigma^{-1}(1)}},\ldots,m_{\sigma^{-1}(n)})} \big)^{\tilde\sigma} 
\Big)_{\lambda_0,\lambda_1,\dots,\lambda_n}
\!\!\!\!\!\!\!\!\!\!\!\!
(f\otimes g_1\otimes\dots\otimes g_n)
\Bigg) \\
& =
(\rho_{\sigma(\tau_1, \ldots, \tau_n)})_{\Lambda} 
\Bigg(
\Big( \circ_{({m_{\sigma^{-1}\!\!(1)}},\ldots,m_{\sigma^{-1}\!\!(n)})} 
\Big)_{\lambda_0,\lambda_{\sigma^{-1}(1)},\dots,\lambda_{\sigma^{-1}(n)}}
\!\!\!\!\!\!\!\!\!\!\!\!\!\!\!\!\!\!\!\!\!\!\!\!
(f\otimes g_{\sigma^{-1}(1)}\otimes\dots\otimes g_{\sigma^{-1}(n)})
\Bigg) \\
& =
(\rho_{\sigma(\tau_1, \ldots, \tau_n)})_{\Lambda} 
\Big(
f_{\lambda_{\sigma^{-1}(1)},\dots,\lambda_{\sigma^{-1}(n)}}
(g_{\sigma^{-1}(1)}\otimes\dots\otimes g_{\sigma^{-1}(n)})
\Big) \\
& =
\Big(
f_{\lambda_{\sigma^{-1}(1)},\dots,\lambda_{\sigma^{-1}(n)}}
(g_{\sigma^{-1}(1)}\otimes\dots\otimes g_{\sigma^{-1}(n)})
\Big)^{\sigma(\tau_1, \ldots, \tau_n)}
\,.
\end{align*}

Finally, we prove that 
the associativity \eqref{20240606:eq2} of the composition maps boils down, 
for the pseudo tensor category $\mc C^*$, 
to the associativity \eqref{eq:operad2-conf} of a conformal operad.
Indeed, in this case, by the definition \eqref{eq:composition*} of the composition in the pseudo tensor
category $\mc C^*$, we have
\begin{align*}
& \Big(\circ_{(l_1, \ldots,l_{M_n})} 
\Big( 
\circ_{(m_1,\ldots,m_n)} 
\otimes \id_{{\mathscr P}(l_1)} 
\otimes \dots
\otimes \id_{{\mathscr P}(l_{M_n})} 
\Big)\Big)_{\lambda_0,\lambda_1,\dots,\lambda_n,\mu_1,\dots,\mu_{M_n}} 
\!\!\!\!\!\!\!\!\!
(f\otimes g_1\otimes\dots \\
&\quad
\dots\otimes 
g_n\otimes h_1\otimes\dots\otimes h_{M_n}) 
= 
\big(\circ_{(l_1, \ldots,l_{M_n})} \big)_{\lambda_0+\dots+\lambda_n,\mu_1,\dots,\mu_{M_n}}
\Big( \\
&\quad \big(\circ_{(m_1,\ldots,m_n)} \big)_{\lambda_0,\lambda_1,\dots,\lambda_n}
(f\otimes g_1\otimes\dots\otimes g_n)
\otimes \big( \id_{{\mathscr P}(l_1)} \big)_{\mu_1}(h_1)
\otimes \dots \\
&\quad
\dots \otimes 
\big( \id_{{\mathscr P}(l_{M_n})} \big)_{\mu_{M_n}}(h_{M_n})
\Big) 
= 
\big(\circ_{(l_1, \ldots,l_{M_n})} \big)_{\lambda_0+\dots+\lambda_n,\mu_1,\dots,\mu_{M_n}}
\Big( \\ 
&\quad 
f_{\lambda_1,\dots,\lambda_n}(g_1\otimes\dots\otimes g_n)
\otimes h_1
\otimes \dots \otimes 
h_{M_n}
\Big) \\
&= 
\big(
f_{\lambda_1,\dots,\lambda_n}(g_1\otimes\dots\otimes g_n)
\big)_{\mu_1,\dots,\mu_{M_n}}
(h_1 \otimes \dots \otimes h_{M_n})
\,,
\end{align*}
and
\begin{align*} 
& \Big[
\circ_{\big(\sum_{j=1}^{M_1} l_j , \ldots, \sum_{j=M_{n-1}+1}^{M_n} l_j\big)}  
\Big( 
\id_{{\mathscr P}(n)} 
\otimes \circ_{(l_1,\ldots,l_{M_1})} 
\otimes \dots 
\otimes \circ_{(l_{M_{n-1}+1},\ldots,l_{M_n})} 
\Big) \\
&\qquad \Big]^\sigma_{\lambda_0,\lambda_1,\dots,\lambda_n,\mu_1,\dots,\mu_{M_n}}
(f\otimes g_1\otimes\dots\otimes g_n\otimes h_1\otimes\dots\otimes h_{M_n})  \\ 
& =
\Big(
\circ_{\big(\sum_{j=1}^{M_1} l_j , \ldots, \sum_{j=M_{n-1}+1}^{M_n} l_j\big)}  
\Big( 
\id_{{\mathscr P}(n)} 
\otimes \circ_{(l_1,\ldots,l_{M_1})} 
\otimes \dots 
\otimes \circ_{(l_{M_{n-1}+1},\ldots,l_{M_n})} 
\Big) \\
&\qquad
\Big)_{\lambda_0,\lambda_1,\mu_1,\dots,\mu_{M_1},\dots,\lambda_n,\mu_{M_{n-1}+1},\dots,\mu_{M_n}}
(f\otimes g_1\otimes h_1\otimes\dots\otimes h_{M_1}\otimes\dots \\
&\qquad\qquad \vphantom{\Bigg(}
\dots\otimes g_n\otimes h_{M_{n-1}+1}\otimes\dots\otimes h_{M_n}) \\
& =
\big(
\circ_{\big(\sum_{j=1}^{M_1} l_j , \ldots, \sum_{j=M_{n-1}+1}^{M_n} l_j\big)}  
\big)_{\lambda_0,\lambda_1+\mu_1+\dots+\mu_{M_1},\dots,\lambda_n+\mu_{M_{n-1}+1}+\dots+\mu_{M_n}}
\Big( \\
&\qquad
\big( \id_{{\mathscr P}(n)} \big)_{\lambda_0} (f)
\otimes 
\big( \circ_{(l_1,\ldots,l_{M_1})} \big)_{\lambda_1,\mu_1,\dots,\mu_{M_1}} 
(g_1\otimes h_1\otimes\dots\otimes h_{M_1})
\otimes \dots \\
&\qquad
\dots \otimes
\big( \circ_{(l_{M_{n-1}+1},\ldots,l_{M_n})} \big)_{\lambda_n,\mu_{M_{n-1}+1},\dots,\mu_{M_n}}
(g_n\otimes h_{M_{n-1}+1}\otimes\dots\otimes h_{M_n})
\Big) \\
& =
\big(
\circ_{\big(\sum_{j=1}^{M_1} l_j , \ldots, \sum_{j=M_{n-1}+1}^{M_n} l_j\big)}  
\big)_{\lambda_0,\lambda_1+\Gamma_1,\dots,\lambda_n+\Gamma_n}
\Big( f \otimes (g_1)_{\mu_1,\dots,\mu_{M_1}}(h_1\otimes\dots \\
&\qquad
\dots \otimes
h_{M_1})
\otimes \dots \otimes
(g_n)_{\mu_{M_{n-1}+1},\dots,\mu_{M_n}}(h_{M_{n-1}+1}\otimes\dots\otimes h_{M_n})
\Big) \\
& =
f_{\lambda_1+\Gamma_1,\dots,\lambda_n+\Gamma_n}
\big(
(g_1)_{\mu_1,\dots,\mu_{M_1}}(h_1\otimes\dots\otimes h_{M_1})
\otimes \dots \\
&\qquad
\dots \otimes
(g_n)_{\mu_{M_{n-1}+1},\dots,\mu_{M_n}}(h_{M_{n-1}+1}\otimes\dots\otimes h_{M_n})
\big)
\,.
\end{align*}
\end{proof}

As a result, Theorem \ref{20170603:thm2-conf} constructing the Lie conformal superalgebra
$\widetilde{W}(\widetilde{\mathscr P})$ associated to the conformal operad $\widetilde{\mathscr P}$
can be viewed as the special case of the pseudo tensor category $\mc C^*$,
of Theorem \ref{thm:main},
constructing the $\mc Lie$-algebra in $\mc C$ associated to an operad $\mathscr P$ 
in the pseudo tensor category $\mc C$.

\begin{remark}
Let $H$ be a cocommutative bialgebra.
Recalling Remark \ref{rem:operad-1} and Example \ref{ex:hopf},
given an $H$-module $V$, we obtain the corresponding operad $\mc Chom^HV$,
defined by letting 
$$
(\mc Chom^HV)(n)
=
\Hom_{H^{\otimes n}}(V^{\otimes n},H^{\otimes n}\otimes_{H}V)
\,,
$$
with the symmetric group action and the composition maps as in Example \ref{ex:hopf}.
We can also define the operad $\widetilde{\mc Chom}^HV$
in the category $\mc C^H$ as the collection of vector superspaces
$$
(\widetilde{\mc Chom}^HV)(n)
=
\Hom_{H^{\otimes n}}(V^{\otimes n},H^{\otimes n}\otimes V)
\,,
$$
with appropriate symmetric group actions and composition maps.
In the special case $H=\mb F[\partial]$,
these coincide with the operad $\mc Chom\,V$
and the operad $\widetilde{\mc Chom}\,V$ in $\mc C^*$, respectively.
\end{remark}

\section*{Acknowledgments}
A.\ De Sole is a member of the GNSAGA INdAM group, he has been  supported 
by the national PRIN grant 2022S8SSW2
and he acknowledges the financial support of INFN, IS CSN4 MMNLP.
R. Heluani was funded by grants from FAPERJ-CNE and CNPq.




\end{document}